\documentclass[english]{article}
\usepackage[T1]{fontenc}
\usepackage[latin9]{inputenc}
\usepackage{array}
\usepackage{float}
\usepackage{dsfont}
\usepackage{amsmath}
\usepackage{amsthm}
\usepackage{amssymb}
\usepackage{wasysym}

\makeatletter

\newcommand{\lyxmathsym}[1]{\ifmmode\begingroup\def\b@ld{bold}
  \text{\ifx\math@version\b@ld\bfseries\fi#1}\endgroup\else#1\fi}

\providecommand{\tabularnewline}{\\}
\floatstyle{ruled}
\newfloat{algorithm}{tbp}{loa}
\providecommand{\algorithmname}{Algorithm}
\floatname{algorithm}{\protect\algorithmname}

\theoremstyle{plain}
\newtheorem{thm}{\protect\theoremname}[section]
\theoremstyle{plain}
\newtheorem{assumption}[thm]{\protect\assumptionname}
\theoremstyle{remark}
\newtheorem{rem}[thm]{\protect\remarkname}
\theoremstyle{definition}
\newtheorem{example}[thm]{\protect\examplename}
\theoremstyle{plain}
\newtheorem{cor}[thm]{\protect\corollaryname}
\theoremstyle{plain}
\newtheorem{lem}[thm]{\protect\lemmaname}

\usepackage{savefnmark}
\usepackage{microtype}
\usepackage{graphicx}
\usepackage{subfigure}
\usepackage{booktabs} % for professional tables

\PassOptionsToPackage{hyphens}{url}
\usepackage{hyperref}

\usepackage[accepted]{icml2025}

\usepackage[capitalize,noabbrev]{cleveref}

\usepackage[textsize=tiny]{todonotes}

\makeatother

\usepackage{babel}
\providecommand{\assumptionname}{Assumption}
\providecommand{\corollaryname}{Corollary}
\providecommand{\examplename}{Example}
\providecommand{\lemmaname}{Lemma}
\providecommand{\remarkname}{Remark}
\providecommand{\theoremname}{Theorem}

\begin{document}
\global\long\def\E{\mathbb{\mathbb{E}}}%
\global\long\def\R{\mathbb{\mathbb{R}}}%
\global\long\def\P{\mathbb{\mathbb{P}}}%
\global\long\def\N{\mathbb{\mathbb{N}}}%
\global\long\def\F{\mathcal{F}}%
\global\long\def\O{O}%
\global\long\def\argmin{\mathbb{\mathrm{argmin}}}%
\global\long\def\dom{\mathrm{dom}}%
\global\long\def\bx{\boldsymbol{x}}%
\global\long\def\by{\boldsymbol{y}}%
\global\long\def\bz{\boldsymbol{z}}%
\global\long\def\bzero{\boldsymbol{0}}%
\global\long\def\bg{\boldsymbol{g}}%
\global\long\def\be{\boldsymbol{e}}%
\global\long\def\RR{\textsf{RR}}%
\global\long\def\SS{\textsf{SS}}%
\global\long\def\IG{\textsf{IG}}%
\global\long\def\avg{\textsf{avg}}%
\global\long\def\suf{\textsf{suffix}}%
\global\long\def\bR{\overline{\R}}%
\global\long\def\defeq{\triangleq}%
\global\long\def\diseq{\overset{\mathcal{D}}{=}}%
\global\long\def\1{\mathds{1}}%
\global\long\def\I{\mathsf{i}}%
\global\long\def\re{\mathsf{r}}%
\global\long\def\qu{\mathsf{q}}%
\global\long\def\inv{\star}%
\global\long\def\domx{\mathcal{C}}%
\global\long\def\charf{I}%

\twocolumn[
\icmltitle{Improved Last-Iterate Convergence of Shuffling Gradient Methods for Nonsmooth Convex Optimization}
% It is OKAY to include author information, even for blind 
% submissions: the style file will automatically remove it for you 
% unless you've provided the [accepted] option to the icml2025 
% package.

% List of affiliations: The first argument should be a (short) 
% identifier you will use later to specify author affiliations 
% Academic affiliations should list Department, University, City, Region, Country 
% Industry affiliations should list Company, City, Region, Country

% You can specify symbols, otherwise they are numbered in order. 
% Ideally, you should not use this facility. Affiliations will be numbered 
% in order of appearance and this is the preferred way. 

\icmlsetsymbol{equal}{*}

\begin{icmlauthorlist}
\icmlauthor{Zijian Liu}{NYU}
\icmlauthor{Zhengyuan Zhou}{NYU,Arena} 
%\icmlauthor{Firstname1 Lastname1}{equal,yyy} 
%\icmlauthor{Firstname2 Lastname2}{equal,yyy,comp} 
%\icmlauthor{Firstname3 Lastname3}{comp} 
%\icmlauthor{Firstname4 Lastname4}{sch} 
%\icmlauthor{Firstname5 Lastname5}{yyy} 
%\icmlauthor{Firstname6 Lastname6}{sch,yyy,comp} 
%\icmlauthor{Firstname7 Lastname7}{comp} 
%\icmlauthor{}{sch} 
%\icmlauthor{Firstname8 Lastname8}{sch} 
%\icmlauthor{Firstname8 Lastname8}{yyy,comp} 
%\icmlauthor{}{sch} 
%\icmlauthor{}{sch} 
\end{icmlauthorlist}

\icmlaffiliation{NYU}{Stern School of Business, New York University} 
\icmlaffiliation{Arena}{Arena Technologies} 
%\icmlaffiliation{yyy}{Department of XXX, University of YYY, Location, Country} 
%\icmlaffiliation{comp}{Company Name, Location, Country} 
%\icmlaffiliation{sch}{School of ZZZ, Institute of WWW, Location, Country}

\icmlcorrespondingauthor{Zijian Liu}{zl3067@stern.nyu.edu} 
%\icmlcorrespondingauthor{Firstname1 Lastname1}{first1.last1@xxx.edu} 
%\icmlcorrespondingauthor{Firstname2 Lastname2}{first2.last2@www.uk}

% You may provide any keywords that you 
% find helpful for describing your paper; these are used to populate 
% the "keywords" metadata in the PDF but will not be shown in the document 
\icmlkeywords{Machine Learning, ICML}

\vskip 0.3in 
]

% this must go after the closing bracket ] following \twocolumn[ ...

% This command actually creates the footnote in the first column 
% listing the affiliations and the copyright notice. 
% The command takes one argument, which is text to display at the start of the footnote. 
% The \icmlEqualContribution command is standard text for equal contribution. 
% Remove it (just {}) if you do not need this facility.

\printAffiliationsAndNotice{}  % leave blank if no need to mention equal contribution 
% \printAffiliationsAndNotice{\icmlEqualContribution} % otherwise use the standard text.
\begin{abstract}
We study the convergence of the shuffling gradient method, a popular
algorithm employed to minimize the finite-sum function with regularization,
in which functions are passed to apply (Proximal) Gradient Descent
(GD) one by one whose order is determined by a permutation on the
indices of functions. In contrast to its easy implementation and effective
performance in practice, the theoretical understanding remains limited.
A recent advance by \cite{pmlr-v235-liu24cg} establishes the first
last-iterate convergence results under various settings, especially
proving the optimal rates for smooth (strongly) convex optimization.
However, their bounds for nonsmooth (strongly) convex functions are
only as fast as Proximal GD. In this work, we provide the first improved
last-iterate analysis for the nonsmooth case demonstrating that the
widely used Random Reshuffle ($\RR$) and Single Shuffle ($\SS$)
strategies are both provably faster than Proximal GD, reflecting the
benefit of randomness. As an important implication, we give the first
(nearly) optimal convergence result for the suffix average under the
$\RR$ sampling scheme in the general convex case, matching the lower
bound shown by \cite{NEURIPS2022_7bc4f74e}.
\end{abstract}

\section{Introduction}

\begin{figure*}[ht]

\noindent\begin{minipage}[t]{1\textwidth}%
\vspace{-0.25in}
\begin{table}[H]
\caption{\label{tab:tab}Summary of our new convergence rates and the best-known
upper/lower bounds under different settings when $T=Kn$ where $K\in\protect\N$.
All results use the function value gap as the convergence measurement.
In the \textquotedbl Shuffling\textquotedbl{} column, $\textsf{ANY}$
means the rate in the same row holds for any type of shuffling scheme
not limited to $\protect\RR/\protect\SS/\protect\IG$. In the \textquotedbl Rate\textquotedbl{}
column, $D_{\star}\triangleq\left\Vert \protect\bx_{\star}-\protect\bx_{1}\right\Vert $
denotes the Euclidean distance (or any upper bound on it) from the
optimal solution $\protect\bx_{\star}$ and the initial point $\protect\bx_{1}$.
$\land$ and $\lor$ indicate $\min$ and $\max$ operations, respectively.
In the last column, $\protect\bx_{Kn+1}^{\protect\avg}\protect\defeq\frac{1}{Kn}\sum_{t=1}^{Kn}\protect\bx_{t+1}$
and $\protect\bx_{Kn+1}^{\protect\suf}\protect\defeq\frac{1}{n}\sum_{t=Kn-n+1}^{Kn}\protect\bx_{t+1}$
respectively refer to the average iterate and the suffix average
of the last one epoch.}
\vspace{0.1in}

\centering{}{\setlength{\extrarowheight}{5pt}%
\begin{tabular}{|c|c|>{\centering}m{0.25\textwidth}|>{\centering}m{0.25\textwidth}|c|}
\hline 
\multicolumn{5}{|c|}{$F=f+\psi$ where $f=\frac{1}{n}\sum_{i=1}^{n}f_{i}$, each $f_{i}$
is convex and $G$-Lipschitz, and $\psi$ is $\mu$-strongly convex}\tabularnewline
\hline 
Setting & Shuffling & Reference & Rate & Output\tabularnewline
\hline 
\multicolumn{1}{|c|}{} & $\textsf{ANY}$ & \cite{pmlr-v235-liu24cg} & $\O\left(\frac{GD_{\star}}{\sqrt{K}}\right)$\footnote{The same rates hold for the last iterate of Proximal GD when the gradient
budget is $Kn$. Also, we remark that these rates cannot apply to
our Algorithm \ref{alg:Alg} once $\psi\neq0$ due to the difference
from the method studied in \cite{pmlr-v235-liu24cg}. See Section
\ref{sec:alg} for details.}\saveFN\sfna\ & $\bx_{Kn+1}$\tabularnewline
\cline{2-5} \cline{3-5} \cline{4-5} \cline{5-5} 
 & \multicolumn{1}{c|}{} & \cite{NEURIPS2022_7bc4f74e} & $\O\left(\frac{GD_{\star}}{n^{1/4}\sqrt{K}}\right)$\footnote{These rates are proved under $\psi=\charf_{\domx}$ where $\charf_{\domx}$
is the characteristic function for the nonempty closed convex set
$\domx$ in $\R^{d}$.}\saveFN\sfnb\ & $\bx_{Kn+1}^{\avg}$\tabularnewline
 & $\RR$ & \textbf{Ours} (Theorem \ref{thm:RR-cvx}) & $\widetilde{\O}\left(\frac{GD_{\star}}{n^{1/4}\sqrt{K}}\right)$\footnote{This rate can automatically improve to $\O\left(\frac{GD_{\star}}{n^{1/4}\sqrt{K}}\right)$,
i.e., no extra logarithmic factor, if $K=\Omega(\log n)$.} & $\bx_{Kn+1}$\tabularnewline
 &  & \textbf{Ours} (Corollary \ref{cor:RR-suffix}) & $\widetilde{\O}\left(\frac{GD_{\star}}{n^{1/4}\sqrt{K}}\right)$ & $\bx_{Kn+1}^{\suf}$\tabularnewline
\cline{2-5} \cline{3-5} \cline{4-5} \cline{5-5} 
$\mu=0$ & \multicolumn{1}{c|}{} & \cite{NEURIPS2022_7bc4f74e} & $\O\left(\frac{GD_{\star}}{n^{1/4}K^{1/4}}\lor\frac{GD_{\star}}{\sqrt{n}}\right)$\useFN\sfnb\ & $\bx_{Kn+1}^{\avg}$\tabularnewline
 & $\SS$ & \textbf{Ours} (Theorem \ref{thm:SS-cvx}) & $\widetilde{\O}\left(\frac{GD_{\star}}{n^{1/4}K^{1/4}}\lor\frac{GD_{\star}}{\sqrt{n}}\right)$ & $\bx_{Kn+1}$\tabularnewline
 &  & \textbf{Ours} (Theorem \ref{thm:SS-cvx-improved}) & $\widetilde{\O}\left(\frac{GD_{\star}}{n^{1/4}K^{1/4}}\land\frac{GD_{\star}}{\sqrt{K}}\right)$\useFN\sfnb\ & $\bx_{Kn+1}$\tabularnewline
\cline{2-5} \cline{3-5} \cline{4-5} \cline{5-5} 
 & $\RR/\SS$ & \cite{NEURIPS2022_7bc4f74e} & $\Omega\left(\frac{1}{n^{1/4}\sqrt{K}}\right)$\footnote{This bound is built under $G=4$, $D_{\star}=1$, and $\psi=\charf_{\domx}$
for $\domx$ being the unit ball centered at $\bzero$. See also discussions
in Subsection \ref{subsec:related}.} & $\bx_{Kn+1}^{\suf}$\tabularnewline
\hline 
\multicolumn{1}{|c|}{} & $\textsf{ANY}$ & \cite{pmlr-v235-liu24cg} & $\widetilde{\O}\left(\frac{\mu D_{\star}^{2}}{K^{2}}+\frac{G^{2}}{\mu K}\right)$\useFN\sfna\ & $\bx_{Kn+1}$\tabularnewline
\cline{2-5} \cline{3-5} \cline{4-5} \cline{5-5} 
$\mu>0$ & $\RR$ & \textbf{Ours} (Theorem \ref{thm:RR-str}) & $\widetilde{\O}\left(\frac{\mu D_{\star}^{2}}{n^{2}K^{2}}+\frac{G^{2}}{\mu\sqrt{n}K}\right)$ & $\bx_{Kn+1}$\tabularnewline
\cline{2-5} \cline{3-5} \cline{4-5} \cline{5-5} 
 & $\SS$ & \textbf{Ours} (Theorem \ref{thm:SS-str}) & $\widetilde{\O}\left(\frac{\mu D_{\star}^{2}}{n^{2}K^{2}}+\frac{G^{2}}{\mu\sqrt{nK}}+\frac{G^{2}}{\mu n}\right)$ & $\bx_{Kn+1}$\tabularnewline
\hline 
\end{tabular}}
\end{table}
\end{minipage}

\end{figure*}

This work considers a common machine learning problem, minimizing
a finite-sum function with regularization, i.e.,
\[
\min_{\bx\in\R^{d}}F(\bx)\defeq f(\bx)+\psi(\bx)\text{ where }f(\bx)\defeq\frac{1}{n}\sum_{i=1}^{n}f_{i}(\bx),
\]
in which $f_{i}$ and $\psi$ are convex and potentially satisfy other
properties, e.g., Lipschitz continuity. Due to the famous empirical
risk minimization framework \cite{shalev2014understanding}, such
a problem arises in a wide range of applications (e.g., SVMs \cite{cortes1995support})
and has been extensively studied in the past few years.

Two text-book level algorithms for solving the problem are Proximal
Gradient Descent (GD) and its variant Proximal Stochastic Gradient
Descent (SGD) \cite{nemirovskij1983problem,nesterov2018lectures,bubeck2015convex,lan2020first},
where the former requires a true gradient in every step in contrast
to the latter only computing the gradient of a single function selected
based on a random index uniformly sampled from $\left[n\right]\defeq\left\{ 1,\cdots,n\right\} $.

Whereas neither of the above classic algorithms is widely adopted
in practice since, when $n$ is large (the standard scenario nowadays),
Proximal GD incurs large computational overhead and Proximal SGD suffers
from cache misses. Instead, the shuffling gradient method is arguably
the most popular and practical choice, in which functions (or data
points) are passed to apply (proximal) gradient descent one by one
whose order is determined by a permutation on $\left[n\right]$. In
particular, three shuffling strategies named Random Reshuffle ($\RR$),
Single Shuffle ($\SS$), and Incremental Gradient ($\IG$) are mostly
used, where the permutation varies randomly in every epoch (containing
$n$ steps) for $\RR$, is randomly sampled at the beginning and employed
through all updates for $\SS$, and is deterministically picked in
advance for $\IG$.

Compared to its easy implementation, lightweight computation, and
effective performance \cite{bottou2009curiously,bottou2012stochastic,bengio2012practical},
the theoretical understanding of the shuffling gradient method remains
limited, especially for the most common output, the last iterate.
A recent advance by \cite{pmlr-v235-liu24cg} establishes the first
last-iterate convergence results measured by the function value gap
under various settings, particularly, proving the optimal rates for
smooth (strongly) convex optimization under the $\RR/\SS/\IG$ sampling
schemes mentioned before. However, their bounds for nonsmooth (strongly)
convex functions are proved for any kind of shuffling strategy (not
limited to $\RR/\SS/\IG$) and only as fast as Proximal GD, leaving
the following unaddressed research question as also mentioned by \cite{pmlr-v235-liu24cg}:
\begin{center}
\textit{For nonsmooth (strongly) convex optimization, can we prove
better last-iterate convergence rates than Proximal GD for $\RR/\SS$
to reflect the benefit of randomness?}
\par\end{center}

\subsection{Our Contributions}

We answer the question affirmatively by establishing the first improved
last-iterate convergence rates for nonsmooth (strongly) convex optimization
under both  $\RR$ and $\SS$ sampling schemes, as summarized in Table
\ref{tab:tab}.

For $\RR$, our new rates are better than the best-known bounds in
\cite{pmlr-v235-liu24cg} by up to a factor of $\Theta(n^{-1/4})$
in the general convex case and a factor of $\Theta(n^{-1/2})$ in
the strongly convex case. As such, our results provide the first concrete
evidence indicating that the $\RR$ sampling scheme indeed converges
faster than Proximal GD, reflecting the benefit of randomness. As
an important implication, we give the first provable and (nearly)
optimal convergence result for the suffix average of the last $n$
points in the optimization trajectory, matching the lower bound shown
by \cite{NEURIPS2022_7bc4f74e} and thus filling in the gap.

For $\SS$, our new rates are better than the bounds in \cite{pmlr-v235-liu24cg}
when the time horizon is below a certain threshold. Specifically,
suppose $T=Kn$ where $K\in\N$, then there exists a critical value
$K_{\star}\in\left(1,n\right]$ such that once $K\leq K_{\star}$,
our bounds decay faster than \cite{pmlr-v235-liu24cg} (and also Proximal
GD) for both general and strongly convex optimization. In the special
case of constrained optimization (i.e., $\psi=\charf_{\domx}$ where
$\charf_{\domx}$ is the characteristic function for the nonempty
closed convex set $\domx$ in $\R^{d}$), we sharpen our bound further
and obtain an improved rate better than \cite{pmlr-v235-liu24cg}
for any $K\in\N$. These results suggest that the $\SS$ strategy
also beats Proximal GD (at least partially), demonstrating the benefit
of using random permutations.

We also highlight that our results (except Theorem \ref{thm:SS-cvx-improved})
hold for \textit{any} $T\in\N$, which as far as we know is new in
the literature on shuffling gradient methods for convex optimization.

Moreover, we propose a novel sufficient condition to guarantee the
last-iterate convergence when the index is selected in a general manner,
not limited to shuffling-based methods.

\subsection{Related Work\label{subsec:related}}

Due to limited space, we will only review the study of shuffling-based
gradient methods for nonsmooth (strongly) convex optimization. For
details and progress in smooth optimization, the reader could refer
to \cite{gurbuzbalaban2019convergence,gurbuzbalaban2021random,ying2018stochastic,pmlr-v97-haochen19a,pmlr-v97-nagaraj19a,pmlr-v119-rajput20a,pmlr-v125-safran20a,NEURIPS2020_cb8acb1d,NEURIPS2020_c8cc6e90,JMLR:v22:20-1238,NEURIPS2021_803ef568,rajput2022permutationbased,pmlr-v162-mishchenko22a,pmlr-v162-tran22a,pmlr-v202-cha23a,NEURIPS2024_84d39572,cai2025last}
for the convex case and \cite{solodov1998incremental,li2019incremental,NEURIPS2020_c8cc6e90,JMLR:v22:20-1238,pmlr-v139-tran21b,pauwels2021incremental,lu2022a,mohtashami2022characterizing,NEURIPS2022_3acb4925,li2023convergence,NEURIPS2023_eeb57fdf,yu2023high,qiu2023new,pmlr-v235-koloskova24a,qiu2024random,josz2024proximal}
for the nonconvex case.

In the following, we assume the time horizon $T=Kn$ where $K\in\N$
for simplicity and only focus on the dependence of $n$ and $K$ in
the convergence rate.

\textbf{Upper Bound. }In the general convex case, the first $\O\left(\frac{1}{\sqrt{K}}\right)$
rate is established by \cite{nedic2001incremental} for $\IG$ when
$\psi=\charf_{\domx}$. Later in \cite{bertsekas2011incrementalgs},
the requirement $\psi=\charf_{\domx}$ is relaxed to $\psi=\varphi+\charf_{\domx}$
where $\varphi$ needs to be Lipschitz on $\domx$. \cite{NIPS2016_c74d97b0}
studies the generalized linear model and provides the $\O\left(\frac{1}{\sqrt{n}}\right)$
upper bound for $\RR/\SS$ when $K=1$. As for general objectives
and general $K\in\N$, \cite{NEURIPS2022_7bc4f74e} is the only work
showing the $\O\left(\frac{1}{n^{1/4}\sqrt{K}}\right)$ rate for $\RR$
and the $\O\left(\frac{1}{n^{1/4}K^{1/4}}\lor\frac{1}{\sqrt{n}}\right)$
rate for $\SS$, both under $\psi=\charf_{\domx}$. However, all the
results mentioned until now only work for the average iterate. Recently,
\cite{pmlr-v235-liu24cg} gives the first last-iterate convergence
result $\O\left(\frac{1}{\sqrt{K}}\right)$ being applied to any convex
$\psi$ and any type of shuffling strategy not limited to $\RR/\SS/\IG$.

In the strongly convex case, \cite{kibardin1979decomposition,nedic2001convergence}
show the $\IG$ sampling scheme guarantees the $\O\left(\frac{1}{K}\right)$
convergence measured by the squared distance from the optimal solution
and the last iterate, assuming strongly convex $f$ and $\psi=\charf_{\domx}$.
\cite{pmlr-v235-liu24cg} proves a similar $\widetilde{\O}\left(\frac{1}{K}\right)$
rate for strongly convex $\psi$ with improvements in two aspects:
one is using a stronger criterion, the function value gap, to measure
convergence, the other is that their result holds for any shuffling
scheme not restricted to $\RR/\SS/\IG$.

\textbf{Lower Bound. }The lower complexity bound of shuffling gradient
methods for nonsmooth (strongly) convex optimization is a long-open
problem. The first insightful observation is by \cite{pmlr-v97-nagaraj19a}
pointing out that any lower bound established for the deterministic
case is also valid here as one can take $f_{i}\equiv f$ and $\psi=\charf_{\domx}$
where $\domx$ is a certain convex set (usually a ball in $\R^{d}$
centered at $\bzero$). Such a reduction immediately implies two results
(not limited to the last iterate) working for any shuffling strategy,
i.e., $\Omega\left(\frac{1}{\sqrt{nK}}\right)$ for the general convex
case and $\Omega\left(\frac{1}{nK}\right)$\footnote{\label{fn:lb}A subtle point is that this lower bound is established
for strongly convex $f$ instead of $\psi$, not strictly fitting
our Assumption \ref{assu:basic}. However, by slightly modifying the
existing proof \cite{bubeck2015convex} to make it work for the first-order
algorithm containing a proximal update step, we can show the same
bound still holds for strongly convex $\psi$. See Appendix \ref{sec:lb}
for details.} for the strongly convex case \cite{nemirovskij1983problem,nesterov2018lectures,bubeck2015convex}.

Though these two lower bounds may be too optimistic for shuffling-based
gradient methods, no further progress has been made until \cite{NEURIPS2022_7bc4f74e},
showing that both $\RR$ and $\SS$ sampling schemes with a constant
stepsize $\eta$ in the general convex case admit the lower bound
$\Omega\left(\min\left\{ 1,\eta\sqrt{\frac{n}{J}}+\eta+\frac{1}{\eta nK}\right\} \right)$
for the suffix average of the last $J$ epochs (i.e., $\frac{1}{Jn}\sum_{j=K-J}^{K-1}\sum_{i=1}^{n}\bx_{jn+i+1}$).
Noticing $\Omega\left(\eta\sqrt{\frac{n}{J}}+\frac{1}{\eta nK}\right)\geq\Omega\left(\frac{1}{J^{1/4}n^{1/4}\sqrt{K}}\right)$
and $\Omega\left(\eta+\frac{1}{\eta nK}\right)\geq\Omega\left(\frac{1}{\sqrt{nK}}\right)$,
we hence can simplify the bound into $\Omega\left(\frac{1}{J^{1/4}n^{1/4}\sqrt{K}}+\frac{1}{\sqrt{nK}}\right)$.
Especially, this implies the $\Omega\left(\frac{1}{n^{1/4}\sqrt{K}}\right)$
barrier for the suffix average of the last one epoch as listed in
Table \ref{tab:tab}.

However, whether the $\Omega\left(\frac{1}{n^{1/4}\sqrt{K}}\right)$
bound also holds for the last iterate under $\RR/\SS$ is still unclear.
Furthermore, whether the general lower bound $\Omega\left(\frac{1}{nK}\right)$
for the strongly convex case mentioned above is tight for shuffling
gradient methods remains unknown as well.

\section{Preliminary}

\textbf{Notation. }$\N$ is the set of all positive integers and $\left[m\right]\defeq\left\{ 1,\dots,m\right\} ,\forall m\in\N$.
$a\land b$ and $a\lor b$ respectively indicate $\min\left\{ a,b\right\} $
and $\max\left\{ a,b\right\} $. $X\diseq Y$ means that two random
variables $X$ and $Y$ have the same probability distribution. $\left\langle \cdot,\cdot\right\rangle $
denotes the Euclidean inner product on $\R^{d}$. $\left\Vert \cdot\right\Vert \defeq\sqrt{\left\langle \cdot,\cdot\right\rangle }$
is the $2$-norm. Given an extended real-valued convex function $h:\R^{d}\to\bR$
where $\bR\defeq\left(-\infty,+\infty\right]$, $\dom h\defeq\left\{ \bx\in\R^{d}:h(\bx)<+\infty\right\} $.
For any $\bx\in\dom h$, $\partial h(\bx)\defeq\left\{ \bg\in\R^{d}:h(\by)\geq h(\bx)+\left\langle \bg,\by-\bx\right\rangle ,\forall\by\in\R^{d}\right\} $
is the set of subgradients at $\bx$. We denote by $\nabla h(\bx)$
an element in $\partial h(\bx)$ when $\partial h(\bx)\neq\varnothing$.
Throughout the paper, $\domx$ always denotes a nonempty closed convex
set in $\R^{d}$ and $\charf_{\domx}$ represents its characteristic
function, i.e., $\charf_{\domx}(\bx)=0$ if $\bx\in\domx$, $+\infty$
otherwise.

In this work, we study the following optimization problem
\[
\min_{\bx\in\R^{d}}F(\bx)\defeq f(\bx)+\psi(\bx)\text{ where }f(\bx)\defeq\frac{1}{n}\sum_{i=1}^{n}f_{i}(\bx).
\]

Our analysis relies on two mild assumptions.
\begin{assumption}
\label{assu:basic}\textup{Each $f_{i}:\R^{d}\to\R$ is convex. $\psi:\R^{d}\to\bR$
is proper, closed, and convex. Moreover, there exists $\mu\geq0$
such that $\psi(\bx)-\psi(\by)-\left\langle \nabla\psi(\by),\bx-\by\right\rangle \geq\frac{\mu}{2}\left\Vert \bx-\by\right\Vert ^{2},\forall\bx\in\R^{d},\by\in\dom\psi,\nabla\psi(\by)\in\partial\psi(\by)$
whenever $\partial\psi(\by)\neq\varnothing$.}
\end{assumption}

Under Assumption \ref{assu:basic}, $\partial f_{i}(\bx)$ is always
nonempty for any $\bx\in\R^{d}$ and $i\in\left[n\right]$ since $\dom f_{i}=\R^{d}$.
\begin{assumption}
\label{assu:lip}\textup{Each $f_{i}$ is $G_{i}$-Lipschitz on $\dom\psi$
for some $G_{i}>0$.}
\end{assumption}

We remark that Assumption \ref{assu:lip} only requires $f_{i}$ to
be Lipschitz on $\dom\psi$ instead of the whole space $\R^{d}$.
Hence, it is also possible to consider the case of strongly convex
$f_{i}$ or $f$ without the domain issue pointed out by \cite{pmlr-v80-nguyen18c}.
However, to keep it simple, we only focus on the situation of $\psi$
being possibly strongly convex.

\section{General Proximal Gradient Method\label{sec:alg}}

\begin{algorithm}[H]
\caption{\label{alg:Alg}General Proximal Gradient Method}

\textbf{Input:} initial point $\bx_{1}\in\dom\psi$, stepsize $\eta_{t}>0,\forall t\in\left[T\right]$.

\textbf{for} $t=1$ \textbf{to} $T$\textbf{ do}

$\quad$Generate an index $\I(t)\in\left[n\right]$

$\quad$$\bx_{t+1}=\argmin_{\bx\in\R^{d}}\psi(\bx)+\left\langle \nabla f_{\I(t)}(\bx_{t}),\bx\right\rangle +\frac{\left\Vert \bx-\bx_{t}\right\Vert ^{2}}{2\eta_{t}}$

\textbf{Output:} $\bx_{T+1}$
\end{algorithm}

\begin{rem}
Algorithm \ref{alg:Alg} is also known as Incremental Subgradient-Proximal
Method \cite{bertsekas2011incrementalgs}. However, we use a different
name here to distinguish it from the term Incremental Gradient in
the literature.
\end{rem}

The algorithmic framework studied in the paper, General Proximal Gradient
Method \cite{bertsekas2011incrementalgs}, is provided in Algorithm
\ref{alg:Alg}. We highlight three key differences from the prior
proximal shuffling gradient methods \cite{kibardin1979decomposition,pmlr-v162-mishchenko22a,pmlr-v235-liu24cg,josz2024proximal}.
First, Algorithm \ref{alg:Alg} is more general since the generation
process of $\I(t)$ is not limited to shuffling-based. Second, Algorithm
\ref{alg:Alg} works for any $T\in\N$, in contrast to $T=Kn$ where
$K\in\N$ required in the studies mentioned above. Moreover, the proximal
update in Algorithm \ref{alg:Alg} happens in every step instead of
at the end of every epoch (containing $n$ gradient descent steps)
in the existing algorithms.

Now we provide some concrete examples of how to generate the index
$\I(t)$. The first one is Example \ref{exa:SGD}, showing that Proximal
SGD is a special case of Algorithm \ref{alg:Alg}.
\begin{example}
\label{exa:SGD}When $\I(1)$ to $\I(T)$ are mutually independent
random variables uniformly distributed on $\left[n\right]$, Algorithm
\ref{alg:Alg} recovers the famous Proximal SGD algorithm.
\end{example}

Next, to formally define different shuffling strategies, we require
some new notations. Henceforth, $\re(t)$ denotes the modulo operation
of $n$, i.e., $\re(t)\defeq t\mod n$, where we use the convention
$Kn\mod n=n,\forall K\in\N$. In addition, we let $\qu(t)$ be the
smallest integer greater than or equal to $\frac{t}{n}$, i.e., $\qu(t)\defeq\left\lceil \frac{t}{n}\right\rceil $
where $\left\lceil \cdot\right\rceil $ is the ceiling function. Remarkably,
the equation $t=(\qu(t)-1)n+\re(t),\forall t\in\N$ always holds.
Lastly, we denote by $S_{n}$ the symmetric group of $\left[n\right]$,
i.e., the set containing all permutations of $\left[n\right]$. 

Equipped with these notations, we introduce the commonly used $\RR$/$\SS$/$\IG$
shuffling schemes as follows.
\begin{example}
\label{exa:RR}When $\I(t)=\pi_{\qu(t)}^{\re(t)}$ where $\pi_{1}$
to $\pi_{\qu(T)}$ are mutually independent random permutations uniformly
distributed on $S_{n}$, it is called the $\RR$ sampling scheme.
\end{example}

\begin{example}
\label{exa:SS}When $\I(t)=\pi^{\re(t)}$ where $\pi$ is a random
permutation uniformly distributed on $S_{n}$, it is called the $\SS$
sampling scheme.
\end{example}

\begin{example}
\label{exa:IG}When $\I(t)=\pi^{\re(t)}$ where $\pi$ is a deterministic
permutation in $S_{n}$, it is called the $\IG$ sampling scheme.
\end{example}

\section{Improved Last-Iterate Convergence Rates\label{sec:rates}}

In this section, we provide our improved last-iterate convergence
rates for Algorithm \ref{alg:Alg} under both $\RR$ and $\SS$ sampling
schemes. To simplify the notation, we define $G_{f,1}\defeq\frac{1}{n}\sum_{i=1}^{n}G_{i}$
and $G_{f,2}\defeq\sqrt{\frac{1}{n}\sum_{i=1}^{n}G_{i}^{2}}$, respectively
representing the arithmetic mean and the root mean square of Lipschitz
parameters. Notably, the following inequality always holds
\begin{equation}
G_{f,1}\leq G_{f,2}<\sqrt{n}G_{f,1}.\label{eq:G}
\end{equation}

Next, to make easy and fair comparisons to the best existing convergence
results of shuffling gradient methods for nonsmooth (strongly) convex
optimization \cite{NEURIPS2022_7bc4f74e,pmlr-v235-liu24cg}, we make
two extra assumptions here. One is the existence of a point $\bx_{\star}\in\R^{d}$
attaining the minimum value of $F$, i.e., $F(\bx_{\star})=F_{\star}\defeq\inf_{\bx\in\R^{d}}F(\bx)$.
Under this assumption, let $D_{\star}\defeq\left\Vert \bx_{\star}-\bx_{1}\right\Vert $
denote the distance between the optimal solution and the initial point.
The other is assuming the time horizon satisfies $T\geq n$ (or one
can simply think $T=Kn$ for $K\in\N$ as in prior works).
\begin{rem}
We clarify that the above assumptions are both unnecessary in the
full statement of every theorem (except Theorem \ref{thm:SS-cvx-improved}).
Concretely, for any reference point $\bz\in\R^{d}$ and $T\in\N$,
the gap $\E\left[F(\bx_{T+1})-F(\bz)\right]$ can always be properly
upper bounded. See Appendix \ref{sec:full-thm} for details.
\end{rem}

In addition, following the convention in nonsmooth optimization for
the general convex case \cite{nesterov2018lectures,lan2020first},
the value of $\eta$ used in the stepsize $\eta_{t}$ in Theorems
\ref{thm:RR-cvx}, \ref{thm:SS-cvx}, and \ref{thm:SS-cvx-improved}
has been optimized to obtain the best dependence on problem-dependent
parameters, e.g., $G_{f,1}$, $G_{f,2}$, and $D_{\star}$. Rates
working for arbitrarily picked $\eta$ are deferred to the corresponding
full version of each theorem in Appendix \ref{sec:full-thm}, in which
the precise logarithmic factor hidden in the $\widetilde{\O}$ notation
is also provided.

\subsection{$\protect\RR$ Sampling Scheme}

This subsection will focus on the $\RR$ sampling scheme. As the reader
will see, our new last-iterate bounds are always better than the best-known
results in \cite{pmlr-v235-liu24cg}.
\begin{thm}
\label{thm:RR-cvx}Under Assumptions \ref{assu:basic} (with $\mu=0$)
and \ref{assu:lip}, suppose the $\RR$ sampling scheme is employed
with one of the following three stepsizes $\eta_{t},\forall t\in\left[T\right]$:
\begin{itemize}
\item $\eta_{t}=\eta\frac{\qu(T)-\qu(t)+1}{\qu(T)\sqrt{T}}$ and $\eta=\frac{D_{\star}}{n^{1/4}\sqrt{G_{f,1}G_{f,2}\left(1+\frac{\log n}{\qu(T)}\right)}}$.
\item $\eta_{t}=\frac{\eta}{\sqrt{T}}$ and $\eta=\frac{D_{\star}}{n^{1/4}\sqrt{G_{f,1}G_{f,2}(1+\log T)}}.$
\item $\eta_{t}=\frac{\eta}{\sqrt{t}}$ and $\eta=\frac{D_{\star}}{n^{1/4}\sqrt{G_{f,1}G_{f,2}}}$.
\end{itemize}
Then Algorithm \ref{alg:Alg} guarantees
\[
\E\left[F(\bx_{T+1})-F_{\star}\right]\leq\widetilde{\O}\left(\frac{n^{1/4}\sqrt{G_{f,1}G_{f,2}}D_{\star}}{\sqrt{T}}\right).
\]
If additionally assuming $T=\Omega(n\log n)$, then the first stepsize
choice achieves the following improved rate
\[
\E\left[F(\bx_{T+1})-F_{\star}\right]\leq O\left(\frac{n^{1/4}\sqrt{G_{f,1}G_{f,2}}D_{\star}}{\sqrt{T}}\right).
\]
\end{thm}

We start with the general convex case and discuss Theorem \ref{thm:RR-cvx}
in detail here. As far as we know, the best and only last-iterate
bound that can be applied to the same setting is $\O\left(\frac{G_{f,1}D_{\star}}{\sqrt{K}}\right)$
for $T=Kn$ where $K\in\N$ \cite{pmlr-v235-liu24cg}. In that case,
Theorem \ref{thm:RR-cvx} achieves the rate $\O\left(\frac{\sqrt{G_{f,1}G_{f,2}}D_{\star}}{n^{1/4}\sqrt{K}}\right)$.
Note that by (\ref{eq:G}), there is 
\[
\frac{\sqrt{G_{f,1}G_{f,2}}D_{\star}}{n^{1/4}\sqrt{K}}\Big/\frac{G_{f,1}D_{\star}}{\sqrt{K}}=\frac{1}{n^{1/4}}\sqrt{\frac{G_{f,2}}{G_{f,1}}}\in\left[\frac{1}{n^{1/4}},1\right).
\]
Therefore, our new result is always better than \cite{pmlr-v235-liu24cg}
by up to a factor of $\Theta(n^{-1/4})$.

In particular, the $\Theta(n^{-1/4})$ improvement can be achieved
when $G_{i}\equiv G$, leading to the rate $\O\left(\frac{GD_{\star}}{n^{1/4}\sqrt{K}}\right)$.
Remarkably, such a rate is as fast as the previously best-known bound
established only for the average iterate $\bx_{Kn+1}^{\avg}\defeq\frac{1}{Kn}\sum_{t=1}^{Kn}\bx_{t+1}$
when $\psi=\charf_{\domx}$ \cite{NEURIPS2022_7bc4f74e}.

An important implication of Theorem \ref{thm:RR-cvx} is to provide
the convergence of $\bx_{T+1}^{\suf}\defeq\frac{1}{n}\sum_{t=T-n+1}^{T}\bx_{t+1}$
as follows.
\begin{cor}
\label{cor:RR-suffix}Under the same setting in Theorem \ref{thm:RR-cvx}
(using the third stepsize), Algorithm \ref{alg:Alg} guarantees 
\[
\E\left[F(\bx_{T+1}^{\suf})-F_{\star}\right]\leq\widetilde{\O}\left(\frac{n^{1/4}\sqrt{G_{f,1}G_{f,2}}D_{\star}}{\sqrt{T}}\right).
\]
\end{cor}

\begin{proof}
Due to the convexity of $F$, $\E\left[F(\bx_{T+1}^{\suf})-F_{\star}\right]\leq\frac{1}{n}\sum_{t=T-n+1}^{T}\E\left[F(\bx_{t+1})-F_{\star}\right].$
We conclude from Theorem \ref{thm:RR-cvx} and the inequality $\frac{1}{n}\sum_{t=T-n+1}^{T}\frac{1}{\sqrt{t}}\leq\frac{2}{\sqrt{T}}$.
\end{proof}

To our best knowledge, Corollary \ref{cor:RR-suffix} is not only
the first provable but also the first optimal rate for the suffix
average since when $T=Kn$ and $G_{i}\equiv G$, it matches the lower
bound $\Omega\left(\frac{1}{n^{1/4}\sqrt{K}}\right)$ (up to logarithmic
factors) shown by \cite{NEURIPS2022_7bc4f74e} proved for $\psi=\charf_{\domx}$.
However, the careful reader may argue that Corollary \ref{cor:RR-suffix}
is not convincing because the original lower bound is established
for the constant stepsize $\eta_{t}\equiv\eta$ and has a stronger
version depending on the value of $\eta$ (see Subsection \ref{subsec:related}).
In Corollary \ref{cor:RR-suffix-full}, we close the gap by giving
$\E\left[F(\bx_{T+1}^{\suf})-F_{\star}\right]\leq\widetilde{\O}\left(\frac{D_{\star}^{2}}{\eta T}+\eta\sqrt{n}G_{f,1}G_{f,2}\right)$
when $\eta_{t}\equiv\eta$ and $T\geq2(n-1)$, which perfectly matches
the original lower bound in \cite{NEURIPS2022_7bc4f74e} by up to
logarithmic factors.

Moreover, we want to talk about the first stepsize choice $\eta_{t}=\eta\frac{\qu(T)-\qu(t)+1}{\qu(T)\sqrt{T}}$,
which is inspired by \cite{pmlr-v235-liu24cg} who showed that the
stepsize schedule proportional to $K-k+1$ in the $k$-th epoch when
$T=Kn$ can remove any extra logarithmic factor in the final rate.
Here, we prove this strategy can also be applied to arbitrary $T\in\N$
but will incur an additional $\O\left(\sqrt{\frac{\log n}{\qu(T)}}\right)$
factor, which is at most in the order of $\O(\sqrt{\log n})$ and
is automatically shaved off once $\qu(T)=\Omega(\log n)\Leftrightarrow T=\Omega(n\log n)$.
We emphasize that, though the principle of our first stepsize is highly
similar to \cite{pmlr-v235-liu24cg}, showing it indeed works for
any $T\in\N$ requires a refined analysis, for example, see Lemma
\ref{lem:stepsize} in the appendix.
\begin{thm}
\label{thm:RR-str}Under Assumptions \ref{assu:basic} (with $\mu>0$)
and \ref{assu:lip}, suppose the $\RR$ sampling scheme is employed
with the stepsize $\eta_{t}=\frac{2}{\mu t},\forall t\in\left[T\right]$,
then Algorithm \ref{alg:Alg} guarantees
\[
\E\left[F(\bx_{T+1})-F_{\star}\right]\leq\widetilde{\O}\left(\frac{\mu D_{\star}^{2}}{T^{2}}+\frac{\sqrt{n}G_{f,1}G_{f,2}}{\mu T}\right).
\]
\end{thm}

Now we turn our attention to the strongly convex case. As before,
we first check the case $T=Kn$, under which Theorem \ref{thm:RR-str}
gives us the rate $\widetilde{\O}\left(\frac{\mu D_{\star}^{2}}{n^{2}K^{2}}+\frac{G_{f,1}G_{f,2}}{\mu\sqrt{n}K}\right)$.
In comparison, the only last-iterate bound in the same setting is
$\widetilde{\O}\left(\frac{\mu D_{\star}^{2}}{K^{2}}+\frac{G_{f,1}^{2}}{\mu K}\right)$
\cite{pmlr-v235-liu24cg}. As one can see, the higher order term $\frac{\mu D_{\star}^{2}}{n^{2}K^{2}}$
achieves acceleration by a factor of $\Theta(n^{-2})$. Notably, the
lower order term $\frac{G_{f,1}G_{f,2}}{\mu\sqrt{n}K}$ is also always
faster due to $G_{f,2}<\sqrt{n}G_{f,1}$ in the aforementioned inequality
(\ref{eq:G}). If further assuming $G_{i}\equiv G$, Theorem \ref{thm:RR-str}
yields the rate $\widetilde{\O}\left(\frac{G^{2}}{\mu\sqrt{n}K}\right)$
(suppose $K$ is large for simplicity) significantly improving upon
the bound $\widetilde{\O}\left(\frac{G^{2}}{\mu K}\right)$ in \cite{pmlr-v235-liu24cg}
by a factor of $\Theta(n^{-1/2})$.

In addition, we also extend the convergence result to the stepsize
$\eta_{t}=\frac{m}{\mu t}$ for any $m\in\N$ like \cite{pmlr-v235-liu24cg}.
The interested reader could refer to Theorem \ref{thm:RR-str-full}
for details.

Before ending this subsection, we should mention that though both
Theorems \ref{thm:RR-cvx} and \ref{thm:RR-str} improve upon \cite{pmlr-v235-liu24cg}
and indicate that $\RR$ enjoys faster convergence than Proximal GD
suggesting the benefit of randomness, these results are however still
slower than Proximal SGD, whose last iterate is known to converge
in $\O\left(\frac{G_{f,2}D_{\star}}{\sqrt{T}}\right)$ and $\widetilde{\O}\left(\frac{\mu D_{\star}^{2}}{T^{2}}+\frac{G_{f,2}^{2}}{\mu T}\right)$
for general and strongly functions, respectively \cite{pmlr-v99-harvey19a,doi:10.1137/19M128908X,orabona2020blog,liu2024revisiting}.
Whether $\RR$ can be proved to guarantee the same rates is unclear\footnote{\label{fn:SGD}We incline to a negative answer at least for the general
convex case. See our discussion after Corollary \ref{cor:RR-suffix-full}
for why.}.

\subsection{$\protect\SS$ Sampling Scheme}

We study the $\SS$ sampling scheme in this subsection. Unlike the
always faster rates achieved by $\RR$ presented before, our new results
for $\SS$ are better than \cite{pmlr-v235-liu24cg} only when a certain
condition is met.
\begin{thm}
\label{thm:SS-cvx}Under Assumptions \ref{assu:basic} (with $\mu=0$)
and \ref{assu:lip}, suppose the $\SS$ sampling scheme is employed
with the stepsize $\eta_{t}=\frac{\eta}{\sqrt{T}},\forall t\in\left[T\right]$
and $\eta=\frac{D_{\star}}{\sqrt{\left(\qu(T)G_{f,2}^{2}\lor\sqrt{n\qu(T)}G_{f,1}G_{f,2}\right)\left(1+\log T\right)}}$,
then Algorithm \ref{alg:Alg} guarantees
\[
\E\left[F(\bx_{T+1})-F_{\star}\right]\leq\widetilde{\O}\left(\frac{\sqrt{G_{f,1}G_{f,2}}D_{\star}}{T^{1/4}}\lor\frac{G_{f,2}D_{\star}}{\sqrt{n}}\right).
\]
\end{thm}

We begin with the general convex case in Theorem \ref{thm:SS-cvx}.
To understand how the above rate compares to the bound $\O\left(\frac{G_{f,1}D_{\star}}{\sqrt{K}}\right)$
when $T=Kn$ where $K\in\N$ \cite{pmlr-v235-liu24cg}, we introduce
a critical value $K_{\star}\defeq\frac{nG_{f,1}^{2}}{G_{f,2}^{2}}$,
which falls into $\left(1,n\right]$ due to (\ref{eq:G}). If $K\leq K_{\star}$,
we notice the rate in Theorem \ref{thm:SS-cvx} equals $\widetilde{\O}\left(\frac{\sqrt{G_{f,1}G_{f,2}}D_{\star}}{n^{1/4}K^{1/4}}\right)$
and is faster than $\O\left(\frac{G_{f,1}D_{\star}}{\sqrt{K}}\right)$.
Otherwise, the rate in Theorem \ref{thm:SS-cvx} equals $\widetilde{\O}\left(\frac{G_{f,2}D_{\star}}{\sqrt{n}}\right)$
and is slower than $\O\left(\frac{G_{f,1}D_{\star}}{\sqrt{K}}\right)$.
This observation suggests an interesting convergence phenomenon for
the $\SS$ sampling scheme, that is the function value gap will decay
in the rate no slower than $\widetilde{\O}\left(\frac{\sqrt{G_{f,1}G_{f,2}}D_{\star}}{n^{1/4}K^{1/4}}\right)$
during the first $\Theta(K_{\star})$ epochs until reaching the $\widetilde{\O}\left(\frac{G_{f,2}D_{\star}}{\sqrt{n}}\right)$
error regime. We emphasize that this does not necessarily imply the
$\SS$ strategy must bear constant optimization error since Theorem
\ref{thm:SS-cvx} only states a convergence upper bound. In other
words, Theorem \ref{thm:SS-cvx} improves upon \cite{pmlr-v235-liu24cg}
when $K$ is small. Especially, if $G_{i}\equiv G$, the rate in Theorem
\ref{thm:SS-cvx} and the threshold $K_{\star}=n$ match the convergence
behavior of $\bx_{Kn+1}^{\avg}$ when $\psi=\charf_{\domx}$ proved
in \cite{NEURIPS2022_7bc4f74e}.

The reader may not feel satisfied with Theorem \ref{thm:SS-cvx} since
the rate does not vanish as $T$ becomes larger and could ask whether
this is an inherent issue of Algorithm \ref{alg:Alg} under $\SS$.
The answer turns out to be negative. A simple but important observation
is that when $\psi=0$ and $T=Kn$, Algorithm \ref{alg:Alg} with
shuffling-based $\I(t)$ and the algorithm studied in \cite{pmlr-v235-liu24cg}
are actually the same method. We recall a key fact that the rate $\O\left(\frac{G_{f,1}D_{\star}}{\sqrt{K}}\right)$
in \cite{pmlr-v235-liu24cg} is proved for any shuffling type as mentioned
in Subsection \ref{subsec:related} (or see Table \ref{tab:tab}).
Hence, at least in the case of $\psi=0$ and $T=Kn$, we should expect
a refined upper bound that converges to $0$ when $K$ approaches
infinity.

Built upon this thought, we prove the following sharper rate when
$\psi=\charf_{\domx}$ (i.e., constrained optimization) and $T=Kn$.
\begin{thm}
\label{thm:SS-cvx-improved}Under Assumptions \ref{assu:basic} (with
$\psi=\charf_{\domx}$) and \ref{assu:lip}, suppose $T=Kn$ where
$K\in\N$ and the $\SS$ sampling scheme is employed with the stepsize
$\eta_{t}=\frac{\eta}{\sqrt{T}},\forall t\in\left[T\right]$ and $\eta=\frac{D_{\star}}{\sqrt{\left(\sqrt{nK}G_{f,1}G_{f,2}\land nG_{f,1}^{2}\right)(1+\log nK)}}$,
then Algorithm \ref{alg:Alg} guarantees
\[
\E\left[F(\bx_{Kn+1})-F_{\star}\right]\leq\widetilde{\O}\left(\frac{\sqrt{G_{f,1}G_{f,2}}D_{\star}}{n^{1/4}K^{1/4}}\land\frac{G_{f,1}D_{\star}}{\sqrt{K}}\right).
\]
\end{thm}

Remarkably, the above result integrates the advantages of Theorem
\ref{thm:SS-cvx} and \cite{pmlr-v235-liu24cg} and is thereby faster
than both. One point we remind the reader is that the existing analysis
in \cite{pmlr-v235-liu24cg} immediately invalids once $\domx\neq\R^{d}$
(equivalently $\psi\neq0$) since in that case Algorithm \ref{alg:Alg}
is no longer the same as the method in \cite{pmlr-v235-liu24cg} due
to different places for the proximal update as discussed in Section
\ref{sec:alg}. Therefore, a more careful analysis is required to
overcome this issue. We also mention that the condition $\psi=\charf_{\domx}$
can be further relaxed to $\psi=\varphi+\charf_{\domx}$ for $\varphi$
being Lipschitz on $\domx$ (see Lemma \ref{lem:special-refined}
for details) as previously assumed in \cite{bertsekas2011incrementalgs}.
However, how to extend Theorem \ref{thm:SS-cvx-improved} to any general
$\psi$ and $T\in\N$ is unknown to us currently, which is left as
an important future work.
\begin{thm}
\label{thm:SS-str}Under Assumptions \ref{assu:basic} (with $\mu>0$)
and \ref{assu:lip}, suppose the $\SS$ sampling scheme is employed
with the stepsize $\eta_{t}=\frac{2}{\mu t},\forall t\in\left[T\right]$,
then Algorithm \ref{alg:Alg} guarantees
\[
\E\left[F(\bx_{T+1})-F_{\star}\right]\leq\widetilde{\O}\left(\frac{\mu D_{\star}^{2}}{T^{2}}+\frac{G_{f,1}G_{f,2}}{\mu\sqrt{T}}+\frac{G_{f,2}^{2}}{\mu n}\right).
\]
\end{thm}

Lastly, we move to the strongly convex case. When $T=Kn$ where $K\in\N$,
one can check the critical time to balance the last two terms is still
$K_{\star}=\frac{nG_{f,1}^{2}}{G_{f,2}^{2}}\in\left(1,n\right]$.
Consequently, compared to $\widetilde{\O}\left(\frac{\mu D_{\star}^{2}}{K^{2}}+\frac{G_{f,1}^{2}}{\mu K}\right)$
\cite{pmlr-v235-liu24cg}, Theorem \ref{thm:SS-str} outperforms in
the regime $K\leq K_{\star}$.

One may expect that when $T$ is a multiple of $n$ and $\psi$ satisfies
certain conditions, we can again obtain an improved rate better than
both Theorem \ref{thm:SS-str} and \cite{pmlr-v235-liu24cg}. However,
even under some additional assumptions, the best improvement upon
Theorem \ref{thm:SS-str} we can obtain currently is the rate $\widetilde{\O}\left(\frac{nG_{f,1}^{2}}{\mu K}\right)$\footnote{\label{fn:SS}The reader who wonders why we cannot do better could
refer to the discussion of Lemma \ref{lem:special-refined}.} (suppose $K$ is large enough). Unfortunately, this is worse than
$\widetilde{\O}\left(\frac{G_{f,1}^{2}}{\mu K}\right)$ in \cite{pmlr-v235-liu24cg}.
Therefore, we simply leave Theorem \ref{thm:SS-str} in its current
form without providing any further refined result and hope that a
bound faster than both Theorem \ref{thm:SS-str} and \cite{pmlr-v235-liu24cg}
could be found in the future.

Like the strongly convex case for $\RR$, the stepsize in Theorem
\ref{thm:SS-str} can also be generalized to $\eta_{t}=\frac{m}{\mu t}$
for any $m\in\N$. For details, see Theorem \ref{thm:SS-str-full}.

To summarize, this subsection provides a new picture for the last-iterate
convergence of the $\SS$ sampling scheme. Precisely, we show $\SS$
indeed boosts the convergence leading to a faster rate than Proximal
GD at least in the early optimization phase, which partially improves
upon \cite{pmlr-v235-liu24cg} and reflects the benefit of randomness.
In the special case of constrained general convex optimization, we
give an even sharper bound demonstrating that $\SS$ provably beats
Proximal GD no matter how long the time horizon is. In contrast, such
information is missed in the best previous bounds \cite{pmlr-v235-liu24cg}.
However, the same as the $\RR$ strategy, there is still a gap between
$\SS$ and Proximal SGD. Understanding this difference is an important
future work.

\section{Key Ideas and Proof Sketch\label{sec:idea}}

This section describes our key ideas in the analysis, provides the
most important Lemma \ref{lem:last-main}, and then uses it to sketch
the proof. Any omitted detail can be found in the appendix.

\subsection{The Existing Analysis Misses Randomness}

Before talking about our analysis, we first provide an intuitive
explanation for why \cite{pmlr-v235-liu24cg} can only achieve the
same rate as Proximal GD. Roughly speaking, this is because \cite{pmlr-v235-liu24cg}
analyzes the shuffling gradient methods in a way similar to Proximal
GD. More technically, their analysis views one whole epoch as a single
update step (this also explains why their proximal step happens at
the end of every epoch) and then measures how close the progress made
in every epoch is to Proximal GD (see their Lemmas D.1 and D.3). Afterwards,
they follow the way developed in \cite{zamani2023exact,liu2024revisiting}
to prove the last-iterate convergence rate.

However, one key point missed in \cite{pmlr-v235-liu24cg} is the
randomness. Precisely, their proof for the nonsmooth case goes through
whenever the index $\I(t)$ satisfies $\left\{ \I((k-1)n+1),\cdots,\I(kn))\right\} =\left[n\right],\forall k\in\left[K\right]$
if $T=Kn$ where $K\in\N$ (this also explains why their rates work
for any shuffling scheme not limited to $\RR/\SS/\IG$). As such,
results in \cite{pmlr-v235-liu24cg} cannot reflect any potential
benefit for randomly generated $\I(t)$.

\subsection{Our Analysis Considers Randomness}

Due to the above discussion, a natural thought arises, identifying
the role of randomness for $\RR/\SS$ and trying to integrate it into
the analysis. 

Let us first understand what randomness $\RR$ and $\SS$ have. A
simple but important observation is that for both of them, regardless
of the dependence between indices, every $\I(t)$ has the same distribution
as the random variable uniformly distributed on $\left[n\right]$,
i.e., $\I(t)\diseq\mathrm{Uniform}\left[n\right],\forall t\in\left[T\right]$.
Hence, we can make an abstraction here, i.e., consider any possible
random process $\I(t)$ satisfying $\I(t)\diseq\mathrm{Uniform}\left[n\right],\forall t\in\left[T\right]$
instead of limited to $\RR/\SS$.

Next, we think about how to inject this abstraction into a concrete
analysis. A potential template we can try to follow is the last-iterate
analysis for the standard Proximal SGD method \cite{liu2024revisiting},
where indices $\I(1)$ to $\I(T)$ are exactly mutually independent
random variables uniformly distributed on $\left[n\right]$, which
can be recognized as the extreme case. Moreover, note that Proximal
SGD converges for any $T\in\N$, it is hence a perfect candidate to
help us remove the requirement $T=Kn$ for $K\in\N$ used in prior
works. Unfortunately, proofs in \cite{liu2024revisiting} cannot work
directly due to the potential dependence between every $\I(t)$ in
the general case. Thus, our analysis contains careful changes and
naturally diverges from \cite{liu2024revisiting}.

In brief, our plan is to analyze the last-iterate convergence of Algorithm
\ref{alg:Alg} (for general indices satisfying $\I(t)\diseq\mathrm{Uniform}\left[n\right],\forall t\in\left[T\right]$)
in a manner inspired by Proximal SGD instead of Proximal GD. We clarify
that though some analyses motivated by SGD have also appeared in \cite{pmlr-v97-nagaraj19a,NEURIPS2022_7bc4f74e}
before, there are however many differences. The first and most important
one is that they only focus on shuffling-based methods without further
abstraction on $\I(t)$. Moreover, their goal is to establish convergence
for the average iterate in contrast to the harder last iterate. In
addition, both of them only take $\psi=\charf_{\domx}$ (i.e., constrained
optimization) into account instead of general $\psi$ studied by us.
Even more, \cite{pmlr-v97-nagaraj19a} additionally requires smoothness
on $f_{i}$, \cite{NEURIPS2022_7bc4f74e} only has results for the
general convex case, and both of them assume $G_{i}\equiv G$. As
the reader will see, our core Lemma \ref{lem:last-main} holds for
general $\I(t)$ and gives a novel sufficient condition for the last-iterate
convergence working any $\psi$ under minimal conditions, Assumptions
\ref{assu:basic} and \ref{assu:lip}. Hence, although the high-level
idea may sound similar, our analysis significantly differs from theirs.

\subsection{A New Sufficient Lemma for the Last-Iterate Rate}

Due to limited space, we will not provide formal analysis but only
state the core result, Lemma \ref{lem:last-main}, a new sufficient
lemma giving the last-iterate convergence rate.
\begin{lem}
\label{lem:last-main}Under Assumptions \ref{assu:basic} and \ref{assu:lip},
suppose the following three conditions hold:
\begin{enumerate}
\item \label{enu:last-main-1}The index satisfies $\I(t)\diseq\mathrm{Uniform}\left[n\right],\forall t\in\left[T\right]$.
\item \label{enu:last-main-2}The stepsize $\eta_{t},\forall t\in\left[T\right]$
is non-increasing.
\item \label{enu:last-main-3}$\left|\E\left[\Omega_{t}(\bx_{s})\right]\right|\leq\Phi\eta_{s},\forall t\in\left[T\right],s\in\left[t\right]$
where $\Omega_{t}(\cdot)\defeq f_{\I(t)}(\cdot)-f(\cdot),\forall t\in\left[T\right]$
and $\Phi\geq0$ is a constant probably depending on $T$, $n$, $\mu$,
and $G_{1}$ to $G_{n}$.
\end{enumerate}
Then Algorithm \ref{alg:Alg} guarantees
\begin{align*}
 & \E\left[F(\bx_{T+1})-F_{\star}\right]\\
 & \leq\O\left(\frac{D_{\star}^{2}}{\sum_{t=1}^{T}\gamma_{t}\eta_{t}}+\left(G_{f,2}^{2}+\Phi\right)\sum_{t=1}^{T}\frac{\gamma_{t}\eta_{t}^{2}}{\sum_{s=t}^{T}\gamma_{s}\eta_{s}}\right),
\end{align*}
where $\gamma_{t}\defeq\prod_{s=1}^{t-1}(1+\mu\eta_{s}),\forall t\in\left[T+1\right]$.
\end{lem}

\begin{rem}
We again assume the existence of an optimum $\bx_{\star}\in\R^{d}$
and use the notation $D_{\star}=\left\Vert \bx_{\star}-\bx_{1}\right\Vert $.
See Lemma \ref{lem:last} for the full version valid for any $\bz\in\R^{d}$.
\end{rem}

We first talk about the three conditions in Lemma \ref{lem:last-main}.
Condition \ref{enu:last-main-1} is from the abstraction of the randomness
in $\RR/\SS$ mentioned before. Condition \ref{enu:last-main-2} is
satisfied by almost all common stepsizes. However, we remark that
it is actually not necessary and only for simplicity. Even without
it, the rate still holds with the only change $\eta_{t}^{2}$ to $\eta_{t}(\eta_{t}\lor\eta_{t+1})$.
Condition \ref{enu:last-main-3} is the most important. Intuitively,
it says that even $\I(t)$ and $\bx_{s}$ are possibly dependent
leading to $\E\left[\Omega_{t}(\bx_{s})\right]\neq0$ (equivalently,
$\E\left[f_{\I(t)}(\bx_{s})\right]\neq\E\left[f(\bx_{s})\right]$),
Algorithm \ref{alg:Alg} still guarantees the last-iterate convergence
once $\left|\E\left[\Omega_{t}(\bx_{s})\right]\right|$ can be controlled
by the value of $\eta_{s}$ with another multiplicative factor $\Phi\geq0$.
Notably, the smaller $\Phi$ is, the better the convergence rate is.
As a sanity check, we consider Proximal SGD and notice that $\Phi=0$
in this case. Then Lemma \ref{lem:last-main} recovers the existing
bound in \cite{liu2024revisiting}.

Armed with Lemma \ref{lem:last-main}, to prove a last-iterate convergence
rate for Algorithm \ref{alg:Alg} employing any general random index
$\I(t)$, it is sufficient to find the corresponding constant $\Phi$.
Here we claim that under the same setting of every theorem in Section
\ref{sec:rates}, there is always
\[
\Phi=\begin{cases}
\Theta\left(\sqrt{n}G_{f,1}G_{f,2}\right) & \text{for }\RR\\
\Theta\left(\sqrt{T}G_{f,1}G_{f,2}+\frac{T}{n}G_{f,2}^{2}\right) & \text{for }\SS
\end{cases}.
\]
With these two values of $\Phi$ and the inequality (\ref{eq:G}),
i.e., $G_{f,1}\leq G_{f,2}<\sqrt{n}G_{f,1}$, we immediately conclude
Theorems \ref{thm:RR-cvx}, \ref{thm:RR-str}, \ref{thm:SS-cvx},
and \ref{thm:SS-str} after plugging in the stepsize. However, we
remark that finding these two values is not a trivial task for $\RR/\SS$.
Though some clues on how to bound $\left|\E\left[\Omega_{t}(\bx_{t})\right]\right|$
can be found in prior works \cite{pmlr-v97-nagaraj19a,NEURIPS2022_7bc4f74e},
our goal is more general and hence harder since we need to bound $\left|\E\left[\Omega_{t}(\bx_{s})\right]\right|$
by a time-dependent stepsize $\eta_{s}$ for any $s\leq t$. More
importantly, our fine-grained analysis on $G_{f,1}$ and $G_{f,2}$
is the key to establishing the superiority of our convergence results
over \cite{pmlr-v235-liu24cg}, which cannot be found in \cite{pmlr-v97-nagaraj19a,NEURIPS2022_7bc4f74e}
due to their simpler settings as mentioned earlier.

Lastly, we briefly discuss Theorem \ref{thm:SS-cvx-improved}. Due
to the dependence on $T$ in $\Phi$ for $\SS$, the above analysis
is inadequate. However, as hinted by \cite{pmlr-v235-liu24cg} (though
algorithms are different), one should expect that $\SS$ at least
provably converges as fast as Proximal GD when $T=Kn$ for $K\in\N$.
As such, we will combine some other techniques to obtain Theorem \ref{thm:SS-cvx-improved}
and not discuss them here.

\section{Conclusion and Future Work}

We provide improved last-iterate convergence rates for nonsmooth (strongly)
convex optimization under both  $\RR$ and $\SS$ sampling schemes,
showing they are faster than Proximal GD (conditionally for $\SS$)
and enjoy the benefit of randomness. Two valued problems are left
to be addressed in the future. One is to explore whether our rates
for $\RR$ are tight. The other is to prove better bounds for $\SS$
if possible.

\clearpage

\section*{Acknowledgments}

This work is supported by the NSF grant ECCS-2419564. We also thank
the anonymous reviewers for their valuable comments.

\section*{Impact Statement}

This paper presents work whose goal is to advance the field of Machine
Learning. There are many potential societal consequences of our work,
none which we feel must be specifically highlighted here.

\bibliographystyle{icml2025}
\bibliography{ref}

\appendix
\onecolumn

\section{Full Theorems\label{sec:full-thm}}

In this section, we provide the full version of each theorem presented
in Section \ref{sec:rates} and the corresponding proof. The proofs
of lemmas used in the analysis are deferred to Sections \ref{sec:analysis}
and \ref{sec:auxiliary} later. One point here we want to remind the
reader is that our new last-iterate results (except Theorem \ref{thm:SS-cvx-improved-full})
hold for any $T\in\N$ instead of only $T=Kn$ for $K\in\N$ used
in most of the existing works for shuffling gradient methods mentioned
in Subsection \ref{subsec:related}.

\subsection{$\protect\RR$ Sampling Scheme\label{subsec:RR-full-thm}}

We first consider the general convex case in Theorem \ref{thm:RR-cvx-full}.
One notable thing is that, under the first stepsize choice, the extra
logarithmic factor can be shaved off once $\qu(T)=\Omega(\log n)\Leftrightarrow T=\Omega(n\log n)$
as previously mentioned in Theorem \ref{thm:RR-cvx}. Finding a stepsize
that can remove any extra logarithmic factor without requiring $T=\Omega(n\log n)$
will be an interesting task.
\begin{thm}
\label{thm:RR-cvx-full}(Full version of Theorem \ref{thm:RR-cvx})
Under Assumptions \ref{assu:basic} (with $\mu=0$) and \ref{assu:lip},
suppose the $\RR$ sampling scheme is employed:
\begin{itemize}
\item Taking the stepsize $\eta_{t}=\eta\frac{\qu(T)-\qu(t)+1}{\qu(T)\sqrt{T}},\forall t\in\left[T\right]$,
then for any $\bz\in\R^{d}$, Algorithm \ref{alg:Alg} guarantees
\[
\E\left[F(\bx_{T+1})-F(\bz)\right]\leq\O\left(\left(\frac{\left\Vert \bz-\bx_{1}\right\Vert ^{2}}{\eta}+\eta\sqrt{n}G_{f,1}G_{f,2}\left(1+\frac{\log n}{\qu(T)}\right)\right)\frac{1}{\sqrt{T}}\right).
\]
Setting $\eta=\frac{\left\Vert \bz-\bx_{1}\right\Vert }{n^{1/4}\sqrt{G_{f,1}G_{f,2}\left(1+\frac{\log n}{\qu(T)}\right)}}$
to optimize the dependence on parameters.
\item Taking the stepsize $\eta_{t}=\frac{\eta}{\sqrt{T}},\forall t\in\left[T\right]$,
then for any $\bz\in\R^{d}$, Algorithm \ref{alg:Alg} guarantees
\[
\E\left[F(\bx_{T+1})-F(\bz)\right]\leq\O\left(\left(\frac{\left\Vert \bz-\bx_{1}\right\Vert ^{2}}{\eta}+\eta\sqrt{n}G_{f,1}G_{f,2}(1+\log T)\right)\frac{1}{\sqrt{T}}\right).
\]
Setting $\eta=\frac{\left\Vert \bz-\bx_{1}\right\Vert }{n^{1/4}\sqrt{G_{f,1}G_{f,2}(1+\log T)}}$
to optimize the dependence on parameters.
\item Taking the stepsize $\eta_{t}=\frac{\eta}{\sqrt{t}},\forall t\in\left[T\right]$,
then for any $\bz\in\R^{d}$, Algorithm \ref{alg:Alg} guarantees
\[
\E\left[F(\bx_{T+1})-F(\bz)\right]\leq\O\left(\left(\frac{\left\Vert \bz-\bx_{1}\right\Vert ^{2}}{\eta}+\eta\sqrt{n}G_{f,1}G_{f,2}(1+\log T)\right)\frac{1}{\sqrt{T}}\right).
\]
Setting $\eta=\frac{\left\Vert \bz-\bx_{1}\right\Vert }{n^{1/4}\sqrt{G_{f,1}G_{f,2}}}$
to optimize the dependence on parameters.
\end{itemize}
\end{thm}

\begin{proof}
Note that the $\RR$ sampling scheme satisfies $\I(t)\diseq\mathrm{Uniform}\left[n\right],\forall t\in\left[T\right]$,
and all these stepsizes listed are non-increasing. Hence, Conditions
\ref{enu:last-condition-1} and \ref{enu:last-condition-2} in Lemma
\ref{lem:last} are fulfilled. If Condition \ref{enu:last-condition-3}
also holds, i.e., $\left|\E\left[\Omega_{t}(\bx_{s})\right]\right|\leq\Phi\eta_{s},\forall t\in\left[T\right],s\in\left[t\right]$
for some $\Phi\geq0$, we can invoke Lemma \ref{lem:last} to obtain
for any $\bz\in\R^{d}$,
\begin{align}
\E\left[F(\bx_{T+1})-F(\bz)\right] & \leq\O\left(\frac{\left\Vert \bz-\bx_{1}\right\Vert ^{2}}{\sum_{t=1}^{T}\gamma_{t}\eta_{t}}+\left(G_{f,2}^{2}+\Phi\right)\sum_{t=1}^{T}\frac{\gamma_{t}\eta_{t}^{2}}{\sum_{s=t}^{T}\gamma_{s}\eta_{s}}\right)\nonumber \\
 & =\O\left(\frac{\left\Vert \bz-\bx_{1}\right\Vert ^{2}}{\sum_{t=1}^{T}\eta_{t}}+\left(G_{f,2}^{2}+\Phi\right)\sum_{t=1}^{T}\frac{\eta_{t}^{2}}{\sum_{s=t}^{T}\eta_{s}}\right),\label{eq:RR-cvx-full-1}
\end{align}
where the second line is by $\gamma_{t}=\prod_{s=1}^{t-1}(1+\mu\eta_{s})=1,\forall t\in\left[T+1\right]$
since $\mu=0$ now.

Our next task is to find a constant $\Phi\geq0$ satisfying
\[
\left|\E\left[\Omega_{t}(\bx_{s})\right]\right|\leq\Phi\eta_{s},\forall t\in\left[T\right],s\in\left[t\right].
\]
For all stepsizes listed, we claim
\begin{equation}
\Phi=4\left(G_{f,2}^{2}+2\sqrt{n}G_{f,1}G_{f,2}\right).\label{eq:RR-cvx-full-Psi}
\end{equation}
To prove our statement, we apply Lemma \ref{lem:RR-Omega} with $\mu=0$
to have:
\begin{itemize}
\item If $s\in\left[(\qu(t)-1)n\right]$, there is
\[
\E\left[\Omega_{t}(\bx_{s})\right]=0\Rightarrow\left|\E\left[\Omega_{t}(\bx_{s})\right]\right|=0\leq\Phi\eta_{s}.
\]
\item If $s\in\left[t\right]\backslash\left[(\qu(t)-1)n\right]$ (which
implies $\qu(s)=\qu(t)$), there is
\begin{align}
\left|\E\left[\Omega_{t}(\bx_{s})\right]\right| & \leq\frac{\sqrt{2}G_{f,2}^{2}}{n}\sum_{j=(\qu(t)-1)n+1}^{s-1}\eta_{j}+\frac{2\sqrt{2}G_{f,1}G_{f,2}}{n}\sum_{i=(\qu(t)-1)n+1}^{s-1}\sqrt{\sum_{j=i}^{s-1}\eta_{j}^{2}}.\nonumber \\
 & \overset{(a)}{\leq}\frac{\sqrt{2}G_{f,2}^{2}}{n}\sum_{j=(\qu(t)-1)n+1}^{s-1}\eta_{j}+\frac{2\sqrt{2}G_{f,1}G_{f,2}}{n}\sum_{i=(\qu(t)-1)n+1}^{s-1}\sqrt{s-i}\eta_{i}\nonumber \\
 & \overset{(b)}{\leq}\frac{\sqrt{2}\left(G_{f,2}^{2}+2\sqrt{n}G_{f,1}G_{f,2}\right)}{n}\sum_{j=(\qu(t)-1)n+1}^{s-1}\eta_{j},\label{eq:RR-cvx-full-Omega}
\end{align}
where $(a)$ is because $\eta_{j}$ is non-increasing and $(b)$ is
due to $\sqrt{s-i}\leq\sqrt{s-(\qu(t)-1)n-1}=\sqrt{\re(s)-1}\leq\sqrt{n}$
when $i\in\left\{ (\qu(t)-1)n+1,\cdots,s-1\right\} $. Note that if
$s=(\qu(t)-1)n+1$, $\left|\E\left[\Omega_{t}(\bx_{s})\right]\right|=0\leq\Phi\eta_{s}$.
Hence, we assume $(\qu(t)-1)n+2\leq s\leq t$ in the following.
\begin{itemize}
\item For the stepsize satisfying $\eta_{t}=\eta_{(\qu(t)-1)n+1},\forall t\in\left[T\right]$
(the first two), we observe that $(\qu(t)-1)n+1\leq j\leq s-1\Rightarrow\eta_{j}=\eta_{(\qu(t)-1)n+1}=\eta_{s}$.
Therefore,
\begin{equation}
\sum_{j=(\qu(t)-1)n+1}^{s-1}\eta_{j}=(\re(s)-1)\eta_{s}\leq n\eta_{s}.\label{eq:RR-cvx-full-2}
\end{equation}
Thus, there is
\[
\left|\E\left[\Omega_{t}(\bx_{s})\right]\right|\overset{\eqref{eq:RR-cvx-full-Omega},\eqref{eq:RR-cvx-full-2}}{\leq}\sqrt{2}\left(G_{f,2}^{2}+2\sqrt{n}G_{f,1}G_{f,2}\right)\eta_{s}\overset{\eqref{eq:RR-cvx-full-Psi}}{\leq}\Phi\eta_{s}.
\]
\item For the stepsize $\eta_{t}=\frac{\eta}{\sqrt{t}},\forall t\in\left[T\right]$,
we know
\begin{align}
\sum_{j=(\qu(t)-1)n+1}^{s-1}\eta_{j} & =\eta\sum_{j=(\qu(t)-1)n+1}^{s-1}\frac{1}{\sqrt{j}}\leq\eta\int_{(\qu(t)-1)n}^{s-1}\frac{1}{\sqrt{j}}\mathrm{d}j=2\eta(\sqrt{s-1}-\sqrt{(\qu(t)-1)n})\nonumber \\
 & =\frac{2\eta(\re(s)-1)}{\sqrt{s-1}+\sqrt{(\qu(t)-1)n}}=\frac{2(\re(s)-1)\sqrt{s}}{\sqrt{s-1}+\sqrt{(\qu(t)-1)n}}\eta_{s}\leq2\sqrt{2}n\eta_{s},\label{eq:RR-cvx-full-3}
\end{align}
where the last step is by $\frac{2(\re(s)-1)\sqrt{s}}{\sqrt{s-1}+\sqrt{(\qu(t)-1)n}}\leq2\sqrt{\frac{s}{s-1}}n\leq2\sqrt{2}n$
when $s\geq(\qu(t)-1)n+2\geq2$. Thus, there is
\[
\left|\E\left[\Omega_{t}(\bx_{s})\right]\right|\overset{\eqref{eq:RR-cvx-full-Omega},\eqref{eq:RR-cvx-full-3}}{\leq}4\left(G_{f,2}^{2}+2\sqrt{n}G_{f,1}G_{f,2}\right)\eta_{s}\overset{\eqref{eq:RR-cvx-full-Psi}}{=}\Phi\eta_{s}.
\]
\end{itemize}
\end{itemize}
By (\ref{eq:RR-cvx-full-1}) and (\ref{eq:RR-cvx-full-Psi}), we have
for all these stepsizes
\begin{align*}
\E\left[F(\bx_{T+1})-F(\bz)\right] & \leq\O\left(\frac{\left\Vert \bz-\bx_{1}\right\Vert ^{2}}{\sum_{t=1}^{T}\eta_{t}}+\left(G_{f,2}^{2}+\sqrt{n}G_{f,1}G_{f,2}\right)\sum_{t=1}^{T}\frac{\eta_{t}^{2}}{\sum_{s=t}^{T}\eta_{s}}\right)\\
 & \leq\O\left(\frac{\left\Vert \bz-\bx_{1}\right\Vert ^{2}}{\sum_{t=1}^{T}\eta_{t}}+\sqrt{n}G_{f,1}G_{f,2}\sum_{t=1}^{T}\frac{\eta_{t}^{2}}{\sum_{s=t}^{T}\eta_{s}}\right),
\end{align*}
where the last step is by noticing $G_{f,2}<\sqrt{n}G_{f,1}$.
\begin{itemize}
\item For the stepsize $\eta_{t}=\eta\frac{\qu(T)-\qu(t)+1}{\qu(T)\sqrt{T}},\forall t\in\left[T\right]$,
by applying Lemma \ref{lem:stepsize} with $\eta_{\star}=\frac{\eta}{\qu(T)\sqrt{T}}$,
we know $\sum_{t=1}^{T}\eta_{t}\geq\frac{\eta\sqrt{T}}{2}$ and $\sum_{t=1}^{T}\frac{\eta_{t}^{2}}{\sum_{s=t}^{T}\eta_{s}}\leq\frac{9\eta}{2\sqrt{T}}\left(1+\frac{\log n}{\qu(T)}\right)$,
which implies
\[
\E\left[F(\bx_{T+1})-F(\bz)\right]\leq\O\left(\left(\frac{\left\Vert \bz-\bx_{1}\right\Vert ^{2}}{\eta}+\eta\sqrt{n}G_{f,1}G_{f,2}\left(1+\frac{\log n}{\qu(T)}\right)\right)\frac{1}{\sqrt{T}}\right).
\]
Setting $\eta=\frac{\left\Vert \bz-\bx_{1}\right\Vert }{n^{1/4}\sqrt{G_{f,1}G_{f,2}\left(1+\frac{\log n}{\qu(T)}\right)}}$
to obtain
\[
\E\left[F(\bx_{T+1})-F(\bz)\right]\leq\O\left(\frac{n^{1/4}\sqrt{G_{f,1}G_{f,2}}\left\Vert \bz-\bx_{1}\right\Vert \sqrt{1+\frac{\log n}{\qu(T)}}}{\sqrt{T}}\right).
\]
Particularly, if $\qu(T)=\Omega(\log n)\Leftrightarrow T=\Omega(n\log n)$,
there is
\[
\E\left[F(\bx_{T+1})-F(\bz)\right]\leq\O\left(\frac{n^{1/4}\sqrt{G_{f,1}G_{f,2}}\left\Vert \bz-\bx_{1}\right\Vert }{\sqrt{T}}\right).
\]
\item For the stepsize $\eta_{t}=\frac{\eta}{\sqrt{T}},\forall t\in\left[T\right]$,
we know
\[
\sum_{t=1}^{T}\frac{\eta_{t}^{2}}{\sum_{s=t}^{T}\eta_{s}}\leq\frac{\eta}{\sqrt{T}}\sum_{t=1}^{T}\frac{1}{T-t+1}\leq\frac{\eta(1+\log T)}{\sqrt{T}}.
\]
Hence, we have
\[
\E\left[F(\bx_{T+1})-F(\bz)\right]\leq\O\left(\left(\frac{\left\Vert \bz-\bx_{1}\right\Vert ^{2}}{\eta}+\eta\sqrt{n}G_{f,1}G_{f,2}(1+\log T)\right)\frac{1}{\sqrt{T}}\right).
\]
Setting $\eta=\frac{\left\Vert \bz-\bx_{1}\right\Vert }{n^{1/4}\sqrt{G_{f,1}G_{f,2}(1+\log T)}}$
to obtain
\[
\E\left[F(\bx_{T+1})-F(\bz)\right]\leq\O\left(\frac{n^{1/4}\sqrt{G_{f,1}G_{f,2}}\left\Vert \bz-\bx_{1}\right\Vert \sqrt{1+\log T}}{\sqrt{T}}\right).
\]
\item For the stepsize $\eta_{t}=\frac{\eta}{\sqrt{t}},\forall t\in\left[T\right]$,
we know for any $t\in\left[T\right]$,
\[
\sum_{s=t}^{T}\eta_{s}=\eta\sum_{s=t}^{T}\frac{1}{\sqrt{s}}\geq\eta\int_{t}^{T+1}\frac{1}{\sqrt{s}}\mathrm{d}s=2\eta(\sqrt{T+1}-\sqrt{t}),
\]
which implies
\begin{align*}
\sum_{t=1}^{T}\frac{\eta_{t}^{2}}{\sum_{s=t}^{T}\eta_{s}} & \leq\eta\sum_{t=1}^{T}\frac{1}{t(\sqrt{T+1}-\sqrt{t})}=\eta\sum_{t=1}^{T}\frac{\sqrt{T+1}+\sqrt{t}}{t(T+1-t)}\leq2\eta\sum_{t=1}^{T}\frac{\sqrt{T+1}}{t(T+1-t)}\\
 & =\frac{2\eta}{\sqrt{T+1}}\sum_{t=1}^{T}\frac{1}{t}+\frac{1}{T+1-t}\leq\frac{4\eta(1+\log T)}{\sqrt{T+1}}.
\end{align*}
Hence, we have
\[
\E\left[F(\bx_{T+1})-F(\bz)\right]\leq\O\left(\left(\frac{\left\Vert \bz-\bx_{1}\right\Vert ^{2}}{\eta}+\eta\sqrt{n}G_{f,1}G_{f,2}(1+\log T)\right)\frac{1}{\sqrt{T}}\right).
\]
Setting $\eta=\frac{\left\Vert \bz-\bx_{1}\right\Vert }{n^{1/4}\sqrt{G_{f,1}G_{f,2}}}$
to obtain
\[
\E\left[F(\bx_{T+1})-F(\bz)\right]\leq\O\left(\frac{n^{1/4}\sqrt{G_{f,1}G_{f,2}}\left\Vert \bz-\bx_{1}\right\Vert (1+\log T)}{\sqrt{T}}\right).
\]
\end{itemize}
\end{proof}

Built upon Theorem \ref{thm:RR-cvx-full}, we give the following rate
for $\bx_{T+1}^{\suf}=\frac{1}{n}\sum_{t=T-n+1}^{T}\bx_{t+1}$.
\begin{cor}
\label{cor:RR-suffix-full}(Full version of Corollary \ref{cor:RR-suffix})
Under Assumptions \ref{assu:basic} (with $\mu=0$) and \ref{assu:lip},
suppose $T\geq n$ and the $\RR$ sampling scheme is employed:
\begin{itemize}
\item Taking the stepsize $\eta_{t}=\eta,\forall t\in\left[T\right]$, then
for any $\bz\in\R^{d}$, Algorithm \ref{alg:Alg} guarantees (additionally
assuming $T\geq2(n-1)$\footnote{This requirement can be further relaxed to $T\geq(1+c)(n-1)$ for
any $c>0$. We simply choose $c=1$ here.})
\[
\E\left[F(\bx_{T+1}^{\suf})-F(\bz)\right]\leq\O\left(\frac{\left\Vert \bz-\bx_{1}\right\Vert ^{2}}{\eta T}+\eta\sqrt{n}G_{f,1}G_{f,2}(1+\log T)\right).
\]
\item Taking the stepsize $\eta_{t}=\frac{\eta}{\sqrt{t}},\forall t\in\left[T\right]$,
then for any $\bz\in\R^{d}$, Algorithm \ref{alg:Alg} guarantees
\[
\E\left[F(\bx_{T+1}^{\suf})-F(\bz)\right]\leq\O\left(\left(\frac{\left\Vert \bz-\bx_{1}\right\Vert ^{2}}{\eta}+\eta\sqrt{n}G_{f,1}G_{f,2}(1+\log T)\right)\frac{1}{\sqrt{T}}\right).
\]
\end{itemize}
\end{cor}

\begin{proof}
Both results are directly implied by Theorem \ref{thm:RR-cvx-full}.
\begin{itemize}
\item For the stepsize $\eta_{t}=\eta,\forall t\in\left[T\right]$, we invoke
the second result in Theorem \ref{thm:RR-cvx-full} (under changing
$\eta$ to $\eta\sqrt{T}$ in it) to obtain
\[
\E\left[F(\bx_{T+1})-F(\bz)\right]\leq\O\left(\frac{\left\Vert \bz-\bx_{1}\right\Vert ^{2}}{\eta T}+\eta\sqrt{n}G_{f,1}G_{f,2}(1+\log T)\right).
\]
Note that the above result actually holds for any $T\in\N$. Hence,
there is by convexity
\[
\E\left[F(\bx_{T+1}^{\suf})-F(\bz)\right]\leq\frac{1}{n}\sum_{t=T-n+1}^{T}\E\left[F(\bx_{t+1})-F(\bz)\right]\leq\frac{1}{n}\sum_{t=T-n+1}^{T}\O\left(\frac{\left\Vert \bz-\bx_{1}\right\Vert ^{2}}{\eta t}+\eta\sqrt{n}G_{f,1}G_{f,2}(1+\log t)\right).
\]
Moreover, we have
\begin{eqnarray*}
\frac{1}{n}\sum_{t=T-n+1}^{T}\frac{1}{t}\leq\frac{1}{T-n+1}\overset{(a)}{\leq}\frac{2}{T} & \text{and} & \frac{1}{n}\sum_{t=T-n+1}^{T}1+\log t\leq1+\log T,
\end{eqnarray*}
where we use $T\geq2(n-1)$ in $(a)$. Finally, we obtain
\[
\E\left[F(\bx_{T+1}^{\suf})-F(\bz)\right]\leq\O\left(\frac{\left\Vert \bz-\bx_{1}\right\Vert ^{2}}{\eta T}+\eta\sqrt{n}G_{f,1}G_{f,2}(1+\log T)\right).
\]
\item For the stepsize $\eta_{t}=\frac{\eta}{\sqrt{t}},\forall t\in\left[T\right]$,
we invoke the third result in Theorem \ref{thm:RR-cvx-full} and follow
almost the same steps for proving Corollary \ref{cor:RR-suffix} to
conclude.
\end{itemize}
\end{proof}

We emphasize two important implications of Corollary \ref{cor:RR-suffix-full}.

The first one is that under the same setting in the proof of lower
bound \cite{NEURIPS2022_7bc4f74e}, i.e., $\psi=\charf_{\domx}$ where
$\domx$ is the unit ball, $G_{i}\equiv4$, $T=Kn$, $\bx_{1}=\bzero$,
the existence of $\bx_{\star}\in\argmin_{\bx\in\R^{d}}F(\bx)$, and
$\eta_{t}\equiv\eta$, we have
\[
\E\left[F(\bx_{Kn+1}^{\suf})-F_{\star}\right]\leq\O\left(\frac{1}{\eta nK}+\eta\sqrt{n}(1+\log nK)\right),
\]
where $F_{\star}=F(\bx_{\star})$. In addition, notice that now
\[
F(\bx_{Kn+1}^{\suf})-F_{\star}=f(\bx_{Kn+1}^{\suf})-f(\bx_{\star})\leq G_{f,1}\left\Vert \bx_{Kn+1}^{\suf}-\bx_{\star}\right\Vert \leq8=\O(1).
\]
We hence obtain
\begin{equation}
\E\left[F(\bx_{Kn+1}^{\suf})-F_{\star}\right]\leq\O\left(\min\left\{ 1,\frac{1}{\eta nK}+\eta\sqrt{n}(1+\log nK)\right\} \right).\label{eq:suffix-exact}
\end{equation}
Remarkably, this rate matches the lower bound $\Omega\left(\min\left\{ 1,\frac{1}{\eta nK}+\eta\sqrt{n}\right\} \right)$
of the suffix average for the last one epoch exactly by up to an extra
logarithmic factor. Therefore, we (almost) close the gap. 

The second one is related to the discussion in Footnote \ref{fn:SGD},
which is that we suspect the rate $\O\left(\frac{n^{1/4}}{\sqrt{T}}\right)$
(or $\O\left(\frac{1}{n^{1/4}\sqrt{K}}\right)$ when $T=Kn$ where
$K\in\N$) is tight for the last iterate under the $\RR$ sampling
scheme in the general convex case. By the first implication above,
one can see our last-iterate bound $\widetilde{\O}\left(\frac{1}{\eta T}+\eta\sqrt{n}\right)$
for the constant stepsize $\eta_{t}\equiv\eta$ (we only include $\eta$,
$n$, and $T$ in the rate for simplicity) is almost tight by up to
extra logarithmic factors. This is because if one can establish the
following bound
\[
\E\left[F(\bx_{T+1})-F_{\star}\right]\leq o\left(\frac{1}{\eta T}+\eta\sqrt{n}\right),
\]
then, using the same proof of (\ref{eq:suffix-exact}), there will
be 
\[
\E\left[F(\bx_{T+1}^{\suf})-F_{\star}\right]\leq o\left(\min\left\{ 1,\frac{1}{\eta T}+\eta\sqrt{n}\right\} \right),
\]
which contradicts the lower bound $\Omega\left(\min\left\{ 1,\frac{1}{\eta nK}+\eta\sqrt{n}\right\} \right)$
in \cite{NEURIPS2022_7bc4f74e}. Hence, at least in the case $\eta_{t}\equiv\eta$,
our last-iterate result $\widetilde{\O}\left(\frac{1}{\eta T}+\eta\sqrt{n}\right)$
should be tight, which immediately implies that $\widetilde{\O}\left(\frac{n^{1/4}}{\sqrt{T}}\right)$
is also tight by picking the optimal $\eta$. As such, we conjecture
that the rate $\O\left(\frac{n^{1/4}}{\sqrt{T}}\right)$ is tight
for $\RR$ in the general convex case though we cannot prove it when
the stepsize is not constant.

Next, Theorem \ref{thm:RR-str-full} shows the convergence rate when
$\psi$ is $\mu$-strongly convex, e.g., the common regularizer $\psi(\bx)=\frac{\mu}{2}\left\Vert \bx\right\Vert ^{2}$.
\begin{thm}
\label{thm:RR-str-full}(Full version of Theorem \ref{thm:RR-str})
Under Assumptions \ref{assu:basic} (with $\mu>0$) and \ref{assu:lip},
suppose the $\RR$ sampling scheme is employed with the stepsize $\eta_{t}=\frac{m}{\mu t},\forall t\in\left[T\right]$
where $m\in\N$, then for any $\bz\in\R^{d}$, Algorithm \ref{alg:Alg}
guarantees
\[
\E\left[F(\bx_{T+1})-F(\bz)\right]\leq\O\left(\frac{\mu\left\Vert \bz-\bx_{1}\right\Vert ^{2}}{\binom{m+T}{m}}+\frac{m\sqrt{n}G_{f,1}G_{f,2}(1+\log T)}{\mu T}\right).
\]
\end{thm}

\begin{proof}
Note that the $\RR$ sampling scheme satisfies $\I(t)\diseq\mathrm{Uniform}\left[n\right],\forall t\in\left[T\right]$,
and the stepsize $\eta_{t}=\frac{m}{\mu t}$ is non-increasing. Hence,
Conditions \ref{enu:last-condition-1} and \ref{enu:last-condition-2}
in Lemma \ref{lem:last} are fulfilled. If Condition \ref{enu:last-condition-3}
also holds, i.e., $\left|\E\left[\Omega_{t}(\bx_{s})\right]\right|\leq\Phi\eta_{s},\forall t\in\left[T\right],s\in\left[t\right]$
for some $\Phi\geq0$, we can invoke Lemma \ref{lem:last} to obtain
for any $\bz\in\R^{d}$,
\begin{equation}
\E\left[F(\bx_{T+1})-F(\bz)\right]\leq\O\left(\frac{\left\Vert \bz-\bx_{1}\right\Vert ^{2}}{\sum_{t=1}^{T}\gamma_{t}\eta_{t}}+\left(G_{f,2}^{2}+\Phi\right)\sum_{t=1}^{T}\frac{\gamma_{t}\eta_{t}^{2}}{\sum_{s=t}^{T}\gamma_{s}\eta_{s}}\right),\label{eq:RR-str-full-1}
\end{equation}
where $\gamma_{t}=\prod_{s=1}^{t-1}(1+\mu\eta_{s}),\forall t\in\left[T+1\right]$.

When $\eta_{t}=\frac{m}{\mu t},\forall t\in\left[T\right]$, there
is
\[
\gamma_{t}=\prod_{s=1}^{t-1}\frac{s+m}{s}=\frac{(m+t-1)!}{m!(t-1)!}=\binom{m+t-1}{m},\forall t\in\left[T+1\right],
\]
which implies
\begin{equation}
\gamma_{t}\eta_{t}=\frac{(m+t-1)!}{m!(t-1)!}\cdot\frac{m}{\mu t}=\frac{1}{\mu}\binom{m+t-1}{m-1},\forall t\in\left[T\right].\label{eq:RR-str-full-gamma-eta}
\end{equation}
Therefore, we know
\begin{equation}
\sum_{t=1}^{T}\gamma_{t}\eta_{t}=\frac{1}{\mu}\sum_{t=1}^{T}\binom{m+t-1}{m-1}=\frac{1}{\mu}\sum_{t=1}^{T}\binom{m+t}{m}-\binom{m+t-1}{m}=\frac{1}{\mu}\left[\binom{m+T}{m}-1\right],\label{eq:RR-str-full-2}
\end{equation}
and
\begin{equation}
\sum_{t=1}^{T}\frac{\gamma_{t}\eta_{t}^{2}}{\sum_{s=t}^{T}\gamma_{s}\eta_{s}}\overset{(a)}{\leq}\sum_{t=1}^{T}\frac{\eta_{t}}{T-t+1}=\sum_{t=1}^{T}\frac{m}{\mu(T-t+1)t}=\frac{m}{\mu(T+1)}\sum_{t=1}^{T}\frac{1}{t}+\frac{1}{T+1-t}\leq\frac{2m(1+\log T)}{\mu(T+1)},\label{eq:RR-str-full-3}
\end{equation}
where $(a)$ is because $\gamma_{t}\eta_{t}$ is non-decreasing in
$t$ as shown in (\ref{eq:RR-str-full-gamma-eta}). We combine (\ref{eq:RR-str-full-1}),
(\ref{eq:RR-str-full-2}) and (\ref{eq:RR-str-full-3}) to have
\begin{equation}
\E\left[F(\bx_{T+1})-F(\bz)\right]\leq\O\left(\frac{\mu\left\Vert \bz-\bx_{1}\right\Vert ^{2}}{\binom{m+T}{m}}+\frac{m\left(G_{f,2}^{2}+\Phi\right)(1+\log T)}{\mu T}\right).\label{eq:RR-str-full-4}
\end{equation}

Our last task is to find a constant $\Phi\geq0$ satisfying
\[
\left|\E\left[\Omega_{t}(\bx_{s})\right]\right|\leq\Phi\eta_{s},\forall t\in\left[T\right],s\in\left[t\right].
\]
Here we claim
\begin{equation}
\Phi=\sqrt{2}G_{f,2}^{2}+2\sqrt{2n}G_{f,1}G_{f,2}.\label{eq:RR-str-full-Psi}
\end{equation}
To prove our statement, we use Lemma \ref{lem:RR-Omega} to have:
\begin{itemize}
\item If $s\in\left[(\qu(t)-1)n\right]$, there is
\[
\E\left[\Omega_{t}(\bx_{s})\right]=0\Rightarrow\left|\E\left[\Omega_{t}(\bx_{s})\right]\right|=0\leq\Phi\eta_{s}.
\]
\item If $s\in\left[t\right]\backslash\left[(\qu(t)-1)n\right]$, there
is
\begin{align*}
\left|\E\left[\Omega_{t}(\bx_{s})\right]\right| & \leq\frac{\sqrt{2}G_{f,2}^{2}}{n}\sum_{j=(\qu(t)-1)n+1}^{s-1}\frac{\gamma_{j}\eta_{j}}{\gamma_{s}}+\frac{2\sqrt{2}G_{f,1}G_{f,2}}{n}\sum_{i=(\qu(t)-1)n+1}^{s-1}\sqrt{\sum_{j=i}^{s-1}\frac{\gamma_{j}^{2}\eta_{j}^{2}}{\gamma_{s}^{2}}}\\
 & \overset{(b)}{\leq}\frac{\sqrt{2}G_{f,2}^{2}}{n}\sum_{j=(\qu(t)-1)n+1}^{s-1}\eta_{s}+\frac{2\sqrt{2}G_{f,1}G_{f,2}}{n}\sum_{i=(\qu(t)-1)n+1}^{s-1}\sqrt{\sum_{j=i}^{s-1}\eta_{s}^{2}}\\
 & \overset{(c)}{\leq}\left(\sqrt{2}G_{f,2}^{2}+2\sqrt{2n}G_{f,1}G_{f,2}\right)\eta_{s}\overset{\eqref{eq:RR-str-full-Psi}}{=}\Phi\eta_{s},
\end{align*}
where $(b)$ holds by $\frac{\gamma_{j}\eta_{j}}{\gamma_{s}}\leq\eta_{s},\forall j\in\left[s-1\right]$
since $\gamma_{t}\eta_{t}$ is non-decreasing as shown in (\ref{eq:RR-str-full-gamma-eta})
and $(c)$ is due to
\begin{align*}
\sum_{j=(\qu(t)-1)n+1}^{s-1}\eta_{s} & =(\re(s)-1)\eta_{s}\leq n\eta_{s},\\
\sum_{i=(\qu(t)-1)n+1}^{s-1}\sqrt{\sum_{j=i}^{s-1}\eta_{s}^{2}} & =\sum_{i=(\qu(t)-1)n+1}^{s-1}\sqrt{s-i}\eta_{s}\leq n^{\frac{3}{2}}\eta_{s}.
\end{align*}
\end{itemize}
By (\ref{eq:RR-str-full-4}) and (\ref{eq:RR-str-full-Psi}), we finally
have 
\begin{align*}
\E\left[F(\bx_{T+1})-F(\bz)\right] & \leq\O\left(\frac{\mu\left\Vert \bz-\bx_{1}\right\Vert ^{2}}{\binom{m+T}{m}}+\frac{m\left(G_{f,2}^{2}+\sqrt{n}G_{f,1}G_{f,2}\right)(1+\log T)}{\mu T}\right)\\
 & \leq\O\left(\frac{\mu\left\Vert \bz-\bx_{1}\right\Vert ^{2}}{\binom{m+T}{m}}+\frac{m\sqrt{n}G_{f,1}G_{f,2}(1+\log T)}{\mu T}\right),
\end{align*}
where the last step is by using $G_{f,2}<\sqrt{n}G_{f,1}$.
\end{proof}

\subsection{$\protect\SS$ Sampling Scheme\label{subsec:SS-full-thm}}

First, Theorem \ref{thm:SS-cvx-full} below gives the convergence
guarantee for general convex functions (i.e., $\mu=0$). Unlike the
previous Theorem \ref{thm:RR-cvx-full}, we currently do not know
how to design a proper stepsize to get rid of the extra logarithmic
factor, which we leave as a future direction. Another crucial fact
is that we also have no idea how to set a time-varying stepsize. Loosely
speaking, this is because $\Phi$ depends on $T$ now. See our analysis
for details.
\begin{thm}
\label{thm:SS-cvx-full}(Full version of Theorem \ref{thm:SS-cvx})
Under Assumptions \ref{assu:basic} and (with $\mu=0$) and \ref{assu:lip},
suppose the $\SS$ sampling scheme is employed with the stepsize $\eta_{t}=\frac{\eta}{\sqrt{T}},\forall t\in\left[T\right]$,
then for any $\bz\in\R^{d}$, Algorithm \ref{alg:Alg} guarantees
\[
\E\left[F(\bx_{T+1})-F(\bz)\right]\leq\O\left(\left(\frac{\left\Vert \bz-\bx_{1}\right\Vert ^{2}}{\eta}+\eta\left(\qu(T)G_{f,2}^{2}\lor\sqrt{n\qu(T)}G_{f,1}G_{f,2}\right)(1+\log T)\right)\frac{1}{\sqrt{T}}\right).
\]
Setting $\eta=\frac{\left\Vert \bz-\bx_{1}\right\Vert }{\sqrt{\left(\qu(T)G_{f,2}^{2}\lor\sqrt{n\qu(T)}G_{f,1}G_{f,2}\right)(1+\log T)}}$
to optimize the dependence on parameters.
\end{thm}

\begin{proof}
Similar to (\ref{eq:RR-cvx-full-1}), if we can find a constant $\Phi\geq0$
such that
\[
\left|\E\left[\Omega_{t}(\bx_{s})\right]\right|\leq\Phi\eta_{s},\forall t\in\left[T\right],s\in\left[t\right].
\]
Then there is
\begin{equation}
\E\left[F(\bx_{T+1})-F(\bz)\right]\leq\O\left(\frac{\left\Vert \bz-\bx_{1}\right\Vert ^{2}}{\sum_{t=1}^{T}\eta_{t}}+\left(G_{f,2}^{2}+\Phi\right)\sum_{t=1}^{T}\frac{\eta_{t}^{2}}{\sum_{s=t}^{T}\eta_{s}}\right),\label{eq:SS-cvx-full-1}
\end{equation}
Here, we claim
\begin{equation}
\Phi=8\qu(T)G_{f,2}^{2}+2\sqrt{2n\qu(T)}G_{f,1}G_{f,2}.\label{eq:SS-cvx-full-Psi}
\end{equation}
To prove our statement, we invoke Lemma \ref{lem:SS-Omega} with $\mu=0$
to have for any $t\in\left[T\right]$ and $s\in\left[t\right]$,
\begin{align}
\left|\E\left[\Omega_{t}(\bx_{s})\right]\right|\leq & 4G_{f,2}^{2}\sum_{j=1}^{s-1}\eta_{j}\left(\1\left[\re(j)=\re(t)\right]+\frac{\1\left[\re(j)\neq\re(t)\right]}{n-1}\right)\nonumber \\
 & +\frac{2}{n}\sum_{i=1}^{n}G_{i}\sqrt{\sum_{j=1}^{s-1}\eta_{j}^{2}\left(G_{f,2}^{2}+G_{i}^{2}\1\left[\re(j)=\re(t)\right]+\frac{nG_{f,2}^{2}-G_{i}^{2}}{n-1}\1\left[\re(j)\neq\re(t)\right]\right)}\nonumber \\
\leq & 4G_{f,2}^{2}\sum_{j=1}^{\qu(T)n}\eta_{j}\left(\1\left[\re(j)=\re(t)\right]+\frac{\1\left[\re(j)\neq\re(t)\right]}{n-1}\right)\nonumber \\
 & +\frac{2}{n}\sum_{i=1}^{n}G_{i}\sqrt{\sum_{j=1}^{\qu(T)n}\eta_{j}^{2}\left(G_{f,2}^{2}+G_{i}^{2}\1\left[\re(j)=\re(t)\right]+\frac{nG_{f,2}^{2}-G_{i}^{2}}{n-1}\1\left[\re(j)\neq\re(t)\right]\right)},\label{eq:SS-cvx-full-Omega}
\end{align}
where the second step is by $s\le\qu(s)n\leq\qu(T)n$. Note that when
the stepsize is constant, there are
\begin{equation}
\sum_{j=1}^{\qu(T)n}\eta_{j}\left(\1\left[\re(j)=\re(t)\right]+\frac{\1\left[\re(j)\neq\re(t)\right]}{n-1}\right)=\eta_{s}\sum_{j=1}^{\qu(T)n}\left(\1\left[\re(j)=\re(t)\right]+\frac{\1\left[\re(j)\neq\re(t)\right]}{n-1}\right)=2\qu(T)\eta_{s},\label{eq:SS-cvx-full-2}
\end{equation}
and
\begin{align}
 & \sum_{j=1}^{\qu(T)n}\eta_{j}^{2}\left(G_{f,2}^{2}+G_{i}^{2}\1\left[\re(j)=\re(t)\right]+\frac{nG_{f,2}^{2}-G_{i}^{2}}{n-1}\1\left[\re(j)\neq\re(t)\right]\right)\nonumber \\
= & \eta_{s}^{2}\sum_{j=1}^{\qu(T)n}\left(G_{f,2}^{2}+G_{i}^{2}\1\left[\re(j)=\re(t)\right]+\frac{nG_{f,2}^{2}-G_{i}^{2}}{n-1}\1\left[\re(j)\neq\re(t)\right]\right)=2n\qu(T)G_{f,2}^{2}\eta_{s}^{2}.\label{eq:SS-cvx-full-3}
\end{align}
Combine (\ref{eq:SS-cvx-full-Omega}), (\ref{eq:SS-cvx-full-2}) and
(\ref{eq:SS-cvx-full-3}) to obtain
\[
\left|\E\left[\Omega_{t}(\bx_{s})\right]\right|\leq\left(8\qu(T)G_{f,2}^{2}+2\sqrt{2n\qu(T)}G_{f,1}G_{f,2}\right)\eta_{s}\overset{\eqref{eq:SS-cvx-full-Psi}}{=}\Phi\eta_{s}.
\]

By (\ref{eq:SS-cvx-full-1}) and (\ref{eq:SS-cvx-full-Psi}), we have
\begin{align*}
\E\left[F(\bx_{T+1})-F(\bz)\right] & \leq\O\left(\frac{\left\Vert \bz-\bx_{1}\right\Vert ^{2}}{\sum_{t=1}^{T}\eta_{t}}+\left(G_{f,2}^{2}+\qu(T)G_{f,2}^{2}+\sqrt{n\qu(T)}G_{f,1}G_{f,2}\right)\sum_{t=1}^{T}\frac{\eta_{t}^{2}}{\sum_{s=t}^{T}\eta_{s}}\right)\\
 & \overset{(a)}{\leq}\O\left(\frac{\left\Vert \bz-\bx_{1}\right\Vert ^{2}}{\sum_{t=1}^{T}\eta_{t}}+\left(\qu(T)G_{f,2}^{2}+\sqrt{n\qu(T)}G_{f,1}G_{f,2}\right)\sum_{t=1}^{T}\frac{\eta_{t}^{2}}{\sum_{s=t}^{T}\eta_{s}}\right)\\
 & \overset{(b)}{\leq}\O\left(\left(\frac{\left\Vert \bz-\bx_{1}\right\Vert ^{2}}{\eta}+\eta\left(\qu(T)G_{f,2}^{2}+\sqrt{n\qu(T)}G_{f,1}G_{f,2}\right)(1+\log T)\right)\frac{1}{\sqrt{T}}\right)\\
 & =\O\left(\left(\frac{\left\Vert \bz-\bx_{1}\right\Vert ^{2}}{\eta}+\eta\left(\qu(T)G_{f,2}^{2}\lor\sqrt{n\qu(T)}G_{f,1}G_{f,2}\right)(1+\log T)\right)\frac{1}{\sqrt{T}}\right),
\end{align*}
where $(a)$ is due to $1\leq\qu(T)$ and $(b)$ holds by plugging
in $\eta_{t}=\frac{\eta}{\sqrt{T}},\forall t\in\left[T\right]$. 

Finally, setting $\eta=\frac{\left\Vert \bz-\bx_{1}\right\Vert }{\sqrt{\left(\qu(T)G_{f,2}^{2}\lor\sqrt{n\qu(T)}G_{f,1}G_{f,2}\right)(1+\log T)}}$
to obtain
\begin{align}
\E\left[F(\bx_{T+1})-F(\bz)\right] & \leq\O\left(\frac{\sqrt{\qu(T)G_{f,2}^{2}\lor\sqrt{n\qu(T)}G_{f,1}G_{f,2}}\left\Vert \bz-\bx_{1}\right\Vert \sqrt{1+\log T}}{\sqrt{T}}\right)\nonumber \\
 & \overset{(c)}{\leq}\O\left(\left(\frac{n^{1/4}\sqrt{G_{f,1}G_{f,2}}}{\sqrt{T}}\lor\frac{\sqrt{G_{f,1}G_{f,2}}}{T^{1/4}}\lor\frac{G_{f,2}}{\sqrt{n}}\right)\left\Vert \bz-\bx_{1}\right\Vert \sqrt{1+\log T}\right),\label{eq:SS-cvx-full-4}
\end{align}
where $(c)$ is by noticing $\qu(T)\leq\frac{T}{n}+1$ and $G_{f,2}<\sqrt{n}G_{f,1}$,
together implying
\begin{align*}
\qu(T)G_{f,2}^{2}\lor\sqrt{n\qu(T)}G_{f,1}G_{f,2} & \leq\left(\frac{T}{n}G_{f,2}^{2}+G_{f,2}^{2}\right)\lor\sqrt{T+n}G_{f,1}G_{f,2}\\
 & =\O\left(\frac{T}{n}G_{f,2}^{2}\lor G_{f,2}^{2}\lor\left(\sqrt{T\lor n}G_{f,1}G_{f,2}\right)\right)\\
 & \leq\O\left(\frac{T}{n}G_{f,2}^{2}\lor\sqrt{n}G_{f,1}G_{f,2}\lor\left(\sqrt{T\lor n}G_{f,1}G_{f,2}\right)\right)\\
 & =\O\left(\sqrt{n}G_{f,1}G_{f,2}\lor\sqrt{T}G_{f,1}G_{f,2}\lor\frac{T}{n}G_{f,2}^{2}\right).
\end{align*}
When $T\geq n$, we observe that $\frac{n^{1/4}}{\sqrt{T}}\leq\frac{1}{T^{1/4}}$,
(\ref{eq:SS-cvx-full-4}) thus reduces to the following form used
in Theorem \ref{thm:SS-cvx}
\[
\E\left[F(\bx_{T+1})-F(\bz)\right]\leq\O\left(\left(\frac{\sqrt{G_{f,1}G_{f,2}}}{T^{1/4}}\lor\frac{G_{f,2}}{\sqrt{n}}\right)\left\Vert \bz-\bx_{1}\right\Vert \sqrt{1+\log T}\right).
\]
\end{proof}

The above Theorem \ref{thm:SS-cvx-full} (or see (\ref{eq:SS-cvx-full-4})
for the final bound) is not very satisfying as the rate will be blocked
at $\widetilde{\O}\left(\frac{1}{\sqrt{n}}\right)$ even if $T$ approaches
$+\infty$. As discussed in the main text, if $\psi=\charf_{\domx}$
and $T=Kn$ where $K\in\N$, we actually can do better. Note that
$\psi=\charf_{\domx}$ can be further relaxed to $\psi=\varphi+\charf_{\domx}$
for $\varphi$ being $G_{\psi}$-Lipschitz on $\domx$. However, this
will introduce new parameters $G_{\psi}$ in the final rate. Instead,
we choose to keep the following simple form. For why the relaxation
holds, see Lemma \ref{lem:special-refined}.
\begin{thm}
\label{thm:SS-cvx-improved-full}(Full version of Theorem \ref{thm:SS-cvx-improved})
Under Assumptions \ref{assu:basic} (with $\psi=\charf_{\domx}$)
and \ref{assu:lip}, suppose $T=Kn$ where $K\in\N$ and the $\SS$
sampling scheme is employed with the stepsize $\eta_{t}=\frac{\eta}{\sqrt{T}},\forall t\in\left[T\right]$,
then for any $\bz\in\R^{d}$, Algorithm \ref{alg:Alg} guarantees
\[
\E\left[F(\bx_{Kn+1})-F(\bz)\right]\leq\O\left(\left(\frac{\left\Vert \bz-\bx_{1}\right\Vert ^{2}}{\eta}+\eta\left(\sqrt{nK}G_{f,1}G_{f,2}\land nG_{f,1}^{2}\right)(1+\log nK)\right)\frac{1}{\sqrt{nK}}\right).
\]
Setting $\eta=\frac{\left\Vert \bz-\bx_{1}\right\Vert }{\sqrt{\left(\sqrt{nK}G_{f,1}G_{f,2}\land nG_{f,1}^{2}\right)(1+\log nK)}}$
to optimize the dependence on parameters.
\end{thm}

\begin{proof}
Under the stepsize $\eta_{t}=\frac{\eta}{\sqrt{T}},\forall t\in\left[T\right]$,
we first know by Theorem \ref{thm:SS-cvx-full},
\begin{align}
\E\left[F(\bx_{T+1})-F(\bz)\right] & \leq\O\left(\left(\frac{\left\Vert \bz-\bx_{1}\right\Vert ^{2}}{\eta}+\eta\left(\qu(T)G_{f,2}^{2}\lor\sqrt{n\qu(T)}G_{f,1}G_{f,2}\right)(1+\log T)\right)\frac{1}{\sqrt{T}}\right)\nonumber \\
\Rightarrow\E\left[F(\bx_{Kn+1})-F(\bz)\right] & \leq\O\left(\left(\frac{\left\Vert \bz-\bx_{1}\right\Vert ^{2}}{\eta}+\eta\left(KG_{f,2}^{2}\lor\sqrt{nK}G_{f,1}G_{f,2}\right)(1+\log nK)\right)\frac{1}{\sqrt{nK}}\right),\label{eq:SS-cvx-improved-full-1}
\end{align}
where in the second line we use $\qu(T)=K$ when $T=Kn$.

Next, note that the $\SS$ sampling scheme and the stepsize $\eta_{t}=\frac{\eta}{\sqrt{T}},\forall t\in\left[T\right]$
respectively satisfy Conditions \ref{enu:special-refined-condition-1}
and \ref{enu:special-refined-condition-2} in Lemma \ref{lem:special-refined}.
Moreover, we can take $\varphi=0$ and $G_{\psi}=0$ in Condition
\ref{enu:special-refined-condition-3}. Hence there is almost surely
for any $\bz\in\R^{d}$,
\begin{align*}
F(\bx_{Kn+1})-F(\bz) & \leq\O\left(\left(\frac{\left\Vert \bz-\bx_{1}\right\Vert ^{2}}{\eta}+\eta nG_{f,1}^{2}\sum_{k=1}^{K}\frac{1}{K-k+1}\right)\frac{1}{\sqrt{nK}}\right)\\
 & \leq\O\left(\left(\frac{\left\Vert \bz-\bx_{1}\right\Vert ^{2}}{\eta}+\eta nG_{f,1}^{2}(1+\log K)\right)\frac{1}{\sqrt{nK}}\right),
\end{align*}
which implies
\begin{equation}
\E\left[F(\bx_{Kn+1})-F(\bz)\right]\leq\O\left(\left(\frac{\left\Vert \bz-\bx_{1}\right\Vert ^{2}}{\eta}+\eta nG_{f,1}^{2}(1+\log K)\right)\frac{1}{\sqrt{nK}}\right).\label{eq:SS-cvx-improved-full-2}
\end{equation}

We combine (\ref{eq:SS-cvx-improved-full-1}) and (\ref{eq:SS-cvx-improved-full-2})
to obtain
\begin{align*}
\E\left[F(\bx_{Kn+1})-F(\bz)\right] & \leq\O\left(\left(\frac{\left\Vert \bz-\bx_{1}\right\Vert ^{2}}{\eta}+\eta\left(\left(KG_{f,2}^{2}\lor\sqrt{nK}G_{f,1}G_{f,2}\right)\land nG_{f,1}^{2}\right)(1+\log nK)\right)\frac{1}{\sqrt{nK}}\right)\\
 & =\O\left(\left(\frac{\left\Vert \bz-\bx_{1}\right\Vert ^{2}}{\eta}+\eta\left(\sqrt{nK}G_{f,1}G_{f,2}\land nG_{f,1}^{2}\right)(1+\log nK)\right)\frac{1}{\sqrt{nK}}\right),
\end{align*}
where we use the fact $(a\lor\sqrt{ab})\land b=\sqrt{ab}\land b$
for $a=KG_{f,2}^{2}$ and $b=nG_{f,1}^{2}$. Finally, Setting $\eta=\frac{\left\Vert \bz-\bx_{1}\right\Vert }{\sqrt{\left(\sqrt{nK}G_{f,1}G_{f,2}\land nG_{f,1}^{2}\right)(1+\log nK)}}$
to obtain
\[
\E\left[F(\bx_{Kn+1})-F(\bz)\right]\leq\O\left(\left(\frac{\sqrt{G_{f,1}G_{f,2}}}{n^{1/4}K^{1/4}}\land\frac{G_{f,1}}{\sqrt{K}}\right)\left\Vert \bz-\bx_{1}\right\Vert \sqrt{1+\log nK}\right).
\]
\end{proof}

Finally, we establish the convergence upper bound for strongly convex
functions in Theorem \ref{thm:SS-str-full}. As mentioned after Theorem
\ref{thm:SS-str}, we actually can improve the rate further to avoid
the $\widetilde{\O}\left(\frac{1}{n}\right)$ barrier when $T=Kn$
where $K\in\N$ and $\psi=\varphi+\charf_{\domx}$ for $\varphi$
being Lipschitz on $\domx$. However, the rate in that case will be
in the order $\widetilde{\O}\left(\frac{n}{K}\right)$ for larger
$K$, which is still slower than the bound $\widetilde{\O}\left(\frac{1}{K}\right)$
in \cite{pmlr-v235-liu24cg}. Therefore, we do not provide it here.
See Lemma \ref{lem:special-refined} for why we can do at most $\widetilde{\O}\left(\frac{n}{K}\right)$.
\begin{thm}
\label{thm:SS-str-full}(Full version of Theorem \ref{thm:SS-str})
Under Assumptions \ref{assu:basic} (with $\mu>0$) and \ref{assu:lip},
suppose the $\SS$ sampling scheme is employed with the stepsize $\eta_{t}=\frac{m}{\mu t},\forall t\in\left[T\right]$
where $m\in\N$, then for any $\bz\in\R^{d}$, Algorithm \ref{alg:Alg}
guarantees
\[
\E\left[F(\bx_{T+1})-F(\bz)\right]\leq\O\left(\frac{\mu\left\Vert \bz-\bx_{1}\right\Vert ^{2}}{\binom{m+T}{m}}+\frac{m(1+\log T)}{\mu}\left(\frac{\sqrt{n}G_{f,1}G_{f,2}}{T}+\frac{G_{f,1}G_{f,2}}{\sqrt{T}}+\frac{G_{f,2}^{2}}{n}\right)\right).
\]
\end{thm}

\begin{proof}
Note that the $\SS$ sampling scheme satisfies $\I(t)\diseq\mathrm{Uniform}\left[n\right],\forall t\in\left[T\right]$,
and the stepsize $\eta_{t}=\frac{m}{\mu t}$ is non-increasing. Hence,
Conditions \ref{enu:last-condition-1} and \ref{enu:last-condition-2}
in Lemma \ref{lem:last} are fulfilled. If Condition \ref{enu:last-condition-3}
also holds, i.e., $\left|\E\left[\Omega_{t}(\bx_{s})\right]\right|\leq\Phi\eta_{s},\forall t\in\left[T\right],s\in\left[t\right]$
for some $\Phi\geq0$, we can invoke Lemma \ref{lem:last} and follow
the same steps of proving (\ref{eq:RR-str-full-4}) to obtain for
any $\bz\in\R^{d}$,
\begin{equation}
\E\left[F(\bx_{T+1})-F(\bz)\right]\leq\O\left(\frac{\mu\left\Vert \bz-\bx_{1}\right\Vert ^{2}}{\binom{m+T}{m}}+\frac{m\left(G_{f,2}^{2}+\Phi\right)(1+\log T)}{\mu T}\right),\label{eq:SS-str-full-1}
\end{equation}
and know
\begin{equation}
\gamma_{t}\eta_{t}=\frac{1}{\mu}\binom{m+t-1}{m-1},\forall t\in\left[T\right].\label{eq:SS-str-full-gamma-eta}
\end{equation}

Our last task is to find a constant $\Phi\geq0$ satisfying
\[
\left|\E\left[\Omega_{t}(\bx_{s})\right]\right|\leq\Phi\eta_{s},\forall t\in\left[T\right],s\in\left[t\right].
\]
Here we claim
\begin{equation}
\Phi=8\left(\frac{T}{n}+1\right)G_{f,2}^{2}+2\sqrt{2(T+n)}G_{f,1}G_{f,2}.\label{eq:SS-str-full-Psi}
\end{equation}
To prove our statement, we invoke Lemma \ref{lem:SS-Omega} to have
for any $t\in\left[T\right]$ and $s\in\left[t\right]$,
\begin{align}
\left|\E\left[\Omega_{t}(\bx_{s})\right]\right|\leq & 4G_{f,2}^{2}\sum_{j=1}^{s-1}\frac{\gamma_{j}\eta_{j}}{\gamma_{s}}\left(\1\left[\re(j)=\re(t)\right]+\frac{\1\left[\re(j)\neq\re(t)\right]}{n-1}\right)\nonumber \\
 & +\frac{2}{n}\sum_{i=1}^{n}G_{i}\sqrt{\sum_{j=1}^{s-1}\frac{\gamma_{j}^{2}\eta_{j}^{2}}{\gamma_{s}^{2}}\left(G_{f,2}^{2}+G_{i}^{2}\1\left[\re(j)=\re(t)\right]+\frac{nG_{f,2}^{2}-G_{i}^{2}}{n-1}\1\left[\re(j)\neq\re(t)\right]\right)}\nonumber \\
\leq & 4G_{f,2}^{2}\eta_{s}\sum_{j=1}^{s-1}\1\left[\re(j)=\re(t)\right]+\frac{\1\left[\re(j)\neq\re(t)\right]}{n-1}\nonumber \\
 & +\frac{2\eta_{s}}{n}\sum_{i=1}^{n}G_{i}\sqrt{\sum_{j=1}^{s-1}G_{f,2}^{2}+G_{i}^{2}\1\left[\re(j)=\re(t)\right]+\frac{nG_{f,2}^{2}-G_{i}^{2}}{n-1}\1\left[\re(j)\neq\re(t)\right]},\label{eq:SS-str-full-2}
\end{align}
where the last step is due to $\frac{\gamma_{j}\eta_{j}}{\gamma_{s}}\leq\eta_{s},\forall j\in\left[s-1\right]$
because $\gamma_{t}\eta_{t}$ is non-decreasing in $t$ as shown in
(\ref{eq:SS-str-full-gamma-eta}). Note that $s\leq\qu(s)n$, we therefore
have
\begin{equation}
\sum_{j=1}^{s-1}\1\left[\re(j)=\re(t)\right]+\frac{\1\left[\re(j)\neq\re(t)\right]}{n-1}\leq\sum_{j=1}^{\qu(s)n}\1\left[\re(j)=\re(t)\right]+\frac{\1\left[\re(j)\neq\re(t)\right]}{n-1}=2\qu(s),\label{eq:SS-str-full-3}
\end{equation}
and
\begin{align}
 & \sum_{j=1}^{s-1}\left(G_{f,2}^{2}+G_{i}^{2}\1\left[\re(j)=\re(t)\right]+\frac{nG_{f,2}^{2}-G_{i}^{2}}{n-1}\1\left[\re(j)\neq\re(t)\right]\right)\nonumber \\
\leq & \sum_{j=1}^{\qu(s)n}G_{f,2}^{2}+G_{i}^{2}\1\left[\re(j)=\re(t)\right]+\frac{nG_{f,2}^{2}-G_{i}^{2}}{n-1}\1\left[\re(j)\neq\re(t)\right]\nonumber \\
= & 2n\qu(s)G_{f,2}^{2}.\label{eq:SS-str-full-4}
\end{align}
Combine (\ref{eq:SS-str-full-2}), (\ref{eq:SS-str-full-3}) and (\ref{eq:SS-str-full-4})
to obtain
\begin{align*}
\left|\E\left[\Omega_{t}(\bx_{s})\right]\right| & \leq8\qu(s)G_{f,2}^{2}\eta_{s}+\frac{2\eta_{s}}{n}\sum_{i=1}^{n}G_{i}\sqrt{2n\qu(s)G_{f,2}^{2}}=\left(8\qu(s)G_{f,2}^{2}+2\sqrt{2n\qu(s)}G_{f,1}G_{f,2}\right)\eta_{s}\\
 & \leq\left(8\left(\frac{T}{n}+1\right)G_{f,2}^{2}+2\sqrt{2(T+n)}G_{f,1}G_{f,2}\right)\eta_{s}\overset{\eqref{eq:SS-str-full-Psi}}{=}\Phi\eta_{s},
\end{align*}
where the second inequality holds by $\qu(s)\leq\qu(T)\leq\frac{T}{n}+1$.

By (\ref{eq:SS-str-full-1}) and (\ref{eq:SS-str-full-Psi}), we finally
have
\begin{align*}
\E\left[F(\bx_{T+1})-F(\bz)\right] & \leq\O\left(\frac{\mu\left\Vert \bz-\bx_{1}\right\Vert ^{2}}{\binom{m+T}{m}}+\frac{m\left(G_{f,2}^{2}+\left(\frac{T}{n}+1\right)G_{f,2}^{2}+\sqrt{T+n}G_{f,1}G_{f,2}\right)(1+\log T)}{\mu T}\right)\\
 & =\O\left(\frac{\mu\left\Vert \bz-\bx_{1}\right\Vert ^{2}}{\binom{m+T}{m}}+\frac{m\left(G_{f,2}^{2}+\sqrt{n}G_{f,1}G_{f,2}+\sqrt{T}G_{f,1}G_{f,2}+\frac{T}{n}G_{f,2}^{2}\right)(1+\log T)}{\mu T}\right)\\
 & \leq\O\left(\frac{\mu\left\Vert \bz-\bx_{1}\right\Vert ^{2}}{\binom{m+T}{m}}+\frac{m(1+\log T)}{\mu}\left(\frac{\sqrt{n}G_{f,1}G_{f,2}}{T}+\frac{G_{f,1}G_{f,2}}{\sqrt{T}}+\frac{G_{f,2}^{2}}{n}\right)\right),
\end{align*}
where the last step is by $G_{f,2}<\sqrt{n}G_{f,1}$. To recover the
rate stated in Theorem \ref{thm:SS-str}, we only need to take $m=2$
and observe that $\frac{\sqrt{n}}{T}\leq\frac{1}{\sqrt{T}}$ when
$T\geq n$.
\end{proof}

\section{Theoretical Analysis\label{sec:analysis}}

This section includes our theoretical analysis in detail. 

\subsection{General Lemmas}

As mentioned in \ref{sec:idea}, the high-level idea is inspired by
\cite{liu2024revisiting}. However, several important modifications
are required to circumvent the potential issue caused by the general
$\I(t)$ considered in our setting. To begin with, we first characterize
the progress made in every single step, provided in the following
Lemma \ref{lem:core}. Note this result holds for any $\I(t)$ even
if it does not equal $\mathrm{Uniform}\left[n\right]$ in distribution.
\begin{lem}
\label{lem:core}Under Assumptions \ref{assu:basic} and \ref{assu:lip},
for any $t\in\left[T\right]$ and $\by\in\R^{d}$, Algorithm \ref{alg:Alg}
guarantees
\[
F(\bx_{t+1})-F(\by)\leq\frac{\left\Vert \by-\bx_{t}\right\Vert ^{2}}{2\eta_{t}}-(1+\mu\eta_{t})\frac{\left\Vert \by-\bx_{t+1}\right\Vert ^{2}}{2\eta_{t}}+\eta_{t}\left(G_{\I(t)}^{2}+G_{f,1}^{2}\right)+\Omega_{t}(\by)-\Omega_{t}(\bx_{t}),
\]
where $\Omega_{t}(\by)\defeq f_{\I(t)}(\by)-f(\by),\forall\by\in\R^{d}$.
\end{lem}

\begin{proof}
Let $t\in\left[T\right]$ be fixed. If $\by\not\in\dom\psi$, then
the inequality holds automatically since $F(\by)=+\infty$ and $F(\bx_{t+1})<+\infty$
almost surely. Thus, we only need to consider the case $\by\in\dom\psi$.
By the convexity of $f_{\I(t)}$, we have
\begin{equation}
f_{\I(t)}(\bx_{t})-f_{\I(t)}(\by)\leq\left\langle \nabla f_{\I(t)}(\bx_{t}),\bx_{t}-\by\right\rangle =\left\langle \nabla f_{\I(t)}(\bx_{t}),\bx_{t+1}-\by\right\rangle +\left\langle \nabla f_{\I(t)}(\bx_{t}),\bx_{t}-\bx_{t+1}\right\rangle .\label{eq:core-1}
\end{equation}
According to the update rule of $\bx_{t+1}$, there exists $\nabla\psi(\bx_{t+1})\in\partial\psi(\bx_{t+1})$
such that
\[
\bzero=\nabla\psi(\bx_{t+1})+\nabla f_{\I(t)}(\bx_{t})+\frac{\bx_{t+1}-\bx_{t}}{\eta_{t}}\Rightarrow\nabla f_{\I(t)}(\bx_{t})=\frac{\bx_{t}-\bx_{t+1}}{\eta_{t}}-\nabla\psi(\bx_{t+1}),
\]
which gives us
\begin{align}
\left\langle \nabla f_{\I(t)}(\bx_{t}),\bx_{t+1}-\by\right\rangle  & =\frac{\left\langle \bx_{t}-\bx_{t+1},\bx_{t+1}-\by\right\rangle }{\eta_{t}}+\left\langle \nabla\psi(\bx_{t+1}),\by-\bx_{t+1}\right\rangle \nonumber \\
 & =\frac{\left\Vert \by-\bx_{t}\right\Vert ^{2}-\left\Vert \by-\bx_{t+1}\right\Vert ^{2}-\left\Vert \bx_{t+1}-\bx_{t}\right\Vert ^{2}}{2\eta_{t}}+\left\langle \nabla\psi(\bx_{t+1}),\by-\bx_{t+1}\right\rangle \nonumber \\
 & \leq\frac{\left\Vert \by-\bx_{t}\right\Vert ^{2}}{2\eta_{t}}-(1+\mu\eta_{t})\frac{\left\Vert \by-\bx_{t+1}\right\Vert ^{2}}{2\eta_{t}}-\frac{\left\Vert \bx_{t+1}-\bx_{t}\right\Vert ^{2}}{2\eta_{t}}+\psi(\by)-\psi(\bx_{t+1}),\label{eq:core-2}
\end{align}
where the last step holds by the $\mu$-strong convexity of $\psi$.
In addition, there is
\begin{align}
\left\langle \nabla f_{\I(t)}(\bx_{t}),\bx_{t}-\bx_{t+1}\right\rangle  & =\left\langle \nabla f_{\I(t)}(\bx_{t})-\nabla f(\bx_{t+1}),\bx_{t}-\bx_{t+1}\right\rangle +\left\langle \nabla f(\bx_{t+1}),\bx_{t}-\bx_{t+1}\right\rangle \nonumber \\
 & \leq\left\langle \nabla f_{\I(t)}(\bx_{t})-\nabla f(\bx_{t+1}),\bx_{t}-\bx_{t+1}\right\rangle +f(\bx_{t})-f(\bx_{t+1}),\label{eq:core-3}
\end{align}
where the last inequality is due to the convexity of $f$.

Plug (\ref{eq:core-2}) and (\ref{eq:core-3}) back into (\ref{eq:core-1})
and rearrange terms to get (recall $\Omega_{t}(\cdot)=f_{\I(t)}(\cdot)-f(\cdot)$)
\begin{align*}
F(\bx_{t+1})-F(\by)\leq & \frac{\left\Vert \by-\bx_{t}\right\Vert ^{2}}{2\eta_{t}}-(1+\mu\eta_{t})\frac{\left\Vert \by-\bx_{t+1}\right\Vert ^{2}}{2\eta_{t}}+\Omega_{t}(\by)-\Omega_{t}(\bx_{t})\\
 & +\left\langle \nabla f_{\I(t)}(\bx_{t})-\nabla f(\bx_{t+1}),\bx_{t}-\bx_{t+1}\right\rangle -\frac{\left\Vert \bx_{t+1}-\bx_{t}\right\Vert ^{2}}{2\eta_{t}}\\
\leq & \frac{\left\Vert \by-\bx_{t}\right\Vert ^{2}}{2\eta_{t}}-(1+\mu\eta_{t})\frac{\left\Vert \by-\bx_{t+1}\right\Vert ^{2}}{2\eta_{t}}+\eta_{t}\left(G_{\I(t)}^{2}+G_{f,1}^{2}\right)+\Omega_{t}(\by)-\Omega_{t}(\bx_{t}),
\end{align*}
where the last line follows by
\begin{align*}
 & \left\langle \nabla f_{\I(t)}(\bx_{t})-\nabla f(\bx_{t+1}),\bx_{t}-\bx_{t+1}\right\rangle \\
\overset{(a)}{\leq} & \left\Vert \nabla f_{\I(t)}(\bx_{t})-\nabla f(\bx_{t+1})\right\Vert \left\Vert \bx_{t+1}-\bx_{t}\right\Vert \overset{(b)}{\leq}\left(G_{\I(t)}+G_{f,1}\right)\left\Vert \bx_{t+1}-\bx_{t}\right\Vert \\
\overset{(c)}{\leq} & \frac{\eta_{t}\left(G_{\I(t)}+G_{f,1}\right)^{2}}{2}+\frac{\left\Vert \bx_{t+1}-\bx_{t}\right\Vert ^{2}}{2\eta_{t}}\leq\eta_{t}\left(G_{\I(t)}^{2}+G_{f,1}^{2}\right)+\frac{\left\Vert \bx_{t+1}-\bx_{t}\right\Vert ^{2}}{2\eta_{t}},
\end{align*}
where $(a)$ is due to Cauchy-Schwarz inequality, $(b)$ is because
$f_{\I(t)}$ and $f$ are respectively Lipschitz on $\dom\psi$ with
the parameter $G_{\I(t)}$ and $G_{f,1}$, and $(c)$ is by AM-GM
inequality.
\end{proof}

Next, we are ready to prove our core last-iterate result, Lemma \ref{lem:last},
which provides new sufficient conditions for the last-iterate convergence
rate. We remind the reader again that Condition \ref{enu:last-condition-2}
is actually not necessary since we can simply stop the proof at the
equation (\ref{eq:last-8}) and change $\eta_{t}^{2}$ to $\eta_{t}(\eta_{t}\lor\eta_{t+1})$
in the final bound.
\begin{lem}
\label{lem:last} (Full version of Lemma \ref{lem:last-main}) Under
Assumptions \ref{assu:basic} and \ref{assu:lip}, suppose the following
three conditions hold:
\begin{enumerate}
\item \label{enu:last-condition-1}The index satisfies $\I(t)\diseq\mathrm{Uniform}\left[n\right],\forall t\in\left[T\right]$.
\item \label{enu:last-condition-2}The stepsize $\eta_{t},\forall t\in\left[T\right]$
is non-increasing.
\item \label{enu:last-condition-3}$\left|\E\left[\Omega_{t}(\bx_{s})\right]\right|\leq\Phi\eta_{s},\forall t\in\left[T\right],s\in\left[t\right]$
where $\Phi\geq0$ is a constant probably depending on $T$, $n$,
$\mu$, and $G_{1}$ to $G_{n}$
\end{enumerate}
Then for any $\bz\in\R^{d}$, Algorithm \ref{alg:Alg} guarantees
\[
\E\left[F(\bx_{T+1})-F(\bz)\right]\leq\O\left(\frac{\left\Vert \bz-\bx_{1}\right\Vert ^{2}}{\sum_{t=1}^{T}\gamma_{t}\eta_{t}}+\left(G_{f,2}^{2}+\Phi\right)\sum_{t=1}^{T}\frac{\gamma_{t}\eta_{t}^{2}}{\sum_{s=t}^{T}\gamma_{s}\eta_{s}}\right),
\]
where $\Omega_{t}(\cdot)$ is defined in Lemma \ref{lem:core} and
$\gamma_{t}\defeq\prod_{s=1}^{t-1}(1+\mu\eta_{s}),\forall t\in\left[T+1\right]$.
\end{lem}

\begin{proof}
It is enough to only consider the case $\bz\in\dom\psi$. In the following
proof, we fix a point $\bz\in\dom\psi$. By Lemma \ref{lem:core},
the following inequality holds for any $t\in\left[T\right]$ and $\by\in\R^{d}$,
\[
F(\bx_{t+1})-F(\by)\leq\frac{\left\Vert \by-\bx_{t}\right\Vert ^{2}}{2\eta_{t}}-(1+\mu\eta_{t})\frac{\left\Vert \by-\bx_{t+1}\right\Vert ^{2}}{2\eta_{t}}+\eta_{t}\left(G_{\I(t)}^{2}+G_{f,1}^{2}\right)+\Omega_{t}(\by)-\Omega_{t}(\bx_{t}).
\]
Multiplying both sides by $\gamma_{t}\eta_{t}$ yields (note that
$(1+\mu\eta_{t})\gamma_{t}=\gamma_{t+1},\forall t\in\left[T\right]$)
\begin{equation}
\gamma_{t}\eta_{t}(F(\bx_{t+1})-F(\by))\leq\frac{\gamma_{t}\left\Vert \by-\bx_{t}\right\Vert ^{2}-\gamma_{t+1}\left\Vert \by-\bx_{t+1}\right\Vert ^{2}}{2}+\gamma_{t}\eta_{t}^{2}\left(G_{\I(t)}^{2}+G_{f,1}^{2}\right)+\gamma_{t}\eta_{t}(\Omega_{t}(\by)-\Omega_{t}(\bx_{t})).\label{eq:last-1}
\end{equation}

Next, inspired by \cite{liu2024revisiting}, we first define the following
non-decreasing sequence
\begin{align}
v_{t} & \defeq\frac{\gamma_{T}\eta_{T}}{\sum_{s=t-1}^{T}\gamma_{s}\eta_{s}},\forall t\in\left[T+1\right]\backslash\left[1\right],\label{eq:last-v1}\\
v_{1} & \defeq v_{2}=\frac{\gamma_{T}\eta_{T}}{\sum_{s=1}^{T}\gamma_{s}\eta_{s}},\label{eq:last-v2}
\end{align}
and then introduce
\begin{equation}
\bz_{t}\defeq\frac{v_{1}}{v_{t}}\bz+\sum_{s=1}^{t-1}\frac{v_{s+1}-v_{s}}{v_{t}}\bx_{s},\forall t\in\left[T+1\right].\label{eq:last-z1}
\end{equation}
Note that $\bz_{t}$ also falls in $\dom\psi$ as it is a convex combination
of points in $\dom\psi$ and admits
\begin{equation}
\bz_{t+1}=\frac{v_{t}}{v_{t+1}}\bz_{t}+\left(1-\frac{v_{t}}{v_{t+1}}\right)\bx_{t},\forall t\in\left[T\right].\label{eq:last-z2}
\end{equation}
For any $t\in\left[T\right]$, we invoke (\ref{eq:last-1}) with $\by=\bz_{t+1}$
to obtain
\begin{align}
 & \gamma_{t}\eta_{t}(F(\bx_{t+1})-F(\bz_{t+1}))\nonumber \\
\leq & \frac{\gamma_{t}\left\Vert \bz_{t+1}-\bx_{t}\right\Vert ^{2}-\gamma_{t+1}\left\Vert \bz_{t+1}-\bx_{t+1}\right\Vert ^{2}}{2}+\gamma_{t}\eta_{t}^{2}\left(G_{\I(t)}^{2}+G_{f,1}^{2}\right)+\gamma_{t}\eta_{t}(\Omega_{t}(\bz_{t+1})-\Omega_{t}(\bx_{t}))\nonumber \\
\leq & \frac{\gamma_{t}\frac{v_{t}}{v_{t+1}}\left\Vert \bz_{t}-\bx_{t}\right\Vert ^{2}-\gamma_{t+1}\left\Vert \bz_{t+1}-\bx_{t+1}\right\Vert ^{2}}{2}+\gamma_{t}\eta_{t}^{2}\left(G_{\I(t)}^{2}+G_{f,1}^{2}\right)+\gamma_{t}\eta_{t}(\Omega_{t}(\bz_{t+1})-\Omega_{t}(\bx_{t})),\label{eq:last-2}
\end{align}
where the second inequality is by $\left\Vert \bz_{t+1}-\bx_{t}\right\Vert ^{2}\leq\frac{v_{t}}{v_{t+1}}\left\Vert \bz_{t}-\bx_{t}\right\Vert ^{2}+\left(1-\frac{v_{t}}{v_{t+1}}\right)\left\Vert \bx_{t}-\bx_{t}\right\Vert ^{2}=\frac{v_{t}}{v_{t+1}}\left\Vert \bz_{t}-\bx_{t}\right\Vert ^{2}$
due to the convexity of $\left\Vert \cdot\right\Vert ^{2}$ and (\ref{eq:last-z2}).
Multiply both sides of (\ref{eq:last-2}) by $v_{t+1}$ and sum up
from $t=1$ to $T$ to obtain
\begin{align}
 & \sum_{t=1}^{T}\gamma_{t}\eta_{t}v_{t+1}(F(\bx_{t+1})-F(\bz_{t+1}))\nonumber \\
\leq & \frac{\gamma_{1}v_{1}\left\Vert \bz_{1}-\bx_{1}\right\Vert ^{2}-\gamma_{T+1}v_{T+1}\left\Vert \bz_{T+1}-\bx_{T+1}\right\Vert ^{2}}{2}+\sum_{t=1}^{T}\gamma_{t}\eta_{t}^{2}v_{t+1}\left(G_{\I(t)}^{2}+G_{f,1}^{2}\right)+\gamma_{t}\eta_{t}v_{t+1}(\Omega_{t}(\bz_{t+1})-\Omega_{t}(\bx_{t}))\nonumber \\
\leq & \frac{\gamma_{1}v_{1}\left\Vert \bz_{1}-\bx_{1}\right\Vert ^{2}}{2}+\sum_{t=1}^{T}\gamma_{t}\eta_{t}^{2}v_{t+1}\left(G_{\I(t)}^{2}+G_{f,1}^{2}\right)+\gamma_{t}\eta_{t}v_{t+1}(\Omega_{t}(\bz_{t+1})-\Omega_{t}(\bx_{t}))\nonumber \\
\overset{\eqref{eq:last-z1}}{=} & \frac{\gamma_{1}v_{1}\left\Vert \bz-\bx_{1}\right\Vert ^{2}}{2}+\sum_{t=1}^{T}\gamma_{t}\eta_{t}^{2}v_{t+1}\left(G_{\I(t)}^{2}+G_{f,1}^{2}\right)+\gamma_{t}\eta_{t}v_{t+1}(\Omega_{t}(\bz_{t+1})-\Omega_{t}(\bx_{t})).\label{eq:last-3}
\end{align}
Now we focus on the term $\sum_{t=1}^{T}\gamma_{t}\eta_{t}v_{t+1}\Omega_{t}(\bz_{t+1})$
and observe that
\begin{align}
\sum_{t=1}^{T}\gamma_{t}\eta_{t}v_{t+1}\Omega_{t}(\bz_{t+1})= & \sum_{t=1}^{T}\gamma_{t}\eta_{t}v_{t+1}(f_{\I(t)}(\bz_{t+1})-f(\bz_{t+1}))\nonumber \\
\overset{(a)}{\leq} & \sum_{t=1}^{T}\gamma_{t}\eta_{t}\left[v_{1}f_{\I(t)}(\bz)+\sum_{s=1}^{t}(v_{s+1}-v_{s})f_{\I(t)}(\bx_{s})-v_{t+1}f(\bz_{t+1})\right]\nonumber \\
\overset{(b)}{=} & \sum_{t=1}^{T}\gamma_{t}\eta_{t}\left[v_{1}(\Omega_{t}(\bz)+f(\bz))+\sum_{s=1}^{t}(v_{s+1}-v_{s})(\Omega_{t}(\bx_{s})+f(\bx_{s}))-v_{t+1}f(\bz_{t+1})\right]\nonumber \\
= & \sum_{t=1}^{T}\gamma_{t}\eta_{t}\left[v_{1}\Omega_{t}(\bz)+\sum_{s=1}^{t}(v_{s+1}-v_{s})\Omega_{t}(\bx_{s})\right]\nonumber \\
 & +\sum_{t=1}^{T}\gamma_{t}\eta_{t}\left[v_{1}f(\bz)+\sum_{s=1}^{t}(v_{s+1}-v_{s})f(\bx_{s})-v_{t+1}f(\bz_{t+1})\right],\label{eq:last-4}
\end{align}
where $(a)$ is due to $f_{\I(t)}(\bz_{t+1})\leq\frac{v_{1}}{v_{t+1}}f_{\I(t)}(\bz)+\sum_{s=1}^{t}\frac{v_{s+1}-v_{s}}{v_{t+1}}f_{\I(t)}(\bx_{s})$
by the convexity of $f_{\I(t)}$ and the definition of $\bz_{t+1}$
in (\ref{eq:last-z1}) and $(b)$ holds by recalling that $\Omega_{t}(\cdot)=f_{\I(t)}(\cdot)-f(\cdot)$. 

Combine (\ref{eq:last-3}) and (\ref{eq:last-4}) to have
\begin{equation}
\sum_{t=1}^{T}\gamma_{t}\eta_{t}v_{t+1}(F(\bx_{t+1})-F(\bz_{t+1}))\leq\frac{\gamma_{1}v_{1}\left\Vert \bz-\bx_{1}\right\Vert ^{2}}{2}+\sum_{t=1}^{T}\gamma_{t}\eta_{t}^{2}v_{t+1}\left(G_{\I(t)}^{2}+G_{f,1}^{2}\right)+\square+\blacksquare,\label{eq:last-5}
\end{equation}
where 
\begin{align}
\square & \defeq\sum_{t=1}^{T}\gamma_{t}\eta_{t}\left[v_{1}\Omega_{t}(\bz)+\sum_{s=1}^{t}(v_{s+1}-v_{s})\Omega_{t}(\bx_{s})-v_{t+1}\Omega_{t}(\bx_{t})\right],\label{eq:last-white}\\
\blacksquare & \defeq\sum_{t=1}^{T}\gamma_{t}\eta_{t}\left[v_{1}f(\bz)+\sum_{s=1}^{t}(v_{s+1}-v_{s})f(\bx_{s})-v_{t+1}f(\bz_{t+1})\right].\label{eq:last-black}
\end{align}
We bound term $\blacksquare$ as follows
\begin{align}
\blacksquare= & \sum_{t=1}^{T}\gamma_{t}\eta_{t}\left[v_{1}(F(\bz)-\psi(z))+\sum_{s=1}^{t}(v_{s+1}-v_{s})(F(\bx_{s})-\psi(\bx_{s}))-v_{t+1}(F(\bz_{t+1})-\psi(\bz_{t+1}))\right]\nonumber \\
= & \sum_{t=1}^{T}\gamma_{t}\eta_{t}\left[v_{1}F(\bz)+\sum_{s=1}^{t}(v_{s+1}-v_{s})F(\bx_{s})-v_{t+1}F(\bz_{t+1})+v_{t+1}\psi(\bz_{t+1})-\sum_{s=1}^{t}(v_{s+1}-v_{s})\psi(\bx_{s})-v_{1}\psi(\bz)\right]\nonumber \\
\overset{(c)}{\leq} & \sum_{t=1}^{T}\gamma_{t}\eta_{t}\left[v_{1}F(\bz)+\sum_{s=1}^{t}(v_{s+1}-v_{s})F(\bx_{s})-v_{t+1}F(\bz_{t+1})\right]\nonumber \\
= & \sum_{t=1}^{T}\gamma_{t}\eta_{t}\left[\sum_{s=1}^{t}(v_{s+1}-v_{s})(F(\bx_{s})-F(\bz))-v_{t+1}(F(\bz_{t+1})-F(\bz))\right]\nonumber \\
= & \sum_{s=1}^{T}\left(\sum_{t=s}^{T}\gamma_{t}\eta_{t}\right)(v_{s+1}-v_{s})(F(\bx_{s})-F(\bz))-\sum_{t=1}^{T}\gamma_{t}\eta_{t}v_{t+1}(F(\bz_{t+1})-F(\bz))\nonumber \\
\overset{(d)}{=} & \sum_{s=2}^{T}\gamma_{s-1}\eta_{s-1}v_{s}(F(\bx_{s})-F(\bz))-\sum_{t=1}^{T}\gamma_{t}\eta_{t}v_{t+1}(F(\bz_{t+1})-F(\bz)),\label{eq:last-6}
\end{align}
where $(c)$ holds by $\psi(\bz_{t+1})\leq\frac{v_{1}}{v_{t+1}}\psi(\bz)+\sum_{s=1}^{t}\frac{v_{s+1}-v_{s}}{v_{t+1}}\psi(\bx_{s})$
due to the convexity of $\psi$ and the definition of $\bz_{t+1}$
in (\ref{eq:last-z1}) and $(d)$ is by noticing that
\begin{itemize}
\item if $s=1$, 
\begin{equation}
v_{2}-v_{1}\overset{\eqref{eq:last-v2}}{=}0;\label{eq:last-simplify-v1}
\end{equation}
\item if $s\in\left[T\right]\backslash\left[1\right]$,
\[
v_{s+1}-v_{s}\overset{\eqref{eq:last-v1}}{=}\frac{\gamma_{T}\eta_{T}}{\sum_{t=s}^{T}\gamma_{t}\eta_{t}}-\frac{\gamma_{T}\eta_{T}}{\sum_{t=s-1}^{T}\gamma_{t}\eta_{t}}=\frac{\gamma_{T}\eta_{T}\gamma_{s-1}\eta_{s-1}}{(\sum_{t=s}^{T}\gamma_{t}\eta_{t})(\sum_{t=s-1}^{T}\gamma_{t}\eta_{t})}\overset{\eqref{eq:last-v1}}{=}\frac{\gamma_{s-1}\eta_{s-1}v_{s}}{\sum_{t=s}^{T}\gamma_{t}\eta_{t}},
\]
which implies
\begin{equation}
\left(\sum_{t=s}^{T}\gamma_{t}\eta_{t}\right)(v_{s+1}-v_{s})=\gamma_{s-1}\eta_{s-1}v_{s}.\label{eq:last-simplify-v2}
\end{equation}
\end{itemize}
We then plug (\ref{eq:last-6}) back into (\ref{eq:last-5}) and rearrange
terms to get
\[
\gamma_{T}\eta_{T}v_{T+1}\left(F(\bx_{T+1})-F(\bz)\right)\leq\frac{\gamma_{1}v_{1}\left\Vert \bz-\bx_{1}\right\Vert ^{2}}{2}+\sum_{t=1}^{T}\gamma_{t}\eta_{t}^{2}v_{t+1}\left(G_{\I(t)}^{2}+G_{f,1}^{2}\right)+\square.
\]
Take expectations and divide by $\gamma_{T}\eta_{T}$ on both sides
and then use $v_{T+1}\overset{\eqref{eq:last-v1}}{=}1$ and $\gamma_{1}=1$
to obtain
\begin{align}
\E\left[F(\bx_{T+1})-F(\bz)\right] & \leq\frac{v_{1}\left\Vert \bz-\bx_{1}\right\Vert ^{2}}{2\gamma_{T}\eta_{T}}+\sum_{t=1}^{T}\frac{\gamma_{t}\eta_{t}^{2}v_{t+1}}{\gamma_{T}\eta_{T}}\left(\E\left[G_{\I(t)}^{2}\right]+G_{f,1}^{2}\right)+\frac{\E\left[\square\right]}{\gamma_{T}\eta_{T}}\nonumber \\
 & \leq\frac{v_{1}\left\Vert \bz-\bx_{1}\right\Vert ^{2}}{2\gamma_{T}\eta_{T}}+2G_{f,2}^{2}\sum_{t=1}^{T}\frac{\gamma_{t}\eta_{t}^{2}v_{t+1}}{\gamma_{T}\eta_{T}}+\frac{\E\left[\square\right]}{\gamma_{T}\eta_{T}},\label{eq:last-7}
\end{align}
where the last line holds due to $\I(t)\diseq\mathrm{Uniform}\left[n\right],\forall t\in\left[T\right]\Rightarrow\E\left[G_{\I(t)}^{2}\right]=\frac{1}{n}\sum_{i=1}^{n}G_{i}^{2}=G_{f,2}^{2}$
and $G_{f,1}\leq G_{f,2}$. 

Lastly, we note that
\begin{align}
\E\left[\square\right] & \overset{\eqref{eq:last-white}}{=}\sum_{t=1}^{T}\gamma_{t}\eta_{t}\left[v_{1}\E\left[\Omega_{t}(\bz)\right]+\sum_{s=1}^{t}(v_{s+1}-v_{s})\E\left[\Omega_{t}(\bx_{s})\right]-v_{t+1}\E\left[\Omega_{t}(\bx_{t})\right]\right]\nonumber \\
 & \overset{(e)}{=}\sum_{t=1}^{T}\gamma_{t}\eta_{t}\left[\sum_{s=1}^{t}(v_{s+1}-v_{s})\E\left[\Omega_{t}(\bx_{s})\right]-v_{t+1}\E\left[\Omega_{t}(\bx_{t})\right]\right]\nonumber \\
 & \overset{(f)}{\leq}\sum_{t=1}^{T}\gamma_{t}\eta_{t}\left[\sum_{s=1}^{t}(v_{s+1}-v_{s})\cdot\Phi\eta_{s}+v_{t+1}\cdot\Phi\eta_{t}\right]\nonumber \\
 & =\Phi\left[\sum_{s=1}^{T}\left(\sum_{t=s}^{T}\gamma_{t}\eta_{t}\right)(v_{s+1}-v_{s})\eta_{s}+\sum_{t=1}^{T}\gamma_{t}\eta_{t}^{2}v_{t+1}\right]\nonumber \\
 & \overset{\eqref{eq:last-simplify-v1},\eqref{eq:last-simplify-v2}}{=}\Phi\left[\sum_{s=2}^{T}\gamma_{s-1}\eta_{s-1}v_{s}\eta_{s}+\sum_{t=1}^{T}\gamma_{t}\eta_{t}^{2}v_{t+1}\right]\label{eq:last-8}\\
 & \overset{(g)}{\leq}2\Phi\sum_{t=1}^{T}\gamma_{t}\eta_{t}^{2}v_{t+1},\label{eq:last-9}
\end{align}
where $(e)$ is due to $\E\left[\Omega_{t}(\bz)\right]=\E\left[f_{\I(t)}(\bz)-f(\bz)\right]=0,\forall t\in\left[T\right]$
since $\I(t)\diseq\mathrm{Uniform}\left[n\right],\forall t\in\left[T\right]$,
$(f)$ holds by the condition of $\left|\E\left[\Omega_{t}(\bx_{s})\right]\right|\leq\Phi\eta_{s},\forall t\in\left[T\right],s\in\left[t\right]$,
and $(g)$ is by the condition of the non-increasing $\eta_{t},\forall t\in\left[T\right]$.
We thereby finish the proof by plugging (\ref{eq:last-9}) back into
(\ref{eq:last-7}) and applying the definition of $v_{t}$ in (\ref{eq:last-v1})
and (\ref{eq:last-v2}).
\end{proof}

Though Lemma \ref{lem:last} is enough for the $\RR$ case, as the
reader already has seen, it is inadequate for the $\SS$ case. So
what do we miss? The problem is that we never use the fact $\left\{ \I((k-1)n+1),\cdots,\I(kn)\right\} =\left[n\right],\forall k\in\left[K\right]$
if $T=Kn$ for $K\in\N$, which is the major factor employed in the
analysis of \cite{pmlr-v235-liu24cg}. So similar to \cite{pmlr-v235-liu24cg},
we may also hope for a last-iterate convergence bound that holds almost
surely. However, due to a different structure in our Algorithm \ref{alg:Alg},
i.e., the place for the proximal update, the existing analysis and
rates in \cite{pmlr-v235-liu24cg} are immediately invalid. Hence,
it is unknown whether Algorithm \ref{alg:Alg} still guarantees a
last-iterate rate in the same order as \cite{pmlr-v235-liu24cg} (or
equivalently, Proximal GD).

Fortunately, we can still prove a similar almost surely convergence
bound with two extra requirements Conditions \ref{enu:special-refined-condition-2}
and \ref{enu:special-refined-condition-3}. Condition \ref{enu:special-refined-condition-2}
here is standard if we consider the whole epoch as a single update
like \cite{pmlr-v235-liu24cg}. In contrast, Condition \ref{enu:special-refined-condition-3}
is more artificial and makes the rate depend on a new parameter $G_{\psi}$.
We currently do not know whether it can be removed or not.

Lastly, we point out that, if ignoring $G_{\psi}$ (or one can think
it is smaller than $G_{f,1}$), the rate in Lemma \ref{lem:special-refined}
for the general convex case is in the same order as \cite{pmlr-v235-liu24cg}.
This is why we can obtain the improved result in Theorem \ref{thm:SS-cvx-improved-full}.
However, for the strongly convex case, Lemma \ref{lem:special-refined}
is worse than \cite{pmlr-v235-liu24cg} because the parameter $\mu$
in our $\widetilde{\gamma}_{k}$ is worse than a factor of $n$ compared
to \cite{pmlr-v235-liu24cg} (in other words, they have $n\mu$).
This crucial point leads us to a worse rate $\widetilde{\O}\left(\frac{n}{K}\right)$
(e.g., setting $\widetilde{\eta}_{k}=\frac{2}{\mu k},\forall k\in\left[K\right]$)
than Proximal GD. We currently do not know whether it is fixable.
If it is, one can improve Theorem \ref{thm:SS-str-full} to show $\SS$
also beats Proximal GD for any $K\in\N$ (again, ignoring $G_{\psi}$).
\begin{lem}
\label{lem:special-refined}Under Assumptions \ref{assu:basic} and
\ref{assu:lip}, suppose $T=Kn$ where $K\in\N$ and the following
three conditions hold:
\begin{enumerate}
\item \label{enu:special-refined-condition-1}The index satisfies $\left\{ \I((k-1)n+1),\cdots,\I(kn)\right\} =\left[n\right],\forall k\in\left[K\right]$.
\item \label{enu:special-refined-condition-2}The stepsize satisfies $\eta_{t}=\widetilde{\eta}_{\qu(t)},\forall t\in\left[T\right]$
where $\widetilde{\eta}_{k},\forall k\in\left[K\right]$ is another
positive sequence.
\item \label{enu:special-refined-condition-3}$\psi=\varphi+\charf_{\domx}$
where $\varphi:\R^{d}\to\R$ is convex and $G_{\psi}$-Lipschitz on
$\dom\psi=\domx$.
\end{enumerate}
Then for any $\bz\in\R^{d}$, Algorithm \ref{alg:Alg} guarantees
almost surely
\[
F(\bx_{Kn+1})-F(\bz)\leq\O\left(\frac{\left\Vert \bz-\bx_{1}\right\Vert ^{2}}{n\sum_{k=1}^{K}\widetilde{\gamma}_{k}\widetilde{\eta}_{k}}+n\left(G_{f,1}^{2}+G_{\psi}^{2}\right)\sum_{k=1}^{K}\frac{\widetilde{\gamma}_{k}\widetilde{\eta}_{k}^{2}}{\sum_{\ell=k}^{K}\widetilde{\gamma}_{\ell}\widetilde{\eta}_{\ell}}\right),
\]
where $\widetilde{\gamma}_{k}\defeq\prod_{\ell=1}^{k-1}(1+\mu\widetilde{\eta}_{\ell}),\forall k\in\left[K+1\right]$.
\end{lem}

\begin{proof}
Given $\by\in\dom\psi=\domx$, for any $t\in\left[T\right]$, we have
by (\ref{eq:core-2})
\[
\left\langle \nabla f_{\I(t)}(\bx_{t}),\bx_{t+1}-\by\right\rangle \leq\frac{\left\Vert \by-\bx_{t}\right\Vert ^{2}}{2\eta_{t}}-(1+\mu\eta_{t})\frac{\left\Vert \by-\bx_{t+1}\right\Vert ^{2}}{2\eta_{t}}-\frac{\left\Vert \bx_{t+1}-\bx_{t}\right\Vert ^{2}}{2\eta_{t}}+\psi(\by)-\psi(\bx_{t+1}).
\]
Combine the above inequality with (\ref{eq:core-1}) to obtain
\begin{align}
 & f_{\I(t)}(\bx_{t})+\psi(\bx_{t+1})-f_{\I(t)}(\by)-\psi(\by)\nonumber \\
\leq & \left\langle \nabla f_{\I(t)}(\bx_{t}),\bx_{t}-\bx_{t+1}\right\rangle +\frac{\left\Vert \by-\bx_{t}\right\Vert ^{2}}{2\eta_{t}}-(1+\mu\eta_{t})\frac{\left\Vert \by-\bx_{t+1}\right\Vert ^{2}}{2\eta_{t}}-\frac{\left\Vert \bx_{t+1}-\bx_{t}\right\Vert ^{2}}{2\eta_{t}}\nonumber \\
\leq & \frac{\left\Vert \by-\bx_{t}\right\Vert ^{2}}{2\eta_{t}}-(1+\mu\eta_{t})\frac{\left\Vert \by-\bx_{t+1}\right\Vert ^{2}}{2\eta_{t}}+\frac{\eta_{t}G_{\I(t)}^{2}}{2},\label{eq:special-refined-1}
\end{align}
where the second step is by applying Cauchy-Schwarz inequality, Assumption
\ref{assu:lip} and AM-GM inequality to have
\[
\left\langle \nabla f_{\I(t)}(\bx_{t}),\bx_{t}-\bx_{t+1}\right\rangle \leq\left\Vert \nabla f_{\I(t)}(\bx_{t})\right\Vert \left\Vert \bx_{t}-\bx_{t+1}\right\Vert \leq G_{\I(t)}\left\Vert \bx_{t}-\bx_{t+1}\right\Vert \leq\frac{\eta_{t}G_{\I(t)}^{2}}{2}+\frac{\left\Vert \bx_{t+1}-\bx_{t}\right\Vert ^{2}}{2\eta_{t}}.
\]

We sum up (\ref{eq:special-refined-1}) from $t=(k-1)n+1$ to $kn$
for a fixed $k\in\left[K\right]$ and notice $\eta_{t}=\widetilde{\eta}_{k}$
by Condition \ref{enu:special-refined-condition-2} to obtain
\begin{align*}
 & \sum_{t=(k-1)n+1}^{kn}f_{\I(t)}(\bx_{t})+\psi(\bx_{t+1})-f_{\I(t)}(\by)-\psi(\by)\\
\leq & \frac{\left\Vert \by-\bx_{(k-1)n+1}\right\Vert ^{2}-(1+\mu\widetilde{\eta}_{k})\left\Vert \by-\bx_{kn+1}\right\Vert ^{2}}{2\widetilde{\eta}_{k}}+\frac{\widetilde{\eta}_{k}}{2}\sum_{t=(k-1)n+1}^{kn}G_{\I(t)}^{2}.
\end{align*}
One more step, by Condition \ref{enu:special-refined-condition-1},
we can rewrite the above inequality into
\begin{align}
F(\bx_{kn+1})-F(\by) & \leq\frac{\left\Vert \by-\bx_{(k-1)n+1}\right\Vert ^{2}-(1+\mu\widetilde{\eta}_{k})\left\Vert \by-\bx_{kn+1}\right\Vert ^{2}}{2n\widetilde{\eta}_{k}}+\frac{\widetilde{\eta}_{k}G_{f,2}^{2}}{2}\nonumber \\
 & +\frac{1}{n}\sum_{t=(k-1)n+1}^{kn}f_{\I(t)}(\bx_{kn+1})-f_{\I(t)}(\bx_{t})+\psi(\bx_{kn+1})-\psi(\bx_{t+1}).\label{eq:special-refined-2}
\end{align}

We now prove bound on $\left\Vert \bx_{s+1}-\bx_{s}\right\Vert $,
which will be used later. Observe that when $\psi=\varphi+\charf_{\domx}$,
there is
\begin{align*}
\bx_{s+1} & =\argmin_{\bx\in\R^{d}}\psi(\bx)+\left\langle \nabla f_{\I(s)}(\bx_{s}),\bx\right\rangle +\frac{\left\Vert \bx-\bx_{s}\right\Vert ^{2}}{2\eta_{s}}\\
 & =\argmin_{\bx\in\domx}\varphi(\bx)+\left\langle \nabla f_{\I(s)}(\bx_{s}),\bx\right\rangle +\frac{\left\Vert \bx-\bx_{s}\right\Vert ^{2}}{2\eta_{s}},
\end{align*}
which implies that, by the optimality condition, there exists $\nabla\varphi(\bx_{s+1})\in\partial\varphi(\bx_{s+1})$
such that for any $\bz\in\domx$,
\[
\left\langle \nabla\varphi(\bx_{s+1})+\nabla f_{\I(s)}(\bx_{s})+\frac{\bx_{s+1}-\bx_{s}}{\eta_{s}},\bx_{s+1}-\bz\right\rangle \leq0.
\]
In particular, set $\bz=\bx_{s}$ to have
\begin{align}
\left\Vert \bx_{s+1}-\bx_{s}\right\Vert ^{2} & \leq\eta_{s}\left\langle \nabla\varphi(\bx_{s+1})+\nabla f_{\I(s)}(\bx_{s}),\bx_{s}-\bx_{s+1}\right\rangle \nonumber \\
\Rightarrow\left\Vert \bx_{s+1}-\bx_{s}\right\Vert  & \leq\eta_{s}\left\Vert \nabla\varphi(\bx_{s+1})+\nabla f_{\I(s)}(\bx_{s})\right\Vert \leq\widetilde{\eta}_{k}\left(G_{\psi}+G_{\I(s)}\right),\label{eq:special-refined-3}
\end{align}
where in the last step we use $\eta_{s}=\widetilde{\eta}_{k}$ when
$s\in\left\{ (k-1)n+1,\cdots,kn\right\} $ by Condition \ref{enu:special-refined-condition-2}
and $\varphi$ is $G_{\psi}$-Lipschitz on $\domx$ by Condition \ref{enu:special-refined-condition-3}.

Now notice that
\begin{align}
 & \sum_{t=(k-1)n+1}^{kn}f_{\I(t)}(\bx_{kn+1})-f_{\I(t)}(\bx_{t})\nonumber \\
\overset{(a)}{\leq} & \sum_{t=(k-1)n+1}^{kn}G_{\I(t)}\left\Vert \bx_{kn+1}-\bx_{t}\right\Vert \leq\sum_{t=(k-1)n+1}^{kn}\sum_{s=t}^{kn}G_{\I(t)}\left\Vert \bx_{s+1}-\bx_{s}\right\Vert \nonumber \\
\overset{\eqref{eq:special-refined-3}}{\leq} & \sum_{t=(k-1)n+1}^{kn}\sum_{s=t}^{kn}\widetilde{\eta}_{k}G_{\I(t)}\left(G_{\psi}+G_{\I(s)}\right)\leq\widetilde{\eta}_{k}\left(n^{2}G_{f,1}G_{\psi}+\frac{n^{2}G_{f,1}^{2}-nG_{f,2}^{2}}{2}\right),\label{eq:special-refined-4}
\end{align}
where $(a)$ is by Assumption \ref{assu:lip}. Moreover, there is
\begin{align}
 & \sum_{t=(k-1)n+1}^{kn}\psi(\bx_{kn+1})-\psi(\bx_{t+1})\nonumber \\
= & \sum_{t=(k-1)n+1}^{kn}\varphi(\bx_{kn+1})-\varphi(\bx_{t+1})=\sum_{t=(k-1)n+1}^{kn-1}\varphi(\bx_{kn+1})-\varphi(\bx_{t+1})\nonumber \\
\overset{(b)}{\leq} & G_{\psi}\sum_{t=(k-1)n+1}^{kn-1}\left\Vert \bx_{kn+1}-\bx_{t+1}\right\Vert \leq G_{\psi}\sum_{t=(k-1)n+1}^{kn-1}\sum_{s=t+1}^{kn}\left\Vert \bx_{s+1}-\bx_{s}\right\Vert \nonumber \\
\overset{\eqref{eq:special-refined-3}}{\leq} & G_{\psi}\sum_{t=(k-1)n+1}^{kn-1}\sum_{s=t+1}^{kn}\widetilde{\eta}_{k}\left(G_{\psi}+G_{\I(s)}\right)\leq\widetilde{\eta}_{k}n(n-1)\left(\frac{G_{\psi}^{2}}{2}+G_{\psi}G_{f,1}\right),\label{eq:special-refined-5}
\end{align}
where $(b)$ is by Condition \ref{enu:special-refined-condition-3}.

Combine (\ref{eq:special-refined-2}), (\ref{eq:special-refined-4})
and (\ref{eq:special-refined-5}) to have for any $\by\in\dom\psi=\domx$
and $k\in\left[K\right]$,
\[
F(\bx_{kn+1})-F(\by)\leq\frac{\left\Vert \by-\bx_{(k-1)n+1}\right\Vert ^{2}-(1+\mu\widetilde{\eta}_{k})\left\Vert \by-\bx_{kn+1}\right\Vert ^{2}}{2n\widetilde{\eta}_{k}}+\widetilde{\eta}_{k}n\left(G_{f,1}+G_{\psi}\right)^{2}.
\]
Note that this inequality also holds for $\by\notin\domx$. Hence,
the above result is true for any $\by\in\R^{d}$. We can then follow
similar steps of proving Lemma \ref{lem:last} to finally obtain for
any $\bz\in\R^{d}$,
\begin{align*}
F(\bx_{kn+1})-F(\bz) & \leq\O\left(\frac{\left\Vert \bz-\bx_{1}\right\Vert ^{2}}{n\sum_{k=1}^{K}\widetilde{\gamma}_{k}\widetilde{\eta}_{k}}+n\left(G_{f,1}+G_{\psi}\right)^{2}\sum_{k=1}^{K}\frac{\widetilde{\gamma}_{k}\widetilde{\eta}_{k}^{2}}{\sum_{\ell=k}^{K}\widetilde{\gamma}_{\ell}\widetilde{\eta}_{\ell}}\right)\\
 & =\O\left(\frac{\left\Vert \bz-\bx_{1}\right\Vert ^{2}}{n\sum_{k=1}^{K}\widetilde{\gamma}_{k}\widetilde{\eta}_{k}}+n\left(G_{f,1}^{2}+G_{\psi}^{2}\right)\sum_{k=1}^{K}\frac{\widetilde{\gamma}_{k}\widetilde{\eta}_{k}^{2}}{\sum_{\ell=k}^{K}\widetilde{\gamma}_{\ell}\widetilde{\eta}_{\ell}}\right),
\end{align*}
where $\widetilde{\gamma}_{k}\defeq\prod_{\ell=1}^{k-1}(1+\mu\widetilde{\eta}_{\ell}),\forall k\in\left[K+1\right]$.
\end{proof}

\subsection{Analysis for the $\protect\RR$ Sampling Scheme\label{subsec:RR-analysis}}

Due to Lemma \ref{lem:last}, our task reduces to bound $\left|\E\left[\Omega_{t}(\bx_{s})\right]\right|,\forall t\in\left[T\right],s\in\left[t\right]$.
In this subsection, we will show how to bound it under the $\RR$
sampling scheme. The main idea is inspired by \cite{NEURIPS2020_2e2c4bf7,NEURIPS2021_107030ca,NEURIPS2022_7bc4f74e},
where they only bound $\left|\E\left[\Omega_{t}(\bx_{t})\right]\right|,\forall t\in\left[T\right]$
in a simpler setting, i.e., $\psi=\charf_{\domx}$, $G_{i}\equiv G$,
and constant stepsize. Here, we provide a fine-grained analysis for
$\left|\E\left[\Omega_{t}(\bx_{s})\right]\right|,\forall t\in\left[T\right],s\in\left[t\right]$
working for the broader setting considered in our paper. As mentioned,
the finer dependence on $G_{f,1}$ and $G_{f,2}$ is the key to help
us obtain improved results over \cite{pmlr-v235-liu24cg}.
\begin{lem}
\label{lem:RR-Omega}Under Assumptions \ref{assu:basic} and \ref{assu:lip},
suppose the $\RR$ sampling scheme is employed, then for any $t\in\left[T\right]$
and $s\in\left[t\right]$, let $\qu\defeq\qu(t)$, Algorithm \ref{alg:Alg}
guarantees
\begin{itemize}
\item $\left|\E\left[\Omega_{t}(\bx_{s})\right]\right|=0$ if $s\in\left[(\qu-1)n\right]$;
\item $\left|\E\left[\Omega_{t}(\bx_{s})\right]\right|\leq\frac{\sqrt{2}G_{f,2}^{2}}{n}\sum_{j=(\qu-1)n+1}^{s-1}\frac{\gamma_{j}\eta_{j}}{\gamma_{s}}+\frac{2\sqrt{2}G_{f,1}G_{f,2}}{n}\sum_{i=(\qu-1)n+1}^{s-1}\sqrt{\sum_{j=i}^{s-1}\frac{\gamma_{j}^{2}\eta_{j}^{2}}{\gamma_{s}^{2}}}$
if $s\in\left[t\right]\backslash\left[(\qu-1)n\right]$;
\end{itemize}
where $\Omega_{t}(\cdot)$ and $\gamma_{t}$ are defined in Lemmas
\ref{lem:core} and \ref{lem:last}, respectively.
\end{lem}

\begin{proof}
Let $\F_{t}\defeq\sigma(\I(1),\cdots,\I(t)),\forall t\in\left[T\right]$
denote the natural filtration and $\F_{0}$ be the trivial $\sigma$-algebra.
Note that $\bx_{t}\in\F_{t-1},\forall t\in\left[T\right]$. Additionally,
we use $\re\defeq\re(s)$ in the following.
\begin{itemize}
\item Given $s\in\left[(\qu-1)n\right]$, we know $\bx_{s}\in\F_{(\qu-1)n}$
and $\I(t)\diseq\mathrm{Uniform}\left[n\right]$ conditioned on $\F_{(\qu-1)n}$
under the $\RR$ sampling scheme. Therefore, $\E\left[f_{\I(t)}(\bx_{s})\mid\F_{(\qu-1)n}\right]=\frac{1}{n}\sum_{i=1}^{n}f_{i}(\bx_{s})=f(\bx_{s})$,
which implies
\[
\E\left[\Omega_{t}(\bx_{s})\right]=\E\left[f_{\I(t)}(\bx_{s})-f(\bx_{s})\right]=\E\left[\E\left[f_{\I(t)}(\bx_{s})\mid\F_{(\qu-1)n}\right]-f(\bx_{s})\right]=0.
\]
\item Given $s\in\left[t\right]\backslash\left[(\qu-1)n\right]$, we know
$\I(t)$ and $\I(s)$ have the same distribution conditioned on $\F_{s-1}$
under the $\RR$ sampling scheme. Hence, there is
\begin{equation}
\E\left[f_{\I(t)}(\bx_{s})\mid\F_{s-1}\right]=\E\left[f_{\I(s)}(\bx_{s})\mid\F_{s-1}\right]\Rightarrow\E\left[f_{\I(t)}(\bx_{s})\right]=\E\left[f_{\I(s)}(\bx_{s})\right]\Rightarrow\E\left[\Omega_{t}(\bx_{s})\right]=\E\left[\Omega_{s}(\bx_{s})\right].\label{eq:RR-Omega-equiv}
\end{equation}
Note that $\qu(s)=\qu$ when $s\in\left[t\right]\backslash\left[(\qu-1)n\right]$.
Thus, the index $\I(s)$ satisfies
\begin{equation}
\I(s)=\pi_{\qu(s)}^{\re(s)}=\pi_{\qu}^{\re}.\label{eq:RR-Omega-index}
\end{equation}
Therefore, we know
\begin{align}
f(\bx_{s}) & =\frac{1}{n}\sum_{i=1}^{n}f_{i}(\bx_{s})=\frac{1}{n}\sum_{i=1}^{\re-1}f_{\pi_{\qu}^{i}}(\bx_{s})+\frac{1}{n}\sum_{i=\re}^{n}f_{\pi_{\qu}^{i}}(\bx_{s})\nonumber \\
 & =\frac{1}{n}\sum_{i=1}^{\re-1}f_{\pi_{\qu}^{i}}(\bx_{s})+\frac{n-\re+1}{n}\E\left[f_{\pi_{\qu}^{\re}}(\bx_{s})\mid\F_{s-1}\right]\nonumber \\
 & \overset{\eqref{eq:RR-Omega-index}}{=}\frac{1}{n}\sum_{i=1}^{\re-1}f_{\pi_{\qu}^{i}}(\bx_{s})+\frac{n-\re+1}{n}\E\left[f_{\I(s)}(\bx_{s})\mid\F_{s-1}\right]\nonumber \\
\Rightarrow\E\left[f(\bx_{s})\right] & =\frac{1}{n}\sum_{i=1}^{\re-1}\E\left[f_{\pi_{\qu}^{i}}(\bx_{s})\right]+\frac{n-\re+1}{n}\E\left[f_{\I(s)}(\bx_{s})\right]\nonumber \\
\Rightarrow\E\left[\Omega_{s}(\bx_{s})\right] & =\frac{1}{n}\sum_{i=1}^{\re-1}\E\left[f_{\I(s)}(\bx_{s})-f_{\pi_{\qu}^{i}}(\bx_{s})\right]\overset{\eqref{eq:RR-Omega-index}}{=}\frac{1}{n}\sum_{i=1}^{\re-1}\E\left[f_{\pi_{\qu}^{\re}}(\bx_{s})-f_{\pi_{\qu}^{i}}(\bx_{s})\right].\label{eq:RR-Omega-1}
\end{align}
Now for any fixed $i\in\left[\re-1\right]$, we introduce $\widehat{\pi}_{\qu}(\re,i)$,
which is generated by exchanging $\pi_{\qu}^{\re}$ and $\pi_{\qu}^{i}$
in $\pi_{\qu}$, i.e.,
\begin{equation}
\widehat{\pi}_{\qu}(\re,i)\defeq(\pi_{\qu}^{1},\cdots,\pi_{\qu}^{i-1},\pi_{\qu}^{\re},\pi_{\qu}^{i+1},\cdots,\pi_{\qu}^{\re-1},\pi_{\qu}^{i},\pi_{\qu}^{\re+1},\cdots,\pi_{\qu}^{n}).\label{eq:RR-Omega-hat-pi}
\end{equation}
We then define the following sequence of points,
\begin{align}
\widehat{\bx}_{(\qu-1)n+1}(\re,i) & \defeq\bx_{(\qu-1)n+1},\label{eq:RR-Omega-hat-x-init}\\
\widehat{\bx}_{(\qu-1)n+1+j}(\re,i) & \defeq\argmin_{\bx\in\R^{d}}\psi(\bx)+\left\langle \nabla f_{\widehat{\pi}_{\qu}^{j}(\re,i)}(\widehat{\bx}_{(\qu-1)n+j}(\re,i)),\bx\right\rangle +\frac{\left\Vert \bx-\widehat{\bx}_{(\qu-1)n+j}(\re,i)\right\Vert ^{2}}{2\eta_{(\qu-1)n+j}},\forall j\in\left[\re-1\right].\label{eq:RR-Omega-hat-x}
\end{align}
By Lemma \ref{lem:RR-same-dist}, $\pi_{\qu}\diseq\widehat{\pi}_{\qu}(\re,i)$,
which implies $(\pi_{\qu},\bx_{s})\diseq(\widehat{\pi}_{\qu}(\re,i),\widehat{\bx}_{s}(\re,i))$\footnote{\label{fn:deterministic}Strictly speaking, this equation requires
that for any $i\in\left[n\right]$ and $\bx\in\R^{d}$, $\nabla f_{i}(\bx)$
is deterministically picked from the subgradient set $\partial f_{i}(\bx)$,
which possibly contains more than one element. We assume it holds
since this is realistic.}. Hence,
\[
\E\left[f_{\pi_{\qu}^{i}}(\bx_{s})\right]=\E\left[f_{\widehat{\pi}_{\qu}^{i}(\re,i)}(\widehat{\bx}_{s}(\re,i))\right]=\E\left[f_{\pi_{\qu}^{\re}}(\widehat{\bx}_{s}(\re,i))\right],
\]
which gives us
\begin{align*}
\E\left[\Omega_{t}(\bx_{s})\right] & \overset{\eqref{eq:RR-Omega-equiv}}{=}\E\left[\Omega_{s}(\bx_{s})\right]\overset{\eqref{eq:RR-Omega-1}}{=}\frac{1}{n}\sum_{i=1}^{\re-1}\E\left[f_{\pi_{\qu}^{\re}}(\bx_{s})-f_{\pi_{\qu}^{i}}(\bx_{s})\right]=\frac{1}{n}\sum_{i=1}^{\re-1}\E\left[f_{\pi_{\qu}^{\re}}(\bx_{s})-f_{\pi_{\qu}^{\re}}(\widehat{\bx}_{s}(\re,i))\right]\\
\Rightarrow\left|\E\left[\Omega_{t}(\bx_{s})\right]\right| & \leq\frac{1}{n}\sum_{i=1}^{\re-1}\E\left[\left|f_{\pi_{\qu}^{\re}}(\bx_{s})-f_{\pi_{\qu}^{\re}}(\widehat{\bx}_{s}(\re,i))\right|\right]\overset{(a)}{\leq}\frac{1}{n}\sum_{i=1}^{\re-1}\E\left[G_{\pi_{\qu}^{\re}}\left\Vert \bx_{s}-\widehat{\bx}_{s}(\re,i)\right\Vert \right]\\
 & =\frac{1}{n}\sum_{i=1}^{\re-1}\E\left[G_{\pi_{\qu}^{\re}}\E\left[\left\Vert \bx_{s}-\widehat{\bx}_{s}(\re,i)\right\Vert \mid\pi_{\qu}^{\re}\right]\right]\overset{(b)}{\leq}\frac{1}{n}\sum_{i=1}^{\re-1}\E\left[G_{\pi_{\qu}^{\re}}\sqrt{\E\left[\left\Vert \bx_{s}-\widehat{\bx}_{s}(\re,i)\right\Vert ^{2}\mid\pi_{\qu}^{\re}\right]}\right],
\end{align*}
where $(a)$ is because $f_{\pi_{\qu}^{\re}}$ is $G_{\pi_{\qu}^{\re}}$-Lipschitz
and $(b)$ is due to H\"{o}lder's inequality. Finally, we invoke
Lemma \ref{lem:RR-stability} to have
\begin{align*}
\left|\E\left[\Omega_{t}(\bx_{s})\right]\right| & \leq\frac{1}{n}\sum_{i=1}^{\re-1}\E\left[\sqrt{2}G_{\pi_{\qu}^{\re}}^{2}\frac{\gamma_{(\qu-1)n+i}\eta_{(\qu-1)n+i}}{\gamma_{s}}+2\sqrt{2}G_{\pi_{\qu}^{\re}}G_{f,2}\sqrt{\sum_{j=i}^{\re-1}\frac{\gamma_{(\qu-1)n+j}^{2}\eta_{(\qu-1)n+j}^{2}}{\gamma_{s}^{2}}}\right]\\
 & =\frac{\sqrt{2}G_{f,2}^{2}}{n}\sum_{i=1}^{\re-1}\frac{\gamma_{(\qu-1)n+i}\eta_{(\qu-1)n+i}}{\gamma_{s}}+\frac{2\sqrt{2}G_{f,1}G_{f,2}}{n}\sum_{i=1}^{\re-1}\sqrt{\sum_{j=i}^{\re-1}\frac{\gamma_{(\qu-1)n+j}^{2}\eta_{(\qu-1)n+j}^{2}}{\gamma_{s}^{2}}}\\
 & =\frac{\sqrt{2}G_{f,2}^{2}}{n}\sum_{j=(\qu-1)n+1}^{s-1}\frac{\gamma_{j}\eta_{j}}{\gamma_{s}}+\frac{2\sqrt{2}G_{f,1}G_{f,2}}{n}\sum_{i=(\qu-1)n+1}^{s-1}\sqrt{\sum_{j=i}^{s-1}\frac{\gamma_{j}^{2}\eta_{j}^{2}}{\gamma_{s}^{2}}},
\end{align*}
where we use the fact $(\qu-1)n+\re=s$ in the final step.
\end{itemize}
\end{proof}

\begin{lem}
\label{lem:RR-stability}Under the same settings in Lemma \ref{lem:RR-Omega},
let $\widehat{\bx}_{s}(\re,i)$ be the point defined by (\ref{eq:RR-Omega-hat-x-init})
and (\ref{eq:RR-Omega-hat-x}), then we have
\[
\E\left[\left\Vert \bx_{s}-\widehat{\bx}_{s}(\re,i)\right\Vert ^{2}\mid\pi_{\qu}^{\re}\right]\leq2G_{\pi_{\qu}^{\re}}^{2}\frac{\gamma_{(\qu-1)n+i}^{2}\eta_{(\qu-1)n+i}^{2}}{\gamma_{s}^{2}}+8G_{f,2}^{2}\sum_{j=i}^{\re-1}\frac{\gamma_{(\qu-1)n+j}^{2}\eta_{(\qu-1)n+j}^{2}}{\gamma_{s}^{2}},
\]
where $\gamma_{t}$ is defined in Lemma \ref{lem:last}.
\end{lem}

\begin{proof}
Note that $\pi_{\qu}^{j}=\widehat{\pi}_{\qu}^{j}(\re,i)$ for all
$j\in\left[i-1\right]$ by the definition of $\widehat{\pi}_{\qu}(\re,i)$
(see (\ref{eq:RR-Omega-hat-pi})) and $\bx_{(\qu-1)n+1}=\widehat{\bx}_{(\qu-1)n+1}(\re,i)$
by the definition (see (\ref{eq:RR-Omega-hat-x-init})). Thus, by
the definition of $\widehat{\bx}_{(\qu-1)n+1+j}$ (see (\ref{eq:RR-Omega-hat-x})),
there is
\[
\bx_{(\qu-1)n+j}=\widehat{\bx}_{(\qu-1)n+j}(\re,i),\forall j\in\left[i\right].
\]
In the following, we denote by $\by_{j}\defeq\bx_{(\qu-1)n+j}$ and
$\widehat{\by}_{j}\defeq\widehat{\bx}_{(\qu-1)n+j}(\re,i),\forall j\in\left\{ i,\cdots,\re\right\} $.
Note that there is $\by_{i}=\widehat{\by}_{i}$.

By Lemma \ref{lem:contractive},
\begin{align}
\left\Vert \by_{i+1}-\widehat{\by}_{i+1}\right\Vert  & \leq\frac{\left\Vert \by_{i}-\widehat{\by}_{i}-\eta_{(\qu-1)n+i}(\nabla f_{\pi_{\qu}^{i}}(\by_{i})-\nabla f_{\widehat{\pi}_{\qu}^{i}(\re,i)}(\widehat{\by}_{i}))\right\Vert }{1+\mu\eta_{(\qu-1)n+i}}=\frac{\eta_{(\qu-1)n+i}\left\Vert \nabla f_{\pi_{\qu}^{i}}(\by_{i})-\nabla f_{\widehat{\pi}_{\qu}^{i}(\re,i)}(\by_{i})\right\Vert }{1+\mu\eta_{(\qu-1)n+i}}\nonumber \\
 & \overset{\eqref{eq:RR-Omega-hat-pi}}{=}\frac{\eta_{(\qu-1)n+i}\left\Vert \nabla f_{\pi_{\qu}^{i}}(\by_{i})-\nabla f_{\pi_{\qu}^{\re}}(\by_{i})\right\Vert }{1+\mu\eta_{(\qu-1)n+i}}\leq\frac{\eta_{(\qu-1)n+i}(G_{\pi_{\qu}^{i}}+G_{\pi_{\qu}^{\re}})}{1+\mu\eta_{(\qu-1)n+i}},\label{eq:RR-stability-1}
\end{align}
where the last step is by the Lipschitz property of $f_{i}$. We invoke
Lemma \ref{lem:contractive} again to obtain for any $i+1\leq j\leq\re-1$,
\begin{align*}
\left\Vert \by_{j+1}-\widehat{\by}_{j+1}\right\Vert  & \leq\frac{\left\Vert \by_{j}-\widehat{\by}_{j}-\eta_{(\qu-1)n+j}(\nabla f_{\pi_{\qu}^{j}}(\by_{j})-\nabla f_{\widehat{\pi}_{\qu}^{j}(\re,i)}(\widehat{\by}_{j}))\right\Vert }{1+\mu\eta_{(\qu-1)n+j}}\\
 & \overset{\eqref{eq:RR-Omega-hat-pi}}{=}\frac{\left\Vert \by_{j}-\widehat{\by}_{j}-\eta_{(\qu-1)n+j}(\nabla f_{\pi_{\qu}^{j}}(\by_{j})-\nabla f_{\pi_{\qu}^{j}}(\widehat{\by}_{j}))\right\Vert }{1+\mu\eta_{(\qu-1)n+j}},
\end{align*}
which implies
\begin{align}
\left\Vert \by_{j+1}-\widehat{\by}_{j+1}\right\Vert ^{2} & \leq\frac{\left\Vert \by_{j}-\widehat{\by}_{j}\right\Vert ^{2}-2\eta_{(\qu-1)n+j}\left\langle \by_{j}-\widehat{\by}_{j},\nabla f_{\pi_{\qu}^{j}}(\by_{j})-\nabla f_{\pi_{\qu}^{j}}(\widehat{\by}_{j})\right\rangle +\eta_{(\qu-1)n+j}^{2}\left\Vert \nabla f_{\pi_{\qu}^{j}}(\by_{j})-\nabla f_{\pi_{\qu}^{j}}(\widehat{\by}_{j})\right\Vert ^{2}}{(1+\mu\eta_{(\qu-1)n+j})^{2}}\nonumber \\
 & \leq\frac{\left\Vert \by_{j}-\widehat{\by}_{j}\right\Vert ^{2}+4\eta_{(\qu-1)n+j}^{2}G_{\pi_{\qu}^{j}}^{2}}{(1+\mu\eta_{(\qu-1)n+j})^{2}},\label{eq:RR-stability-2}
\end{align}
where the last line is by
\begin{align*}
\left\langle \by_{j}-\widehat{\by}_{j},\nabla f_{\pi_{\qu}^{j}}(\by_{j})-\nabla f_{\pi_{\qu}^{j}}(\widehat{\by}_{j})\right\rangle  & \overset{\text{Assumption }\ref{assu:basic}}{\geq}0,\\
\left\Vert \nabla f_{\pi_{\qu}^{j}}(\by_{j})-\nabla f_{\pi_{\qu}^{j}}(\widehat{\by}_{j})\right\Vert ^{2} & \overset{\text{Assumption }\ref{assu:lip}}{\leq}4G_{\pi_{\qu}^{j}}^{2}.
\end{align*}

Finally, unrolling (\ref{eq:RR-stability-2}) recursively to obtain
\begin{align*}
\left\Vert \by_{\re}-\widehat{\by}_{\re}\right\Vert ^{2} & \leq\frac{\left\Vert \by_{i+1}-\widehat{\by}_{i+1}\right\Vert ^{2}}{\prod_{\ell=i+1}^{\re-1}(1+\mu\eta_{(\qu-1)n+\ell})^{2}}+\sum_{j=i+1}^{\re-1}\frac{4\eta_{(\qu-1)n+j}^{2}G_{\pi_{\qu}^{j}}^{2}}{\prod_{\ell=j}^{\re-1}(1+\mu\eta_{(\qu-1)n+\ell})^{2}}\\
 & \overset{\eqref{eq:RR-stability-1}}{\leq}\frac{\eta_{(\qu-1)n+i}^{2}(G_{\pi_{\qu}^{i}}+G_{\pi_{\qu}^{\re}})^{2}}{\prod_{\ell=i}^{\re-1}(1+\mu\eta_{(\qu-1)n+\ell})^{2}}+\sum_{j=i+1}^{\re-1}\frac{4\eta_{(\qu-1)n+j}^{2}G_{\pi_{\qu}^{j}}^{2}}{\prod_{\ell=j}^{\re-1}(1+\mu\eta_{(\qu-1)n+\ell})^{2}}\\
 & \leq\frac{2\eta_{(\qu-1)n+i}^{2}(G_{\pi_{\qu}^{i}}^{2}+G_{\pi_{\qu}^{\re}}^{2})}{\prod_{\ell=i}^{\re-1}(1+\mu\eta_{(\qu-1)n+\ell})^{2}}+\sum_{j=i+1}^{\re-1}\frac{4\eta_{(\qu-1)n+j}^{2}G_{\pi_{\qu}^{j}}^{2}}{\prod_{\ell=j}^{\re-1}(1+\mu\eta_{(\qu-1)n+\ell})^{2}},\\
 & \leq\frac{2\eta_{(\qu-1)n+i}^{2}G_{\pi_{\qu}^{\re}}^{2}}{\prod_{\ell=i}^{\re-1}(1+\mu\eta_{(\qu-1)n+\ell})^{2}}+\sum_{j=i}^{\re-1}\frac{4\eta_{(\qu-1)n+j}^{2}G_{\pi_{\qu}^{j}}^{2}}{\prod_{\ell=j}^{\re-1}(1+\mu\eta_{(\qu-1)n+\ell})^{2}}
\end{align*}
Therefore, we know
\begin{align*}
\E\left[\left\Vert \by_{\re}-\widehat{\by}_{\re}\right\Vert ^{2}\mid\pi_{\qu}^{\re}\right] & \leq\frac{2\eta_{(\qu-1)n+i}^{2}G_{\pi_{\qu}^{\re}}^{2}}{\prod_{\ell=i}^{\re-1}(1+\mu\eta_{(\qu-1)n+\ell})^{2}}+\sum_{j=i}^{\re-1}\frac{4\eta_{(\qu-1)n+j}^{2}\E\left[G_{\pi_{\qu}^{j}}^{2}\mid\pi_{\qu}^{\re}\right]}{\prod_{\ell=j}^{\re-1}(1+\mu\eta_{(\qu-1)n+\ell})^{2}}\\
 & \overset{(a)}{=}\frac{2\eta_{(\qu-1)n+i}^{2}G_{\pi_{\qu}^{\re}}^{2}}{\prod_{\ell=i}^{\re-1}(1+\mu\eta_{(\qu-1)n+\ell})^{2}}+\sum_{j=i}^{\re-1}\frac{8\eta_{(\qu-1)n+j}^{2}G_{f,2}^{2}}{\prod_{\ell=j}^{\re-1}(1+\mu\eta_{(\qu-1)n+\ell})^{2}}\\
 & \overset{(b)}{=}\frac{2\gamma_{(\qu-1)n+i}^{2}\eta_{(\qu-1)n+i}^{2}G_{\pi_{\qu}^{\re}}^{2}}{\gamma_{s}^{2}}+\sum_{j=i}^{\re-1}\frac{8\gamma_{(\qu-1)n+j}^{2}\eta_{(\qu-1)n+j}^{2}G_{f,2}^{2}}{\gamma_{s}^{2}},
\end{align*}
where $(a)$ is by (w.l.o.g., we assume $n\geq2$ now, otherwise,
our final bound holds automatically when $n=1$ since $\bx_{s}=\widehat{\bx}_{s}(\re,i)$
in that case)
\[
\E\left[G_{\pi_{\qu}^{j}}^{2}\mid\pi_{\qu}^{\re}\right]=\frac{nG_{f,2}^{2}-G_{\pi_{\qu}^{\re}}^{2}}{n-1}\leq\frac{nG_{f,2}^{2}}{n-1}\leq2G_{f,2}^{2},\forall j\neq\re,
\]
and $(b)$ is due to $\gamma_{t}=\prod_{s=1}^{t-1}(1+\mu\eta_{s}),\forall t\in\left[T+1\right]$.
We hence obtain the desired bound on $\E\left[\left\Vert \bx_{s}-\widehat{\bx}_{s}(\re,i)\right\Vert ^{2}\mid\pi_{\qu}^{\re}\right]$
as $\by_{\re}=\bx_{(\qu-1)n+\re}=\bx_{s}$ and $\widehat{\by}_{\re}=\widehat{\bx}_{(\qu-1)n+\re}(\re,i)=\widehat{\bx}_{s}(\re,i)$.
\end{proof}

\subsection{Analysis for the $\protect\SS$ Sampling Scheme\label{subsec:SS-analysis}}

This subsection helps us to bound $\left|\E\left[\Omega_{t}(\bx_{s})\right]\right|,\forall t\in\left[T\right],s\in\left[t\right]$
for the $\SS$ sampling scheme. The proof is inspired by \cite{NEURIPS2022_7bc4f74e}.
Again, our result can be viewed as a finer generalization than theirs
and hence requires more careful analysis.
\begin{lem}
\label{lem:SS-Omega}Under Assumptions \ref{assu:basic} and \ref{assu:lip},
suppose the $\SS$ sampling scheme is employed, then for any $t\in\left[T\right]$
and $s\in\left[t\right]$, let $\re\defeq\re(t)$, Algorithm \ref{alg:Alg}
guarantees
\begin{align*}
\left|\E\left[\Omega_{t}(\bx_{s})\right]\right|\leq & 4G_{f,2}^{2}\sum_{j=1}^{s-1}\frac{\gamma_{j}\eta_{j}}{\gamma_{s}}\left(\1\left[\re(j)=\re\right]+\frac{\1\left[\re(j)\neq\re\right]}{n-1}\right)\\
 & +\frac{2}{n}\sum_{i=1}^{n}G_{i}\sqrt{\sum_{j=1}^{s-1}\frac{\gamma_{j}^{2}\eta_{j}^{2}}{\gamma_{s}^{2}}\left(G_{f,2}^{2}+G_{i}^{2}\1\left[\re(j)=\re\right]+\frac{nG_{f,2}^{2}-G_{i}^{2}}{n-1}\1\left[\re(j)\neq\re\right]\right)},
\end{align*}
where $\Omega_{t}(\cdot)$ and $\gamma_{t}$ are defined in Lemmas
\ref{lem:core} and \ref{lem:last}, respectively.
\end{lem}

\begin{proof}
Under the $\SS$ sampling scheme, for any $t\in\left[T\right]$, $\bx_{t}$
can be recognized as being generated by a deterministic map $\mathcal{A}_{t}$\footnote{Same as Footnote \ref{fn:deterministic}, we also assume $\nabla f_{i}(\bx)$
is deterministically picked from $\partial f_{i}(\bx)$ for any $i\in\left[n\right]$
and $\bx\in\R^{d}$.} from the permutation $\pi$ to $\R^{d}$ when the initial point $\bx_{1}$
and the stepsize $\eta_{t},\forall t\in\left[T\right]$ are fixed.
In other words, we can write
\[
\bx_{t}=\mathcal{A}_{t}(\pi),\forall t\in\left[T\right].
\]
We also recall the following fact about the index
\begin{equation}
\I(t)=\pi^{\re(t)}=\pi^{\re}.\label{eq:SS-Omega-index}
\end{equation}
Hence, there is
\[
\E\left[f_{\I(t)}(\bx_{s})\right]=\E\left[f_{\pi^{\re}}(\mathcal{A}_{s}(\pi))\right]=\sum_{i=1}^{n}\E\left[f_{i}(\mathcal{A}_{s}(\pi))\1\left[\pi^{\re}=i\right]\right].
\]
For any $i\in\left[n\right]$, let $\widehat{\pi}(\re,\inv_{i})$
denote the permutation obtained by exchanging $\pi^{\re}$ and $\pi^{\inv_{i}}$
where $\inv_{i}$ is the unique index satisfying $\pi^{\inv_{i}}=i$.
By applying Lemma \ref{lem:SS-same-dist} with $\phi(\cdot)=f_{i}(\mathcal{A}_{s}(\cdot))$,
there is
\[
\E\left[f_{i}(\mathcal{A}_{s}(\pi))\1\left[\pi^{\re}=i\right]\right]=\frac{1}{n}\E\left[f_{i}(\mathcal{A}_{s}(\widehat{\pi}(\re,\inv_{i})))\right],\forall i\in\left[n\right],
\]
which implies
\[
\E\left[f_{\I(t)}(\bx_{s})\right]=\sum_{i=1}^{n}\E\left[f_{i}(\mathcal{A}_{s}(\pi))\1\left[\pi^{\re}=i\right]\right]=\frac{1}{n}\sum_{i=1}^{n}\E\left[f_{i}(\mathcal{A}_{s}(\widehat{\pi}(\re,\inv_{i})))\right].
\]
Therefore,
\begin{align*}
\left|\E\left[\Omega_{t}(\bx_{s})\right]\right| & =\left|\frac{1}{n}\sum_{i=1}^{n}\E\left[f_{i}(\mathcal{A}_{s}(\widehat{\pi}(\re,\inv_{i})))-f_{i}(\mathcal{A}_{s}(\pi))\right]\right|\leq\frac{1}{n}\sum_{i=1}^{n}\E\left[\left|f_{i}(\mathcal{A}_{s}(\widehat{\pi}(\re,\inv_{i})))-f_{i}(\mathcal{A}_{s}(\pi))\right|\right]\\
 & \overset{(a)}{\leq}\frac{1}{n}\sum_{i=1}^{n}\E\left[G_{i}\left\Vert \mathcal{A}_{s}(\widehat{\pi}(\re,\inv_{i}))-\mathcal{A}_{s}(\pi)\right\Vert \right]=\frac{1}{n}\sum_{i=1}^{n}G_{i}\E\left[\left\Vert \widehat{\bx}_{s}(\re,\inv_{i})-\bx_{s}\right\Vert \right],
\end{align*}
where $(a)$ is because $f_{i}$ is $G_{i}$-Lipschitz on $\dom\psi$
and $\widehat{\bx}_{s}(\re,\inv_{i})$ is the output of running Algorithm
\ref{alg:Alg} with the same initial point $\bx_{1}$ and the stepsize
$\eta_{t},\forall t\in\left[T\right]$ under the $\SS$ sampling scheme
but using the permutation $\widehat{\pi}(\re,\inv_{i})$, i.e.,
\begin{align}
\widehat{\bx}_{1}(\re,\inv_{i}) & \defeq\bx_{1},\label{eq:SS-Omega-hat-x-init}\\
\widehat{\bx}_{j+1}(\re,\inv_{i}) & \defeq\argmin_{\bx\in\R^{d}}\psi(\bx)+\left\langle \nabla f_{\widehat{\pi}^{\re(j)}(\re,\inv_{i})}(\widehat{\bx}_{j}(\re,\inv_{i})),\bx\right\rangle +\frac{\left\Vert \bx-\widehat{\bx}_{j}(\re,\inv_{i})\right\Vert ^{2}}{2\eta_{j}},\forall j\in\left[s-1\right].\label{eq:SS-Omega-hat-x}
\end{align}

Finally, by Lemma \ref{lem:SS-stability}, we have
\begin{align*}
\frac{1}{n}\sum_{i=1}^{n}G_{i}\E\left[\left\Vert \widehat{\bx}_{s}(\re,\inv_{i})-\bx_{s}\right\Vert \right]\leq & \frac{1}{n}\sum_{i=1}^{n}G_{i}\cdot2\left(G_{f,1}+G_{i}\right)\sum_{j=1}^{s-1}\frac{\gamma_{j}\eta_{j}}{\gamma_{s}}\left(\1\left[\re(j)=\re\right]+\frac{\1\left[\re(j)\neq\re\right]}{n-1}\right)\\
 & +\frac{1}{n}\sum_{i=1}^{n}G_{i}\cdot2\sqrt{\sum_{j=1}^{s-1}\frac{\gamma_{j}^{2}\eta_{j}^{2}}{\gamma_{s}^{2}}\left(G_{f,2}^{2}+G_{i}^{2}\1\left[\re(j)=\re\right]+\frac{nG_{f,2}^{2}-G_{i}^{2}}{n-1}\1\left[\re(j)\neq\re\right]\right)}\\
= & 2\left(G_{f,1}^{2}+G_{f,2}^{2}\right)\sum_{j=1}^{s-1}\frac{\gamma_{j}\eta_{j}}{\gamma_{s}}\left(\1\left[\re(j)=\re\right]+\frac{\1\left[\re(j)\neq\re\right]}{n-1}\right)\\
 & +\frac{2}{n}\sum_{i=1}^{n}G_{i}\sqrt{\sum_{j=1}^{s-1}\frac{\gamma_{j}^{2}\eta_{j}^{2}}{\gamma_{s}^{2}}\left(G_{f,2}^{2}+G_{i}^{2}\1\left[\re(j)=\re\right]+\frac{nG_{f,2}^{2}-G_{i}^{2}}{n-1}\1\left[\re(j)\neq\re\right]\right)}.
\end{align*}
The proof is completed by using the fact $G_{f,1}\leq G_{f,2}$.
\end{proof}

\begin{lem}
\label{lem:SS-stability}Under the same settings in Lemma \ref{lem:SS-Omega},
let $\widehat{\bx}_{s}(\re,\inv_{i})$ be the point defined by (\ref{eq:SS-Omega-hat-x-init})
and (\ref{eq:SS-Omega-hat-x}), then we have
\begin{align*}
\E\left[\left\Vert \bx_{s}-\widehat{\bx}_{s}(\re,\inv_{i})\right\Vert \right]\leq & 2\left(G_{f,1}+G_{i}\right)\sum_{j=1}^{s-1}\frac{\gamma_{j}\eta_{j}}{\gamma_{s}}\left(\1\left[\re(j)=\re\right]+\frac{\1\left[\re(j)\neq\re\right]}{n-1}\right)\\
 & +2\sqrt{\sum_{j=1}^{s-1}\frac{\gamma_{j}^{2}\eta_{j}^{2}}{\gamma_{s}^{2}}\left(G_{f,2}^{2}+G_{i}^{2}\1\left[\re(j)=\re\right]+\frac{nG_{f,2}^{2}-G_{i}^{2}}{n-1}\1\left[\re(j)\neq\re\right]\right)},
\end{align*}
where $\gamma_{t}$ is defined in Lemma \ref{lem:last}.
\end{lem}

\begin{proof}
For simplicity, we denote $\delta_{s}\defeq\left\Vert \bx_{s}-\widehat{\bx}_{s}(\re,\inv_{i})\right\Vert ,\forall s\in\left[t\right]$.
Moreover, let $\widehat{\I}(s)$ represent the index trajectory generated
under the permutation $\widehat{\pi}(\re,\inv_{i})$, i.e.,
\begin{equation}
\widehat{\I}(s)\defeq\widehat{\pi}^{\re(s)}(\re,\inv_{i}),\forall s\in\left[t-1\right].\label{eq:SS-stability-index}
\end{equation}

By Lemma \ref{lem:contractive}, we have
\begin{align}
\delta_{s+1}^{2} & \leq\frac{\delta_{s}^{2}-2\eta_{s}\left\langle \bx_{s}-\widehat{\bx}_{s}(\re,\inv_{i}),\nabla f_{\I(s)}(\bx_{s})-\nabla f_{\widehat{\I}(s)}(\widehat{\bx}_{s}(\re,\inv_{i}))\right\rangle +\eta_{s}^{2}\left\Vert \nabla f_{\I(s)}(\bx_{s})-\nabla f_{\widehat{\I}(s)}(\widehat{\bx}_{s}(\re,\inv_{i}))\right\Vert ^{2}}{(1+\mu\eta_{s})^{2}}\nonumber \\
 & \overset{(a)}{\leq}\frac{\delta_{s}^{2}+2\eta_{s}\left\Vert \nabla f_{\I(s)}(\bx_{s})-\nabla f_{\widehat{\I}(s)}(\bx_{s})\right\Vert \delta_{s}+\eta_{s}^{2}\left\Vert \nabla f_{\I(s)}(\bx_{s})-\nabla f_{\widehat{\I}(s)}(\widehat{\bx}_{s}(\re,\inv_{i}))\right\Vert ^{2}}{(1+\mu\eta_{s})^{2}}\nonumber \\
 & \overset{(b)}{\leq}\frac{\delta_{s}^{2}+2\eta_{s}\left\Vert \nabla f_{\I(s)}(\bx_{s})-\nabla f_{\widehat{\I}(s)}(\bx_{s})\right\Vert \delta_{s}+2\eta_{s}^{2}\left(G_{\I(s)}^{2}+G_{\widehat{\I}(s)}^{2}\right)}{(1+\mu\eta_{s})^{2}},\label{eq:SS-stability-1}
\end{align}
where $(a)$ is by 
\begin{align*}
 & \left\langle \bx_{s}-\widehat{\bx}_{s}(\re,\inv_{i}),\nabla f_{\I(s)}(\bx_{s})-\nabla f_{\widehat{\I}(s)}(\widehat{\bx}_{s}(\re,\inv_{i}))\right\rangle \\
= & \left\langle \bx_{s}-\widehat{\bx}_{s}(\re,\inv_{i}),\nabla f_{\widehat{\I}(s)}(\bx_{s})-\nabla f_{\widehat{\I}(s)}(\widehat{\bx}_{s}(\re,\inv_{i}))\right\rangle +\left\langle \bx_{s}-\widehat{\bx}_{s}(\re,\inv_{i}),\nabla f_{\I(s)}(\bx_{s})-\nabla f_{\widehat{\I}(s)}(\bx_{s})\right\rangle \\
\overset{\text{Assumption }\ref{assu:basic}}{\geq} & \left\langle \bx_{s}-\widehat{\bx}_{s}(\re,\inv_{i}),\nabla f_{\I(s)}(\bx_{s})-\nabla f_{\widehat{\I}(s)}(\bx_{s})\right\rangle \geq-\left\Vert \nabla f_{\I(s)}(\bx_{s})-\nabla f_{\widehat{\I}(s)}(\bx_{s})\right\Vert \delta_{s},
\end{align*}
and $(b)$ holds due to
\[
\left\Vert \nabla f_{\I(s)}(\bx_{s})-\nabla f_{\widehat{\I}(s)}(\widehat{\bx}_{s}(\re,\inv_{i}))\right\Vert \overset{\text{Assumption }\ref{assu:lip}}{\leq}G_{\I(s)}+G_{\widehat{\I}(s)}\Rightarrow\left\Vert \nabla f_{\I(s)}(\bx_{s})-\nabla f_{\widehat{\I}(s)}(\widehat{\bx}_{s}(\re,\inv_{i}))\right\Vert ^{2}\leq2\left(G_{\I(s)}^{2}+G_{\widehat{\I}(s)}^{2}\right).
\]

Now recall that $\gamma_{t}=\prod_{s=1}^{t-1}(1+\mu\eta_{s}),\forall t\in\left[T+1\right]$.
Multiply both sides of (\ref{eq:SS-stability-1}) by $\gamma_{s+1}^{2}$
to obtain
\begin{align}
\gamma_{s+1}^{2}\delta_{s+1}^{2} & \leq\gamma_{s}^{2}\delta_{s}^{2}+2\gamma_{s}\eta_{s}\left\Vert \nabla f_{\I(s)}(\bx_{s})-\nabla f_{\widehat{\I}(s)}(\bx_{s})\right\Vert \gamma_{s}\delta_{s}+2\gamma_{s}^{2}\eta_{s}^{2}\left(G_{\I(s)}^{2}+G_{\widehat{\I}(s)}^{2}\right)\nonumber \\
\Rightarrow\gamma_{s+1}^{2}\delta_{s+1}^{2} & \leq\gamma_{1}^{2}\delta_{1}^{2}+\sum_{j=1}^{s}2\gamma_{j}\eta_{j}\left\Vert \nabla f_{\I(j)}(\bx_{j})-\nabla f_{\widehat{\I}(j)}(\bx_{j})\right\Vert \gamma_{j}\delta_{j}+2\gamma_{j}^{2}\eta_{j}^{2}\left(G_{\I(j)}^{2}+G_{\widehat{\I}(j)}^{2}\right)\nonumber \\
 & =\sum_{j=1}^{s}2\gamma_{j}\eta_{j}\left\Vert \nabla f_{\I(j)}(\bx_{j})-\nabla f_{\widehat{\I}(j)}(\bx_{j})\right\Vert \gamma_{j}\delta_{j}+2\gamma_{j}^{2}\eta_{j}^{2}\left(G_{\I(j)}^{2}+G_{\widehat{\I}(j)}^{2}\right),\label{eq:SS-stability-2}
\end{align}
where the last equation is by $\delta_{1}=\left\Vert \bx_{1}-\widehat{\bx}_{1}(\re,\inv_{i})\right\Vert \overset{\eqref{eq:SS-Omega-hat-x-init}}{=}0$.

Next, we use induction to prove
\begin{equation}
\gamma_{s}^{2}\delta_{s}^{2}\leq4\left(\sum_{j=1}^{s-1}\gamma_{j}\eta_{j}\left\Vert \nabla f_{\I(j)}(\bx_{j})-\nabla f_{\widehat{\I}(j)}(\bx_{j})\right\Vert \right)^{2}+\sum_{j=1}^{s-1}4\gamma_{j}^{2}\eta_{j}^{2}\left(G_{\I(j)}^{2}+G_{\widehat{\I}(j)}^{2}\right),\forall s\in\left[t\right].\label{eq:SS-stability-hypothesis}
\end{equation}
For $s=1$, (\ref{eq:SS-stability-hypothesis}) holds as $\delta_{1}=0$.
Suppose (\ref{eq:SS-stability-hypothesis}) is true for all indices
in $\left[s\right]$ where $s\in\left[t-1\right]$. Then for $s+1$,
\begin{itemize}
\item if $\gamma_{s+1}^{2}\delta_{s+1}^{2}\leq\max_{k\in\left[s\right]}\gamma_{k}^{2}\delta_{k}^{2}$,
we know
\begin{align*}
\gamma_{s+1}^{2}\delta_{s+1}^{2} & \leq\max_{k\in\left[s\right]}\gamma_{k}^{2}\delta_{k}^{2}\\
 & \overset{\eqref{eq:SS-stability-hypothesis}}{\leq}\max_{k\in\left[s\right]}\left[4\left(\sum_{j=1}^{k-1}\gamma_{j}\eta_{j}\left\Vert \nabla f_{\I(j)}(\bx_{j})-\nabla f_{\widehat{\I}(j)}(\bx_{j})\right\Vert \right)^{2}+\sum_{j=1}^{k-1}4\gamma_{j}^{2}\eta_{j}^{2}\left(G_{\I(j)}^{2}+G_{\widehat{\I}(j)}^{2}\right)\right]\\
 & \leq4\left(\sum_{j=1}^{s}\gamma_{j}\eta_{j}\left\Vert \nabla f_{\I(j)}(\bx_{j})-\nabla f_{\widehat{\I}(j)}(\bx_{j})\right\Vert \right)^{2}+\sum_{j=1}^{s}4\gamma_{j}^{2}\eta_{j}^{2}\left(G_{\I(j)}^{2}+G_{\widehat{\I}(j)}^{2}\right);
\end{align*}
\item if $\gamma_{s+1}^{2}\delta_{s+1}^{2}>\max_{k\in\left[s\right]}\gamma_{k}^{2}\delta_{k}^{2}$,
we know
\begin{align*}
\gamma_{s+1}^{2}\delta_{s+1}^{2} & \overset{\eqref{eq:SS-stability-2}}{\leq}\sum_{j=1}^{s}2\gamma_{j}\eta_{j}\left\Vert \nabla f_{\I(j)}(\bx_{j})-\nabla f_{\widehat{\I}(j)}(\bx_{j})\right\Vert \gamma_{j}\delta_{j}+2\gamma_{j}^{2}\eta_{j}^{2}\left(G_{\I(j)}^{2}+G_{\widehat{\I}(j)}^{2}\right)\\
 & \leq2\left(\sum_{j=1}^{s}\gamma_{j}\eta_{j}\left\Vert \nabla f_{\I(j)}(\bx_{j})-\nabla f_{\widehat{\I}(j)}(\bx_{j})\right\Vert \right)\gamma_{s+1}\delta_{s+1}+2\gamma_{j}^{2}\eta_{j}^{2}\left(G_{\I(j)}^{2}+G_{\widehat{\I}(j)}^{2}\right)\\
 & \overset{(c)}{\leq}\frac{\gamma_{s+1}^{2}\delta_{s+1}^{2}}{2}+2\left(\sum_{j=1}^{s}\gamma_{j}\eta_{j}\left\Vert \nabla f_{\I(j)}(\bx_{j})-\nabla f_{\widehat{\I}(j)}(\bx_{j})\right\Vert \right)^{2}+\sum_{j=1}^{s}2\gamma_{j}^{2}\eta_{j}^{2}\left(G_{\I(j)}^{2}+G_{\widehat{\I}(j)}^{2}\right)\\
\Rightarrow\gamma_{s+1}^{2}\delta_{s+1}^{2} & \leq4\left(\sum_{j=1}^{s}\gamma_{j}\eta_{j}\left\Vert \nabla f_{\I(j)}(\bx_{j})-\nabla f_{\widehat{\I}(j)}(\bx_{j})\right\Vert \right)^{2}+\sum_{j=1}^{s}4\gamma_{j}^{2}\eta_{j}^{2}\left(G_{\I(j)}^{2}+G_{\widehat{\I}(j)}^{2}\right),
\end{align*}
where $(c)$ is due to AM-GM inequality.
\end{itemize}
Therefore, we always have
\[
\gamma_{s+1}^{2}\delta_{s+1}^{2}\leq4\left(\sum_{j=1}^{s}\gamma_{j}\eta_{j}\left\Vert \nabla f_{\I(j)}(\bx_{j})-\nabla f_{\widehat{\I}(j)}(\bx_{j})\right\Vert \right)^{2}+\sum_{j=1}^{s}4\gamma_{j}^{2}\eta_{j}^{2}\left(G_{\I(j)}^{2}+G_{\widehat{\I}(j)}^{2}\right).
\]
By induction, (\ref{eq:SS-stability-hypothesis}) holds for any $s\in\left[t\right]$,
which implies
\begin{align}
\gamma_{s}\delta_{s} & \leq2\sum_{j=1}^{s-1}\gamma_{j}\eta_{j}\left\Vert \nabla f_{\I(j)}(\bx_{j})-\nabla f_{\widehat{\I}(j)}(\bx_{j})\right\Vert +2\sqrt{\sum_{j=1}^{s-1}\gamma_{j}^{2}\eta_{j}^{2}\left(G_{\I(j)}^{2}+G_{\widehat{\I}(j)}^{2}\right)}\nonumber \\
\Rightarrow\gamma_{s}\E\left[\delta_{s}\right] & \leq2\underbrace{\E\left[\sum_{j=1}^{s-1}\gamma_{j}\eta_{j}\left\Vert \nabla f_{\I(j)}(\bx_{j})-\nabla f_{\widehat{\I}(j)}(\bx_{j})\right\Vert \right]}_{\defeq\fullmoon}+2\underbrace{\E\left[\sqrt{\sum_{j=1}^{s-1}\gamma_{j}^{2}\eta_{j}^{2}\left(G_{\I(j)}^{2}+G_{\widehat{\I}(j)}^{2}\right)}\right]}_{\defeq\newmoon}.\label{eq:SS-stability-3}
\end{align}
\begin{itemize}
\item For term $\fullmoon$, note that if $\inv_{i}=\re$, then $\pi=\widehat{\pi}(\re,\star_{i})\Rightarrow\nabla f_{\I(j)}(\bx_{j})=\nabla f_{\widehat{\I}(j)}(\bx_{j}),\forall j\in\left[s-1\right]$.
So there is
\begin{align*}
\left\Vert \nabla f_{\I(j)}(\bx_{j})-\nabla f_{\widehat{\I}(j)}(\bx_{j})\right\Vert  & =\left\Vert \nabla f_{\I(j)}(\bx_{j})-\nabla f_{\widehat{\I}(j)}(\bx_{j})\right\Vert \1\left[\inv_{i}\ne\re\right]\\
 & =\sum_{\ell\in\left[n\right]\backslash\left\{ \re\right\} }\left\Vert \nabla f_{\I(j)}(\bx_{j})-\nabla f_{\widehat{\I}(j)}(\bx_{j})\right\Vert \1\left[\pi^{\ell}=i\right].
\end{align*}
When $\pi^{\ell}=i$ for some $\ell\in\left[n\right]\backslash\left\{ \re\right\} $,
we observe that $\widehat{\I}(j)\overset{\eqref{eq:SS-stability-index}}{=}\widehat{\pi}^{\re(j)}(\re,\ell)=\pi^{\re(j)}=\I(j)$
if $\re(j)\neq\re$ and $\re(j)\neq\ell$, which implies
\begin{align*}
\left\Vert \nabla f_{\I(j)}(\bx_{j})-\nabla f_{\widehat{\I}(j)}(\bx_{j})\right\Vert \1\left[\pi^{\ell}=i\right] & =\left\Vert \nabla f_{\I(j)}(\bx_{j})-\nabla f_{\widehat{\I}(j)}(\bx_{j})\right\Vert \1\left[\pi^{\ell}=i\right]\1\left[\re(j)=\re\text{ or }\ell\right]\\
 & =\left\Vert \nabla f_{\pi^{\re}}(\bx_{j})-\nabla f_{i}(\bx_{j})\right\Vert \1\left[\pi^{\ell}=i\right]\1\left[\re(j)=\re\text{ or }\ell\right]\\
 & \leq\left(G_{\pi^{\re}}+G_{i}\right)\1\left[\pi^{\ell}=i\right]\1\left[\re(j)=\re\text{ or }\ell\right],
\end{align*}
where the second to last step is by $\left\{ \I(j),\widehat{\I}(j)\right\} =\left\{ \pi^{\re(j)},\widehat{\pi}^{\re(j)}(\re,\ell)\right\} =\left\{ \pi^{\re},i\right\} $
under the events $\pi^{\ell}=i$ and $\re(j)=\re\text{ or }\ell$.
Thus, for any $j\in\left[s-1\right]$,
\begin{align*}
\left\Vert \nabla f_{\I(j)}(\bx_{j})-\nabla f_{\widehat{\I}(j)}(\bx_{j})\right\Vert  & \leq\sum_{\ell\in\left[n\right]\backslash\left\{ \re\right\} }\left(G_{\pi^{\re}}+G_{i}\right)\1\left[\pi^{\ell}=i\right]\1\left[\re(j)=\re\text{ or }\ell\right]\\
\Rightarrow\E\left[\left\Vert \nabla f_{\I(j)}(\bx_{j})-\nabla f_{\widehat{\I}(j)}(\bx_{j})\right\Vert \right] & \leq\sum_{\ell\in\left[n\right]\backslash\left\{ \re\right\} }\E\left[\left(G_{\pi^{\re}}+G_{i}\right)\1\left[\pi^{\ell}=i\right]\right]\1\left[\re(j)=\re\text{ or }\ell\right]\\
 & \overset{(d)}{=}\sum_{\ell\in\left[n\right]\backslash\left\{ \re\right\} }\left(\frac{nG_{f,1}-G_{i}}{n(n-1)}+\frac{G_{i}}{n}\right)\1\left[\re(j)=\re\text{ or }\ell\right]\\
 & =\left(G_{f,1}+\frac{n-2}{n}G_{i}\right)\left(\1\left[\re(j)=\re\right]+\frac{\1\left[\re(j)\neq\re\right]}{n-1}\right)\\
 & \leq\left(G_{f,1}+G_{i}\right)\left(\1\left[\re(j)=\re\right]+\frac{\1\left[\re(j)\neq\re\right]}{n-1}\right),
\end{align*}
where $(d)$ is by, for any fixed $\ell\in\left[n\right]\backslash\left\{ \re\right\} $,
there are
\[
\E\left[G_{\pi^{\re}}\1\left[\pi^{\ell}=i\right]\right]=\sum_{k\in\left[n\right]\backslash\left\{ i\right\} }G_{k}\P\left[\pi^{\ell}=i,\pi^{\re}=k\right]=\frac{1}{n(n-1)}\sum_{k\in\left[n\right]\backslash\left\{ i\right\} }G_{k}=\frac{nG_{f,1}-G_{i}}{n(n-1)},
\]
and 
\[
\E\left[G_{i}\1\left[\pi^{\ell}=i\right]\right]=\frac{G_{i}}{n}.
\]
We thereby have
\begin{equation}
\Circle\leq\left(G_{f,1}+G_{i}\right)\sum_{j=1}^{s-1}\gamma_{j}\eta_{j}\left(\1\left[\re(j)=\re\right]+\frac{\1\left[\re(j)\neq\re\right]}{n-1}\right).\label{eq:SS-stability-white}
\end{equation}
\item For term $\newmoon$, for any fixed $j\in\left[s-1\right]$, we claim
the following three facts hold
\begin{eqnarray*}
\I(j)\diseq\mathrm{Uniform}\left[n\right], & \widehat{\I}(j)=1\text{ if }\re(j)=\re, & \widehat{\I}(j)\diseq\mathrm{Uniform}\left[n\right]\backslash\left\{ i\right\} \text{ if }\re(j)\neq\re,
\end{eqnarray*}
in which the first one follows by the definition of the $\SS$ sampling
scheme, the second one is true by recalling $\widehat{\I}(j)\overset{\eqref{eq:SS-stability-index}}{=}\widehat{\pi}^{\re(j)}(\re,\inv_{i})=i$
if $\re(j)=\re$, and the third one is by (\ref{eq:SS-stability-index})
and Lemma \ref{lem:SS-marginal}. Therefore, by H\"{o}lder's inequality
\begin{align}
\newmoon & \leq\sqrt{\sum_{j=1}^{s-1}\gamma_{j}^{2}\eta_{j}^{2}\left(\E\left[G_{\I(j)}^{2}\right]+\E\left[G_{\widehat{\I}(j)}^{2}\right]\right)}\nonumber \\
 & =\sqrt{\sum_{j=1}^{s-1}\gamma_{j}^{2}\eta_{j}^{2}\left(G_{f,2}^{2}+G_{i}^{2}\1\left[\re(j)=\re\right]+\frac{nG_{f,2}^{2}-G_{i}^{2}}{n-1}\1\left[\re(j)\neq\re\right]\right)}.\label{eq:SS-stability-black}
\end{align}
\end{itemize}
Finally, we conclude by plugging (\ref{eq:SS-stability-white}) and
(\ref{eq:SS-stability-black}) back into (\ref{eq:SS-stability-3})
and dividing both sides by $\gamma_{s}$.
\end{proof}

\section{Auxiliary Lemmas\label{sec:auxiliary}}

This section includes some technical results applied in the analysis
presented in the previous sections.

We first provide two algebraic inequalities for the stepsize proportional
to $\qu(T)-\qu(t)+1$, which is used for the $\RR$ sampling scheme
in Theorem \ref{thm:RR-cvx-full}.
\begin{lem}
\label{lem:stepsize}Suppose $\eta_{t}=\eta_{\star}(\qu(T)-\qu(t)+1),\forall t\in\left[T\right]$
where $\eta_{\star}>0$ is a constant, then there are
\begin{eqnarray*}
\sum_{t=1}^{T}\eta_{t}\geq\frac{\eta_{\star}\qu(T)T}{2} & and & \sum_{t=1}^{T}\frac{\eta_{t}^{2}}{\sum_{s=t}^{T}\eta_{s}}\leq\frac{9\eta_{\star}(\qu(T)+\log n)}{2}.
\end{eqnarray*}
\end{lem}

\begin{proof}
Note that for any $t\in\left[T\right]$, there is
\begin{align}
\sum_{s=t}^{T}\eta_{s} & =\sum_{s=(\qu(T)-1)n+1}^{T}\eta_{s}+\sum_{s=\qu(t)n+1}^{(\qu(T)-1)n}\eta_{s}+\sum_{s=t}^{\qu(t)n}\eta_{s}\nonumber \\
 & =\eta_{\star}\left[\re(T)+n\left(\sum_{j=\qu(t)+1}^{\qu(T)-1}\qu(T)-j+1\right)+(n-\re(t)+1)(\qu(T)-\qu(t)+1)\right]\nonumber \\
 & =\eta_{\star}\left[\re(T)+\frac{n}{2}(\qu(T)-\qu(t)-1)(\qu(T)-\qu(t)+2)+(n-\re(t)+1)(\qu(T)-\qu(t)+1)\right].\label{eq:stepsize-1}
\end{align}
In particular, for $t=1$,
\[
\sum_{s=1}^{T}\eta_{s}=\eta_{\star}\left[\re(T)+\frac{n}{2}(\qu(T)-2)(\qu(T)+1)+n\qu(T)\right]=\eta_{\star}\left[\frac{n}{2}\qu(T)(\qu(T)+1)+\re(T)-n\right].
\]
\begin{itemize}
\item If $\qu(T)=1$ (i.e., $T\in\left[n\right]$), we have
\[
\sum_{s=1}^{T}\eta_{s}=\eta_{\star}\re(T)=\eta_{\star}\qu(T)T\geq\frac{\eta_{\star}\qu(T)T}{2}.
\]
\item If $\qu(T)\geq2$ (i.e., $T\geq n+1$), we have
\begin{align*}
\sum_{s=1}^{T}\eta_{s} & =\eta_{\star}\left[\frac{n}{2}\qu(T)(\qu(T)+1)+\re(T)-n\right]\overset{(a)}{\geq}\eta_{\star}\left[\frac{n}{2}\qu(T)(\frac{T}{n}+1)+\re(T)-n\right]\\
 & =\eta_{\star}\left[\frac{\qu(T)T}{2}+\frac{\qu(T)}{2}n+\re(T)-n\right]\overset{(b)}{\geq}\eta_{\star}\left[\frac{\qu(T)T}{2}+\re(T)\right]\geq\frac{\eta_{\star}\qu(T)T}{2},
\end{align*}
where $(a)$ is by $\qu(T)\geq\frac{T}{n}$ and $(b)$ is due to $\qu(T)\geq2$
now.
\end{itemize}
Hence, we always have
\[
\sum_{s=1}^{T}\eta_{s}\geq\frac{\eta_{\star}\qu(T)T}{2}.
\]

Next, we observe
\begin{align}
\sum_{t=1}^{T}\frac{\eta_{t}^{2}}{\sum_{s=t}^{T}\eta_{s}}\overset{\eqref{eq:stepsize-1}}{=} & \eta_{\star}\sum_{t=(\qu(T)-1)n+1}^{T}\frac{1}{\re(T)-\re(t)+1}\nonumber \\
 & +\eta_{\star}\sum_{k=1}^{\qu(T)-1}\sum_{i=1}^{n}\frac{(\qu(T)-k+1)^{2}}{\re(T)+\frac{n}{2}(\qu(T)-k-1)(\qu(T)-k+2)+(n-i+1)(\qu(T)-k+1)}.\label{eq:stepsize-2}
\end{align}
\begin{itemize}
\item For the first part in (\ref{eq:stepsize-2}), we have
\[
\sum_{t=(\qu(T)-1)n+1}^{T}\frac{1}{\re(T)-\re(t)+1}=\sum_{i=1}^{\re(T)}\frac{1}{i}\leq1+\log\re(T).
\]
\item For the second part in (\ref{eq:stepsize-2}), under relabeling the
index, we have
\begin{align*}
 & \sum_{k=1}^{\qu(T)-1}\sum_{i=1}^{n}\frac{(\qu(T)-k+1)^{2}}{\re(T)+\frac{n}{2}(\qu(T)-k-1)(\qu(T)-k+2)+(n-i+1)(\qu(T)-k+1)}\\
= & \sum_{k=2}^{\qu(T)}\sum_{i=1}^{n}\frac{k^{2}}{\re(T)+\frac{n}{2}(k-2)(k+1)+ik}\leq\sum_{i=1}^{n}\frac{4}{\re(T)+2i}+\sum_{k=3}^{\qu(T)}\frac{2k^{2}}{(k-2)(k+1)}\\
\leq & 2\log\left(1+\frac{2n}{\re(T)}\right)+\frac{9}{2}(\qu(T)-2)\leq2\log\left(1+\frac{2n}{\re(T)}\right)+\frac{9}{2}\qu(T)-\frac{9}{2}.
\end{align*}
\end{itemize}
So there is
\begin{align*}
\sum_{t=1}^{T}\frac{\eta_{t}^{2}}{\sum_{s=t}^{T}\eta_{s}} & \leq\eta_{\star}\left[\log\re(T)+2\log\left(1+\frac{2n}{\re(T)}\right)+\frac{9}{2}\qu(T)-\frac{7}{2}\right]\\
 & \leq\eta_{\star}\left[2\log(\re(T)+2n)+\frac{9}{2}\qu(T)-\frac{7}{2}\right]\\
 & \leq\eta_{\star}\left[2\log n+\frac{9}{2}\qu(T)+2\log3-\frac{7}{2}\right]\\
 & \leq\frac{9\eta_{\star}(\qu(T)+\log n)}{2}.
\end{align*}
\end{proof}

Next, we introduce Lemma \ref{lem:contractive}, which gives a general
upper bound on the distance between two points output by different
proximal updates but using the same stepsize and plays a key role
in bounding $\left|\E\left[\Omega_{t}(\bx_{s})\right]\right|,\forall t\in\left[T\right],s\in\left[t\right]$.
\begin{lem}
\label{lem:contractive}Under Assumption \ref{assu:basic}, given
$\bar{\bx},\bar{\by},\bg_{\bar{\bx}},\bg_{\bar{\by}}\in\R^{d}$ and
$\eta>0$, let
\begin{align*}
\tilde{\bx} & \defeq\argmin_{\bx\in\R^{d}}\psi(\bx)+\left\langle \bg_{\bar{\bx}},\bx\right\rangle +\frac{\left\Vert \bx-\bar{\bx}\right\Vert ^{2}}{2\eta},\\
\tilde{\by} & \defeq\argmin_{\by\in\R^{d}}\psi(\by)+\left\langle \bg_{\bar{\by}},\by\right\rangle +\frac{\left\Vert \by-\bar{\by}\right\Vert ^{2}}{2\eta},
\end{align*}
then there is
\[
\left\Vert \tilde{\bx}-\tilde{\by}\right\Vert \leq\frac{\left\Vert \bar{\bx}-\bar{\by}-\eta(\bg_{\bar{\bx}}-\bg_{\bar{\by}})\right\Vert }{1+\mu\eta}.
\]
\end{lem}

\begin{proof}
By the definition of $\tilde{\bx}$, there exists $\nabla\psi(\tilde{\bx})\in\partial\psi(\tilde{\bx})$
such that $\bzero=\nabla\psi(\tilde{\bx})+\bg_{\bar{\bx}}+\frac{\tilde{\bx}-\bar{\bx}}{\eta}$,
which implies
\[
\left\langle \eta\nabla\psi(\tilde{\bx})+\tilde{\bx},\tilde{\bx}-\tilde{\by}\right\rangle =\left\langle \bar{\bx}-\eta\bg_{\bar{\bx}},\tilde{\bx}-\tilde{\by}\right\rangle .
\]
Similarly, we have
\[
\left\langle \eta\nabla\psi(\tilde{\by})+\tilde{\by},\tilde{\by}-\tilde{\bx}\right\rangle =\left\langle \bar{\by}-\eta\bg_{\bar{\by}},\tilde{\by}-\tilde{\bx}\right\rangle .
\]
Sum up the above two equations to obtain
\[
\eta\left\langle \nabla\psi(\tilde{\bx})-\nabla\psi(\tilde{\by}),\tilde{\bx}-\tilde{\by}\right\rangle +\left\Vert \tilde{\bx}-\tilde{\by}\right\Vert ^{2}=\left\langle \bar{\bx}-\bar{\by}-\eta(\bg_{\bar{\bx}}-\bg_{\bar{\by}}),\tilde{\bx}-\tilde{\by}\right\rangle .
\]
Note that Assumption \ref{assu:basic} implies
\[
\left\langle \nabla\psi(\tilde{\bx})-\nabla\psi(\tilde{\by}),\tilde{\bx}-\tilde{\by}\right\rangle \geq\mu\left\Vert \tilde{\bx}-\tilde{\by}\right\Vert ^{2}.
\]
Hence, there is
\[
\left\Vert \tilde{\bx}-\tilde{\by}\right\Vert ^{2}\leq\frac{\left\langle \bar{\bx}-\bar{\by}-\eta(\bg_{\bar{\bx}}-\bg_{\bar{\by}}),\tilde{\bx}-\tilde{\by}\right\rangle }{1+\mu\eta}\Rightarrow\left\Vert \tilde{\bx}-\tilde{\by}\right\Vert \leq\frac{\left\Vert \bar{\bx}-\bar{\by}-\eta(\bg_{\bar{\bx}}-\bg_{\bar{\by}})\right\Vert }{1+\mu\eta}.
\]
\end{proof}

Finally, we provide some useful facts related to the random permutation
inspired by \cite{NEURIPS2021_107030ca,NEURIPS2022_7bc4f74e}. We
recall that $S_{n}$ is the symmetric group of $\left[n\right]$,
i.e., the set containing all permutations of $\left[n\right]$.
\begin{lem}
\label{lem:RR-same-dist}Suppose $\pi=(\pi^{1},\cdots,\pi^{n})$ is
uniformly distributed on $S_{n}$, for any $\re\in\left[n\right]$
and $i\in\left[\re-1\right]$, we define $\widehat{\pi}(\re,i)$ by
exchanging $\pi^{\re}$ and $\pi^{i}$ in $\pi$, then there is
\[
\pi\diseq\widehat{\pi}(\re,i).
\]
\end{lem}

\begin{proof}
It is enough to prove that, for any fixed $o\in S_{n}$, there is
\[
\P\left[\pi=o\right]=\P\left[\widehat{\pi}(\re,i)=o\right],
\]
which clearly holds as both sides equal $\frac{1}{n!}$.
\end{proof}

\begin{lem}
\label{lem:SS-same-dist}Suppose $\pi=(\pi^{1},\cdots,\pi^{n})$ is
uniformly distributed on $S_{n}$, for any $\re,i\in\left[n\right]$,
we define $\widehat{\pi}(\re,\inv_{i})$ by exchanging $\pi^{\re}$
and $\pi^{\inv_{i}}$ in $\pi$ where $\inv_{i}$ is the unique index
satisfying $\pi^{\inv_{i}}=i$, then there is
\[
\E\left[\phi(\pi)\1\left[\pi^{\re}=i\right]\right]=\frac{1}{n}\E\left[\phi(\widehat{\pi}(\re,\inv_{i}))\right],
\]
where $\phi:S_{n}\to\R$ can ba any deterministic map.
\end{lem}

\begin{proof}
We observe that
\begin{equation}
\E\left[\phi(\widehat{\pi}(\re,\inv_{i}))\right]=\E\left[\sum_{j=1}^{n}\phi(\widehat{\pi}(\re,\inv_{i}))\1\left[\inv_{i}=j\right]\right]=\E\left[\sum_{j=1}^{n}\phi(\widehat{\pi}(\re,j))\1\left[\pi^{j}=i\right]\right]=\sum_{j=1}^{n}\E\left[\phi(\widehat{\pi}(\re,j))\1\left[\pi^{j}=i\right]\right].\label{eq:SS-same-dist-1}
\end{equation}
For any fixed $j\in\left[n\right]$, let $\mathcal{T}_{\re\leftrightarrow j}:S_{n}\to S_{n}$
be the operation of exchanging the value of $o^{\re}$ and $o^{j}$
for $o\in S_{n}$. Then we know
\begin{align}
\E\left[\phi(\widehat{\pi}(\re,j))\1\left[\pi^{j}=i\right]\right] & =\frac{1}{n!}\sum_{o\in S_{n}}\phi(\mathcal{T}_{\re\leftrightarrow j}(o))\1\left[o^{j}=i\right]=\frac{1}{n!}\sum_{o\in S_{n}}\phi(\mathcal{T}_{\re\leftrightarrow j}(o))\1\left[\left(\mathcal{T}_{\re\leftrightarrow j}(o)\right)^{\re}=i\right]\nonumber \\
 & =\frac{1}{n!}\sum_{o\in S_{n}}\phi(o)\1\left[o^{r}=i\right]=\E\left[\phi(\pi)\1\left[\pi^{r}=i\right]\right],\label{eq:SS-same-dist-2}
\end{align}
where the second to last step is by $\left\{ \mathcal{T}_{\re\leftrightarrow j}(o):o\in S_{n}\right\} =S_{n}$.
Combine (\ref{eq:SS-same-dist-1}) and (\ref{eq:SS-same-dist-2})
to finally obtain
\[
\E\left[\phi(\pi)\1\left[\pi^{\re}=i\right]\right]=\frac{1}{n}\sum_{j=1}^{n}\E\left[\phi(\widehat{\pi}(\re,j))\1\left[\pi^{j}=i\right]\right]=\frac{1}{n}\E\left[\phi(\widehat{\pi}(\re,\inv_{i}))\right].
\]
\end{proof}

\begin{lem}
\label{lem:SS-marginal}Suppose $\pi=(\pi^{1},\cdots,\pi^{n})$ is
uniformly distributed on $S_{n}$, for any $\re,i\in\left[n\right]$,
we define $\widehat{\pi}(\re,\inv_{i})$ by exchanging $\pi^{\re}$
and $\pi^{\inv_{i}}$ in $\pi$ where $\inv_{i}$ is the unique index
satisfying $\pi^{\inv_{i}}=i$, then there is 
\[
\widehat{\pi}^{k}(\re,\inv_{i})\diseq\mathrm{Uniform}\left[n\right]\backslash\left\{ i\right\} ,\forall k\in\left[n\right],k\neq\re.
\]
\end{lem}

\begin{proof}
Given $k\in\left[n\right]$ and $k\neq\re$, let $j\in\left[n\right]$
be fixed. We define the function $\phi(\pi)\defeq\1\left[\pi^{k}=j\right]$
and notice that
\[
\P\left[\widehat{\pi}^{k}(\re,\inv_{i})=j\right]=\E\left[\phi(\widehat{\pi}(\re,\inv_{i}))\right]\overset{\text{Lemma \ref{lem:SS-same-dist}}}{=}n\E\left[\phi(\pi)\1\left[\pi^{\re}=i\right]\right]=n\P\left[\pi^{\re}=i,\pi^{k}=j\right]=\begin{cases}
0 & j=i\\
\frac{1}{n-1} & j\neq i
\end{cases},
\]
which concludes the result.
\end{proof}

\section{Lower Bound for Strongly Convex $\psi$\label{sec:lb}}

We present a lower bound that can be applied to the first-order algorithm
containing a proximal update.

Given $F=f+\psi:\R^{d}\to\bR$ satisfying Assumptions \ref{assu:basic}
and \ref{assu:lip} and an initial point $\bx_{1}\in\R^{d}$, we consider
a family of algorithms obeying the following update rules in a total
of $T$ iterations,
\begin{align}
\by_{t+1} & \in\bx_{1}+\mathrm{Span}\cup_{s\in\left[t\right]}\left\{ \bx_{s}-\bx_{1},\by_{s}-\bx_{1},\nabla f(\bx_{s}),\nabla f(\by_{s})\right\} ,\label{eq:update-1}\\
\bx_{t+1} & =\argmin_{\bx\in\R^{d}}\psi(\bx)+\frac{\left\Vert \bx-\by_{t+1}\right\Vert ^{2}}{2\gamma_{t}},\label{eq:update-2}
\end{align}
where $\by_{1}=\bx_{1}$, $\nabla f(\bz_{t})\in\partial f(\bz),\forall t\in\left[T\right]$
for $\bz\in\left\{ \bx,\by\right\} $, and $\gamma_{t},\forall t\in\left[T\right]$
is a positive sequence. Note that (\ref{eq:update-1}) can be viewed
as a generalization of the existing span assumption \cite{nesterov2018lectures},
as it contains more information based on the output of the proximal
update in (\ref{eq:update-2}). In particular, (\ref{eq:update-1})
and (\ref{eq:update-2}) recover Proximal GD with the stepsize sequence
$\eta_{t},\forall t\in\left[T\right]$ by setting $\by_{t+1}=\bx_{1}+\bx_{t}-\bx_{1}-\eta_{t}\nabla f(\bx_{t})=\bx_{t}-\eta_{t}\nabla f(\bx_{t})$
and $\gamma_{t}=\eta_{t}$.

Now we are ready to prove the lower bound. As mentioned in Footnote
\ref{fn:lb}, the proof is only a simple variation of the existing
analysis in \cite{bubeck2015convex}.
\begin{thm}
\label{thm:lb}For any given $D_{\star}>0$, $G>0$, $\mu>0$, $T\in\N$
satisfying $T\geq\frac{G^{2}}{\mu^{2}D_{\star}^{2}}-1$, $d\in\N$
satisfying $d\geq T+1$, and $\bx_{1}\in\R^{d}$, there exist a function
$F=f+\psi:\R^{d}\to\R$ where $f$ is convex and $G$-Lipschitz on
$\R^{d}$ and $\psi$ is $\mu$-strongly convex on $\R^{d}$ and a
subgradient oracle $\nabla f$ such that any algorithm in the form
of (\ref{eq:update-1}) and (\ref{eq:update-2}) starting with $\bx_{1}$
has
\[
\min_{t\in\left[T\right]}F(\bx_{t+1})-F(\bx_{\star})\geq\frac{G^{2}}{2\mu(T+1)},
\]
where $\bx_{\star}=\argmin_{\bx\in\R^{d}}F(\bx)$ satisfying $\left\Vert \bx_{\star}-\bx_{1}\right\Vert \leq D_{\star}$.
\end{thm}

\begin{proof}
W.l.og., we assume $\bx_{1}=\bzero$. For a general point $\bx_{1}\in\R^{d}$,
one can change every $\bx$ to $\bx-\bx_{1}$ in the following definition
of $f$ and $\psi$ and then conclude by similar steps.

The construction of the hard instance is essentially the same as \cite{bubeck2015convex}.
Let $f(\bx)\defeq G\max_{j\in\left[T+1\right]}\bx\left[j\right]$
and $\psi(\bx)\defeq\frac{\mu}{2}\left\Vert \bx\right\Vert ^{2}$,
where $\bx\left[j\right]$ is the $j$-th coordinate of $\bx$. Note
that
\[
\partial f(\bx)=G\cdot\mathrm{Conv}\left\{ \be_{j}:j\in\mathrm{argmax}_{i\in\left[T+1\right]}\bx\left[i\right]\right\} ,
\]
where $\be_{j}\in\R^{d}$ denotes the vector that takes $1$ in the
$j$-th coordinate and $0$ in any other place. So $f$ is $G$-Lipschitz
on $\R^{d}$. $\psi$ is $\mu$-strongly convex on $\R^{d}$ from
its definition.

Next, we claim the minimum value $F(\bx_{\star})=-\frac{G^{2}}{2\mu(T+1)}$
where $\bx_{\star}$ satisfies
\[
\bx_{\star}\left[j\right]=\begin{cases}
-\frac{G}{\mu(T+1)} & j\in\left[T+1\right],\\
0 & j\in\left[d\right]\backslash\left[T+1\right].
\end{cases}
\]
We consider two cases:
\begin{itemize}
\item $\max_{j\in\left[T+1\right]}\bx\left[j\right]\geq0$. We observe $\bzero$
falls into this case and note that $F(\bx)\geq\psi(\bx)\geq0=F(\bzero)$.
\item $\max_{j\in\left[T+1\right]}\bx\left[j\right]<0$. We observe $\bx_{\star}$
falls into this case. Now suppose $\max_{j\in\left[T+1\right]}\bx\left[j\right]=-a$
for some $a>0$, we have $\left|\bx\left[j\right]\right|\geq a,\forall j\in\left[T+1\right]$.
Thus, there is
\begin{align*}
F(\bx) & =-Ga+\frac{\mu}{2}\sum_{j=1}^{d}\left|\bx\left[j\right]\right|^{2}=-Ga+\frac{\mu}{2}\sum_{j=1}^{T+1}\left|\bx\left[j\right]\right|^{2}+\frac{\mu}{2}\sum_{j=T+2}^{d}\left|\bx\left[j\right]\right|^{2}\\
 & \geq-Ga+\frac{\mu(T+1)a^{2}}{2}\geq-\frac{G^{2}}{2\mu(T+1)}=F(\bx_{\star}).
\end{align*}
\end{itemize}
We thereby have
\begin{equation}
\min_{\bx\in\R^{d}}F(\bx)=\min\left\{ F(\bzero),F(\bx_{\star})\right\} =F(\bx_{\star})=-\frac{G^{2}}{2\mu(T+1)}.\label{eq:lb-1}
\end{equation}
Moreover, $\left\Vert \bx_{\star}\right\Vert =\frac{G}{\mu\sqrt{T+1}}\leq D_{\star}$
once $T\geq\frac{G^{2}}{\mu^{2}D_{\star}^{2}}-1$.

Because $\bx_{1}=\bzero$ now, we have $\by_{t+1}\in\mathrm{Span}\cup_{s\in\left[t\right]}\left\{ \bx_{s},\by_{s},\nabla f(\bx_{s}),\nabla f(\by_{s})\right\} $.
Moreover, we can explicitly write
\begin{align*}
\bx_{t+1} & =\argmin_{\bx\in\R^{d}}\psi(\bx)+\frac{\left\Vert \bx-\by_{t+1}\right\Vert ^{2}}{2\gamma_{t}}=\argmin_{\bx\in\R^{d}}\frac{\mu\left\Vert \bx\right\Vert ^{2}}{2}+\frac{\left\Vert \bx-\by_{t+1}\right\Vert ^{2}}{2\gamma_{t}}\\
 & =\frac{\by_{t+1}}{1+\mu\gamma_{t}}\in\mathrm{Span}\cup_{s\in\left[t\right]}\left\{ \bx_{s},\by_{s},\nabla f(\bx_{s}),\nabla f(\by_{s})\right\} .
\end{align*}
Now we define the subgradient oracle
\[
\nabla f(\bx)\defeq\be_{j_{\bx}}\text{ where }j_{\bx}\defeq\min\mathrm{argmax}_{i\in\left[T+1\right]}\bx\left[i\right].
\]
By induction, one can show $\mathrm{Span}\cup_{s\in\left[t\right]}\left\{ \bx_{s},\by_{s},\nabla f(\bx_{s}),\nabla f(\by_{s})\right\} \subseteq\mathrm{Span}\left\{ \be_{1},\cdots,\be_{t}\right\} ,\forall t\in\left[T\right]$.
As such, $\bx_{t+1}\in\mathrm{Span}\left\{ \be_{1},\cdots,\be_{t}\right\} ,\forall t\in\left[T\right]$,
which implies $F(\bx_{t+1})\geq f(\bx_{t+1})\geq G\bx_{t+1}\left[t+1\right]=0,\forall t\in\left[T\right]$.
We hence conclude
\[
\min_{t\in\left[T\right]}F(\bx_{t+1})-F(\bx_{\star})\overset{\eqref{eq:lb-1}}{\geq}\frac{G^{2}}{2\mu(T+1)}.
\]
\end{proof}

\end{document}